\documentclass[12pt,oneside]{amsart}

\usepackage[cp1251]{inputenc} 
\usepackage[T2A]{fontenc} 

\usepackage[english]{babel}
\usepackage{amsmath,  amssymb, amsthm}
\usepackage{floatflt}
\usepackage{multicol}
\usepackage{amsfonts}
\usepackage{caption}
\usepackage{xcolor}
\usepackage{enumerate}
\usepackage{comment}
\usepackage{float}

\usepackage{pdfpages} 

\usepackage{graphicx} 

\usepackage{tikz} 

\newcommand{\dist} {{\textrm{dist}}} 

\newtheorem{thm}{Theorem}
\newtheorem{lemma}{Lemma}

\newtheorem{theorem}{Theorem}

\numberwithin{equation}{section}

\DeclareCaptionLabelSeparator{dot}{. }
\title{Simple closed geodesics on regular spherical polyhedra}
\author{Darya Sukhorebska}
\thanks{The author is partially supported by DFG grant TRR 191.\\
\textbf{Acknowledgements.}The author is grateful to Alexander Borisenko for useful remarks, Alexander Lytchak for the discussion, and Jeroen Winkel for valuable comments.  }

 \date{}  %
 
\address{Karlsruher Institut f\"ur Technologie, Englerstr. 2, 76131 Karlsruhe, Germany }
\address{B.Verkin Institute for Low Temperature Physics and Engineering of the National Academy of Sciences of Ukraine, 47 Nauky Ave., Kharkiv, 61103, Ukraine}

\email{darya.sukhorebska@kit.edu}

\begin{document}


\begin{abstract}
     In this article, we found all simple closed geodesics on regular 
     spherical octahedra and spherical cubes. 
     In addition, we estimate the number of simple closed  geodesics on 
     regular spherical tetrahedra. 

\medskip	


\medskip
\noindent
\textnormal{MSC2020:} 52B05, 53C22, 52B10

\end{abstract}

\maketitle


 \section{Introduction}
Describing simple closed geodesics on a convex  regular or not
surface in Euclidean space is a long-standing problem.
The first results, in particular on convex polyhedra, 
were presented by Alexandrov~\cite{Alek50}, Pogorelov~\cite{Pog49}, 
and Toponogov ~\cite{Top63}.
Notice that a geodesic does not go through a vertex of a polyhedron, 
that makes the behavior of geodesics differ from the smooth case.
D. Fuchs and E. Fuchs completed and systematized results on
closed geodesics on regular polyhedra in Euclidean space~\cite{FucFuc07}.
 In particular, they showed that on a regular polyhedron, that is not a tetrahedron, 
 there are finitely many simple closed geodesics. 
On a regular tetrahedron in Euclidean space, 
there are infinitely many simple closed geodesics. 
Akopyan and Petrunin~\cite{AkPet2018} showed that isosceles tetrahedra are the only convex surfaces in Euclidean space that admit arbitrary long simple closed geodesics. 

In this paper, we study simple closed geodesics on regular spherical tetrahedra, octahedra, and cubes. 
In this case, facets of the polyhedra have constant curvature $1$ 
and the intrinsic geometry of the polyhedra depends on the value of the planar angle of a facet. 

If there is a simple closed geodesic $\gamma$ on a polyhedron in Euclidean space, 
then there are infinitely many geodesics equivalent to $\gamma$. 
All of them are parallel to each other in the development of the polyhedron. 
By distinct geodesics on the euclidean polyhedron, we mean geodesics
that are not equivalent and can not be transformed into each other
by symmetries of the polyhedron.
If there is a simple closed geodesic $\gamma$ on 
a spherical polyhedron, then $\gamma$ is unique within its class of equivalents.
By different geodesics on the spherical polyhedron, we mean 
geodesics that can not be mapped into each other by symmetries of the surface.

In the work~\cite{BorSuh2021}, a behavior of simple closed geodesics 
on a regular spherical tetrahedron is described, 
and conditions of the existence of a simple closed geodesic 
in terms of the estimates on the plane angle $\alpha$ of the tetrahedron are found. 
In particular, it was shown that there are finitely many geodesics 
on a spherical tetrahedron (unlike to euclidean case). 
Borisenko, in his work~\cite{Bor2022}, 
presented the necessary and sufficient condition of 
the existence of a simple closed geodesic 
in terms of the length of an abstract shortest curve.

In the current work, we estimate the number of different simple closed geodesics 
on a spherical tetrahedron with a planar angle~$\alpha$.

\newtheorem*{theorem:tetr}{Theorem \ref{geod_tetr}}
\begin{theorem:tetr}
The number $N(\alpha)$ of different simple closed geodesics on a spherical tetrahedron 
 with the planar angle $\alpha \in (\pi/3, 2\pi/3)$
 satisfies
  $$c_1(\alpha)<N(\alpha) < c_2(\alpha),$$ where 
  \begin{equation} 
 c_1(\alpha) \approx \frac{3 \cos^2(\alpha)}{ 8\sin^2(\alpha/2) (4\sin^2(\alpha/2)-1)}. \notag
 \end{equation}
 \begin{equation} 
 c_2(\alpha) \approx \frac{2 \sin^2(\alpha/2)}{ 4\sin^2(\alpha/2) - 1}  +1.
 \notag
 \end{equation}
\end{theorem:tetr}
When $\alpha$ tends to $\pi/3$,
the number of simple closed geodesics on a spherical tetrahedron grows to infinity. 

On a regular euclidean octahedron  
there are only two distinct simple closed geodesics and 
there are only three of them on a cube. 
In the present work, we showed that the same holds for spherical octahedra and spherical cubes. 

\newtheorem*{theorem:oct}{Theorem \ref{geod_oct}}
\begin{theorem:oct} 
There are only two different simple closed geodesics on  regular spherical octahedra.
\end{theorem:oct} 

\newtheorem*{theorem:cube}{Theorem \ref{geod_cube}}
\begin{theorem:cube}  
There are only three different simple closed geodesics on spherical cubes.
\end{theorem:cube} 

In the Euclidean case, the proof follows from a tiling of a plane.
However, there is no tiling of a unite sphere with regular triangles or squares 
for any planar angle.
To prove our result we analyze the domain that 
a simple closed geodesic bounds on a spherical polyhedron.

  \section{Main definitions}

 A \textit{spherical triangle} is a convex polygon
 on a unit sphere bounded by  three shortest lines. 
 A triangle is regular if the lengths of its edges are equal.
 A regular triangle is uniquely determined by the value  $\alpha$ of   inner angles.
The   length of the edge  is equal to
\begin{equation}\label{a_triangle}
a_{t} =\text{arccos} \left(  \frac{\cos\alpha}{1-\cos\alpha}  \right), 
\end{equation}
\begin{equation*} \label{alim_triangle}
\lim\limits_{\alpha\to\pi/3 } a_t = 0; \;\;\; 
\lim\limits_{\alpha\to\pi/2 } a_t = \pi/2.
\end{equation*}

A \textit{regular spherical tetrahedron} $A_1 A_2 A_3 A_4$  
is a closed convex polyhedron 
that consists of four  regular spherical triangles glued by three in each vertex. 
A planar angle $\alpha$ of a  regular spherical tetrahedron 
 satisfies the conditions  $\pi/3 < \alpha \le 2\pi/3$.
If $\alpha=2\pi/3$, then the tetrahedron is the unit  two-dimensional sphere.
There are infinitely many simple closed geodesics on it.
In the following we suppose that 
the angle $\alpha$ satisfies  $\pi/3 < \alpha < 2\pi/3$ on the tetrahedron.

A \textit{regular spherical octahedron} is a closed convex polyhedron
that consist of eight spherical triangles glued by four in each vertex. 
A planar angle $\alpha$ of the octahedron  
 satisfies   $\pi/3 < \alpha \le \pi/2$.
If $\alpha= \pi/2$, then the octahedron is the unit  two-dimensional sphere
with infinitely many simple closed geodesics on it.
In the following we assume that the planar angle $\alpha$ of the octahedron satisfies  
$\pi/3 < \alpha < \pi/2$.

Consider a spherical triangle $ABC$ with the edges $a$ opposite to the angle $A$,
$b$ opposite to the angle $B$ and $c$, opposite to $C$.
Consider two cosine formulas for the sides 
\begin{equation}\label{cos_rule}
    \cos b = \cos a \, \cos c + \sin a \,\sin c\, \cos B
\end{equation}
\begin{equation*}
    \cos c = \cos a \,\cos b + \sin a \,\sin b\, \cos C
\end{equation*}
Combining this two equations we have 
\begin{equation*}
    \cos b = \cos a \left(  \cos a \,\cos b + \sin a\, \sin b\, \cos C \right) 
    + \sin a\, \sin c\, \cos B;
\end{equation*}
\begin{equation*}
    \cos b \, (1-\cos^2a)=  \cos a\, \sin a\, \sin b\, \cos C + \sin a\, \sin c\, \cos B;
\end{equation*}
Dividing both sides on $\sin a$ we obtain
\begin{equation}\label{cosin_rule}
    \cos b\, \sin a=  \cos a\, \sin b\, \cos C +  \sin c\, \cos B;
\end{equation}
Since in any spherical triangle, the sine of the angles
are to each other as the sine of their opposite sides, one can rewrite 
the equation~(\ref{cosin_rule}) as following
\begin{equation}\label{formula}
    \cos b \,\sin A=  \cos a\, \sin B \,\cos C +  \sin C\, \cos B.
\end{equation}

A \textit{spherical square} is a convex polygon on a unit sphere bounded by 
the four shortest lines  of the same length and having four equal angles $\alpha$.  
Denote the vertices of a spherical square as $A_1A_2A_3A_4$ and the length of 
the side of a square as $a_s$. 
Applying the formula~(\ref{formula}) to the triangle $A_1A_2A_4$, we get
$$ \cos a_s\, \sin (\alpha/2) =  
\cos a_s \,\sin (\alpha/2)\, \cos \alpha + 
\sin \alpha \,\cos (\alpha/2). $$
Modifying this equation, we get that the length of the edge is equal to
\begin{equation}\label{a_cube}
 a_s =\text{arccos} \left( \cot^2 (\alpha/2) \right).
\end{equation}
\begin{equation*} \label{alim_cube}
\lim\limits_{\alpha\to\pi/2  } a_s = 0; \;\;\;  
\lim\limits_{\alpha\to 2\pi/3  }a_s=  \text{arccos} (1/3).
\end{equation*}
Using (\ref{cos_rule}) we can also calculate the length $d$
of the diagonal of the cube
\begin{equation}\label{diagonal}
    \cos d = \frac{\cos^4 (\alpha/2) - \cos^2 \alpha}
                    {\sin^4  (\alpha/2)}.
\end{equation}
\begin{equation*} \label{dlim_cube}
\lim\limits_{\alpha\to\pi/2  } d = 0; \;\;\;  
\lim\limits_{\alpha\to 2\pi/3  } d=  \text{arccos} (-1/3).
\end{equation*}

A \textit{spherical cube} is a closed convex polyhedron
that consist of six spherical squares glued by three in each vertex. 
A planar angle $\alpha$ of a  cube
 satisfies the conditions  $\pi/2 < \alpha \le  2\pi/3$.
If $\alpha= 2\pi/3$, then the cube is a unit  two-dimensional sphere
and has infinitely many simple closed geodesics on it.
In the following we assume that on the cube the planar angle  $\alpha$ satisfies  
$\pi/2 < \alpha <  2\pi/3$.

For all regular spherical polyhedra the inner geometry of the surface depends on the planar angle $\alpha$.

 A \textit{geodesic}  is a locally shortest curve $\gamma: [0,1] \rightarrow M$.  
 The geodesic is  \textit{closed} if $\gamma(0)=\gamma(1)$ and $\gamma'(0)=\gamma'(1)$. 
 The geodesic is called \textit{simple} if it has no points of self intersection,
 i.e. the map $\gamma: (0,1)  \rightarrow M$ is injective.

 A geodesic has the following properties on a convex polyhedron: 
1) it consists of line segments on facets of the polyhedron; 
2) it forms equal angles with edges of adjacent facets; 
3) a geodesic cannot pass through a vertex of a convex polyhedron~\cite{Alek50}.

 In what follows the words \textit{tetrahedron}, \textit{octahedron} or \textit{cube}
 refer  to a regular spherical tetrahedron, octahedron, or spherical cube 
 respectively unless the opposite is specified. 
 The word  \textit{geodesic} refers to a simple closed geodesic.

In Euclidean space, one can unfold the facets of a polyhedron onto 
the plane in the order in which 
these facets are intersected by geodesic. 
The resulting polygon on a plane is called \textbf{a development} of the polyhedron.
The geodesic is unfolded into a straight line segment inside the development. 
For a spherical polyhedron, one can perform the same unfolding onto a
unite two-dimensional sphere. 
However, in this case, we should consider the local embedding 
of the development into the sphere. 
The global embedding needs to be checked separately. 
The geodesic is locally unfolded into an arc of a great circle of a sphere.

 \section{Properties}

  Borisenko \cite{Bor2020} proved the  generalization of  Toponogov theorem \cite{Top63} 
to the case of two-dimensional Alexandrov space.

\begin{thm}[Borisenko \cite{Bor2020}]\label{Bor_length}
    Let $G$ be a domain homeomorphic to a disc and $G$ lies in a two-dimansional 
    Alexandrov space of curvature $\ge c$ (in sense of Alexandrov).
    If the boundary curve $\gamma$ of $G$ is $\lambda$-convex, $\lambda>0$
    and $c+\lambda^2>0$,
    then the length $s(\gamma)$ of $\gamma$ satisfies 
    $$s(\gamma) \leq \frac{2\pi \sqrt{|c|}}{\sqrt{c+\lambda^2}}.$$
    The equality holds if and only if the domain $G$ is a disc on the plane of constant 
    curvature $c$.
\end{thm}

A curve $\gamma$ is called $\lambda$-convex with $\lambda>0$ if any subarc $\gamma_0$ of $\gamma$
satisfies $\tau(\gamma_0)/s(\gamma_0)\ge0>0$,
where $\tau(\gamma_0)$ is the integral geodesic curvature (the swerve) of subarc $\gamma_0$
and $s(\gamma_0)$ is the length of $\gamma_0$.
 
A convex spherical polyhedron with facets  of curvature $1$ is a two-dimensional
Alexandrov space of curvature $\ge 1$ in sense of Alexandrov. 
A simple closed geodesic $\gamma$ bounds a domain, homeomorphic to a disc and $\lambda=0$.
In this case we have the following corollary from the Theorem~\ref{Bor_length},
that we will use.

\begin{lemma}[The length of a geodesic]\label{length}
The length of a simple closed geodesic on a convex spherical polyhedron is   $< 2\pi$.
\end{lemma}
 
We call two geodesics \textbf{equivalent} if they intersect edges of a polyhedron
in the same order. 
In other words it means that  developments of a polyhedron along these geodesics 
are equal polygons with equal labelings. 

\begin{lemma}[Uniqueness of a geodesic]\label{uniqness}
   Assume there are two geodesics $\gamma_1$ and $\gamma_2$ on a spherical polyhedron.
   If $\gamma_1$ and $\gamma_2$ are equivalent then they  coincide. 
\end{lemma}
\begin{proof}
 First, assume  $\gamma_1$ and $\gamma_2$ do not intersect. 
 Then they enclose an annulus $\Omega$ 
 that does not have any vertex of the polyhedron inside. 
 Thus $\Omega$ is locally isometric to a unite sphere. 
 From the Gauss-Bonnet theorem follows
 that the integral of the curvature over the area of  $\Omega$ equals zero.
Thus  $\gamma_1$ and $\gamma_2$ coincide.

 If $\gamma_1$ and $\gamma_2$ intersect then 
 they have at least two points of intersection $Z_k$, $k=1,2,\dots$. 
 Between $Z_k$ and $Z_{k+1}$ geodesics bound a region $\Omega_k$ 
 that does not have a vertex of the polygon. 
Moreover, $\Omega_k$ is isometric to a lune on a unite sphere, $k=1,2, \dots$ 
(see Figure~\ref{uniqness_pic}.).
Thus the length of a segment of $\gamma_i$ that 
belongs to the boundary of $\Omega_k$ equals $\pi$, $i=1,2$.
Hence, the length of $\gamma_i$ is greater or equal to $2\pi$. 
This contradicts Lemma~\ref{length}.   

\begin{figure}[h]
\centering
\begin{tikzpicture}[line cap=round,line join=round,x=0.7cm,y=0.7cm]
\clip(-1,-1) rectangle (10,7);

\draw [shift={(5,2)},line width=0.5pt]  plot[domain=2.677945044588987:3.6052402625905993,variable=\t]({1*4.47213595499958*cos(\t r)+0*4.47213595499958*sin(\t r)},{0*4.47213595499958*cos(\t r)+1*4.47213595499958*sin(\t r)});
\draw [shift={(-2,2)},line width=0.5pt]  plot[domain=-0.27829965900511144:0.27829965900511133,variable=\t]({1*7.280109889280519*cos(\t r)+0*7.280109889280519*sin(\t r)},{0*7.280109889280519*cos(\t r)+1*7.280109889280519*sin(\t r)});
\draw [shift={(3,-4)},line width=0.5pt]  plot[domain=1.3258176636680326:1.8157749899217608,variable=\t]({1*8.246211251235321*cos(\t r)+0*8.246211251235321*sin(\t r)},{0*8.246211251235321*cos(\t r)+1*8.246211251235321*sin(\t r)});
\draw [shift={(3,8)},line width=0.5pt]  plot[domain=4.4674103172578254:4.957367643511554,variable=\t]({1*8.246211251235321*cos(\t r)+0*8.246211251235321*sin(\t r)},{0*8.246211251235321*cos(\t r)+1*8.246211251235321*sin(\t r)});
\draw [shift={(4,3)},line width=0.5pt]  plot[domain=1.892546881191539:2.819842099193151,variable=\t]({1*3.1622776601683795*cos(\t r)+0*3.1622776601683795*sin(\t r)},{0*3.1622776601683795*cos(\t r)+1*3.1622776601683795*sin(\t r)});
\draw [shift={(9.26,1.83)},line width=0.5pt]  plot[domain=2.0674508120046835:2.670461554730127,variable=\t]({1*4.743047543510396*cos(\t r)+0*4.743047543510396*sin(\t r)},{0*4.743047543510396*cos(\t r)+1*4.743047543510396*sin(\t r)});
\draw [shift={(3,4)},line width=0.5pt]  plot[domain=5.176036589385496:5.81953769817878,variable=\t]({1*4.47213595499958*cos(\t r)+0*4.47213595499958*sin(\t r)},{0*4.47213595499958*cos(\t r)+1*4.47213595499958*sin(\t r)});
\draw [shift={(3,4)},line width=0.5pt]  plot[domain=-0.46364760900080615:0.4636476090008061,variable=\t]({1*4.47213595499958*cos(\t r)+0*4.47213595499958*sin(\t r)},{0*4.47213595499958*cos(\t r)+1*4.47213595499958*sin(\t r)});
\draw [shift={(4.52,-1.39)},line width=0.5pt,dash pattern=on 1pt off 1pt]  plot[domain=1.9923002292021055:2.7655023419803824,variable=\t]({1*3.715171597651984*cos(\t r)+0*3.715171597651984*sin(\t r)},{0*3.715171597651984*cos(\t r)+1*3.715171597651984*sin(\t r)});
\draw [shift={(5,-4)},line width=0.5pt,dash pattern=on 1pt off 1pt]  plot[domain=1.2490457723982544:1.892546881191539,variable=\t]({1*6.324555320336759*cos(\t r)+0*6.324555320336759*sin(\t r)},{0*6.324555320336759*cos(\t r)+1*6.324555320336759*sin(\t r)});
\draw [shift={(8,4)},line width=0.5pt,dash pattern=on 1pt off 1pt]  plot[domain=2.761086276477428:3.522099030702158,variable=\t]({1*5.385164807134505*cos(\t r)+0*5.385164807134505*sin(\t r)},{0*5.385164807134505*cos(\t r)+1*5.385164807134505*sin(\t r)});
\draw [shift={(4.98,0.87)},line width=0.5pt]  plot[domain=1.1956789173869815:1.9391452482038598,variable=\t]({1*5.513374647164838*cos(\t r)+0*5.513374647164838*sin(\t r)},{0*5.513374647164838*cos(\t r)+1*5.513374647164838*sin(\t r)});

\draw [shift={(1.9122040192957854,9.15116621670033)},line width=0.5pt]  plot[domain=4.554596187407736:5.133379376634906,variable=\t]({1*8.223328519933942*cos(\t r)+0*8.223328519933942*sin(\t r)},{0*8.223328519933942*cos(\t r)+1*8.223328519933942*sin(\t r)});
\draw [shift={(1.48,7.09)},line width=0.5pt]  plot[domain=5.320959227028316:5.837122023141484,variable=\t]({1*6.632239370605173*cos(\t r)+0*6.632239370605173*sin(\t r)},{0*6.632239370605173*cos(\t r)+1*6.632239370605173*sin(\t r)});
\draw [shift={(5.66,-2.17)},line width=0.5pt,dash pattern=on 1pt off 1pt]  plot[domain=1.2955970507433245:2.0450509569284785,variable=\t]({1*6.647405169472385*cos(\t r)+0*6.647405169472385*sin(\t r)},{0*6.647405169472385*cos(\t r)+1*6.647405169472385*sin(\t r)});
\draw [shift={(6.02,-0.75)},line width=0.5pt,dash pattern=on 1pt off 1pt]  plot[domain=2.21765714690538:2.8231791279631517,variable=\t]({1*5.640436445632524*cos(\t r)+0*5.640436445632524*sin(\t r)},{0*5.640436445632524*cos(\t r)+1*5.640436445632524*sin(\t r)});

\draw [shift={(5.7004381342824075,8.156403736725176)},line width=0.5pt]  plot[domain=3.948141888997408:4.641794766491713,variable=\t]({1*7.358720397387098*cos(\t r)+0*7.358720397387098*sin(\t r)},{0*7.358720397387098*cos(\t r)+1*7.358720397387098*sin(\t r)});
\draw [shift={(0.599732294452347,7.216800029388058)},line width=0.5pt]  plot[domain=5.332107939336842:5.757650938653997,variable=\t]({1*7.886139944999266*cos(\t r)+0*7.886139944999266*sin(\t r)},{0*7.886139944999266*cos(\t r)+1*7.886139944999266*sin(\t r)});
\draw [shift={(2.4118251586025004,-3.9242153576091887)},line width=0.5pt,dotted]  plot[domain=0.9632222797501806:1.5403598247227832,variable=\t]({1*8.757976042711176*cos(\t r)+0*8.757976042711176*sin(\t r)},{0*8.757976042711176*cos(\t r)+1*8.757976042711176*sin(\t r)});
\draw [shift={(5.7004381342824075,-0.36714417983295927)},line width=0.5pt,dotted]  plot[domain=2.0980769966976154:2.5789585620821067,variable=\t]({1*6.006306432004119*cos(\t r)+0*6.006306432004119*sin(\t r)},{0*6.006306432004119*cos(\t r)+1*6.006306432004119*sin(\t r)});

\draw (5.488880781142076,2.0540484148691713) node[anchor=north west] {$\Omega_1$};
\draw (1.0682381995767751,3.5069868857332933) node[anchor=north west] {$\Omega_2$};
\begin{scriptsize}
\draw [fill=black] (1,0) circle (1.5pt);
\draw [fill=black] (5,0) circle (1.5pt);
\draw [fill=black] (1,4) circle (1.5pt);
\draw [fill=black] (5,4) circle (1.5pt);
\draw [fill=black] (3,6) circle (1.5pt);
\draw [fill=black] (7,6) circle (1.5pt);
\draw [fill=black] (7,2) circle (1.5pt);
\draw [fill=black] (3,2) circle (1.5pt);
\draw [fill=black] (0.62,1.03) circle (1.5pt);
\draw [fill=black] (5.271617957443751,1.6484715075821397) circle (1.5pt);
\draw [fill=black] (7.466357371119691,4.227270397202993) circle (1.5pt);
\draw [fill=black] (2.6205970707936324,3.7509535680923025) circle (1.5pt);
\draw [fill=black] (0.6082502196012998,2.8441172112733906) circle (1.5pt);
\draw [fill=black] (5.180052880027632,0.7971531103224692) circle (1.5pt);
\draw [fill=black] (7.4115553026060255,3.266392603602813) circle (1.5pt);
\draw [fill=black] (2.6781535246328065,4.823377248059349) circle (1.5pt);
\draw [fill=black] (3.6018006541288377,1.1032848086327531) circle (1.5pt);
\draw[color=black] (3.6,1.4) node {$Z_1$};
\draw [fill=black] (4.995196221550772,4.444078274656856) circle (1.5pt);
\draw[color=black] (5,4.8) node {$Z_2$};
\end{scriptsize}
\end{tikzpicture}
\caption{\;}
\label{uniqness_pic}
\end{figure}

\end{proof}

\section{Tetrahedron}

 Any simple closed geodesic $\gamma$ on a tetrahedron  
 has $p$~vertices on each of two opposite edges of the tetrahedron,
 $q$~vertices on each of other two opposite edges, 
 and $(p + q)$~vertices on each of the remaining two opposite edges.
The integers $(p,q)$ are coprime and satisfy $0\leq p<q$.
The pair $(p,q)$ is called the \textbf{type $(p, q)$} of geodesic and
it uniquely defines a geodesic  (see~\cite{Protasov07}, \cite{FucFuc07}). 
Notice, that it includes pairs $(0,1)$ and $(1,1)$.

The necessary condition on the existence of a geodesic of type $(p,q)$
was proved in \cite{BorSuh2021}.  Namely 
\begin{thm}[Borisenko, Sukhorebska \cite{BorSuh2021}]\label{thm_tetr_necess}
    If the planar angle $\alpha$ of a regular spherical tetrahedron 
    satisfies 
    \begin{equation}\label{necess}
        \alpha > 2\sin^{-1}\sqrt{ \frac{p^2+pq+q^2}{4(p^2+pq+q^2)-\pi^2}}
    \end{equation}
    then there is no simple closed geodesic of type $(p,q)$ on the tetrahedron.
\end{thm}

In  \cite{BorSuh2021} it was also found a small $\varepsilon>0$
depending on $(p,q)$ such that
on a tetrahedron with a planar angle $\alpha<\pi/3+\varepsilon$ there 
exists a simple closed geodesic of type $(p,q)$. 
Using this, Borisenko in his work~\cite{Bor2022} 
proved necessary and sufficient condition of
existence a  geodesic on a spherical tetrahedron.  
It was shown that the geodesic on the tetrahedron exists if and only if
the length of an abstract shortest curve is less than $2\pi$.
 From this follows a sufficient condition on the length of the edge of the tetrahedron
which contains a geodesic of type $(p,q)$.

 \begin{thm}[Borisenko \cite{Bor2022}]\label{thm_tetr_suff}
     If the edge $a_t$ of a regular spherical tetrahedron satisfies the inequality
     \begin{equation}\label{suff}
        a_t<2\sin^{-1} \frac{\pi}{\sqrt{p^2+pq+q^2} + \sqrt{(p^2+pq+q^2)+2\pi^2}}, 
     \end{equation}
      then this tetrahedron has a simple closed geodesic of type $(p,q)$. 
 \end{thm}

Let $N(\alpha)$ be a number of geodesics on a tetrahedron 
with a planar angle~$\alpha$.
From the Theorem~\ref{thm_tetr_necess} it follows $N(\alpha)$ is finite.

\begin{theorem}\label{geod_tetr}
The number $N(\alpha)$ of different simple closed geodesics on a spherical tetrahedron 
 with the planar angle $\alpha \in (\pi/3, 2\pi/3)$
 satisfies $$c_1(\alpha)<N(\alpha) < c_2(\alpha),$$ where 
  \begin{equation}\label{const_below}
 c_1(\alpha) \approx \frac{3 \cos^2(\alpha)}{ 4\sin^2(\alpha/2) (4\sin^2(\alpha/2)-1)}.
 \end{equation}
 \begin{equation}\label{const_above}
 c_2(\alpha) \approx \frac{2 \sin^2(\alpha/2)}{ 4\sin^2(\alpha/2) - 1}  +1.
 \end{equation}
\end{theorem}

The functions $c_1(\alpha)$ and $c_2(\alpha)$ is monotonically decrease 
on the interval $(\pi/3, 2\pi/3)$.
Moreover
\begin{equation*}
c_1\left(\frac{\pi}{2}\right) \approx 0; \;\;\;
 c_2\left(\frac{\pi}{2}\right) \approx 2. 
\end{equation*}
\begin{equation*}
  c_1\left(\frac{2\pi}{3}\right) \approx \frac{1}{16}; \;\;\;
  c_2\left(\frac{2\pi}{3}\right) \approx  \frac{7}{4}.
\end{equation*}

Since $0< N(\alpha)< 2$ when  $\alpha  \in (\pi/2, 2\pi/3)$, it follows that 
there is only one simple closed geodesic  on a tetrahedron with such
planar angle.  
Same result was proved in~\cite[Lemma~4]{BorSuh2021}. 
It was also shown that this geodesic has type $(0,1)$.

If $\alpha \rightarrow \pi/3$, then $c_1(\alpha) \rightarrow  +\infty$ and 
$c_2(\alpha) \rightarrow  +\infty$.  
It means that when $\alpha $ tends to $\pi/3$, 
the number of geodesics tends to infinity.

 Note, that for each ordered pair of coprime integers $(p,q)$, $0\le p \le q$ there exists 
three simple closed geodesics of type $(p,q)$. 
They can be mapped into each other with an isometry of the tetrahedron.

\begin{proof}[Proof of the Theorem~\ref{geod_tetr}]
Fix the angle  $\alpha\in (\pi/3, 2\pi/3)$ and 
consider a  geodesic  $\gamma_{p,q}$  of type $(p,q)$ on a tetrahedron 
with the planar angle $\alpha$.
From Theorem~\ref{thm_tetr_suff} follows that the geodesic of type $(p,q)$
exists if 
\begin{equation}\label{p_q_a_above} 
p^2+pq+q^2 < \frac{\pi^2}{4} \frac{\left( 2\sin^2(a_t/2)-1 \right)^2}{\sin^2(a_t/2)}.
\end{equation}
Using (\ref{a_triangle}) we can calculate 
$$ \sin (a_t/2) = \frac{4\sin^2(\alpha/2)-1}{4\sin^2(\alpha/2)}.$$
Rewrite (\ref{p_q_a_above}) as follows
\begin{equation}\label{p_q_above} 
 p^2+pq+q^2 <  \frac{\pi^2 \cos^2 \alpha}{4\sin^2(\alpha/2) 
 \left( 4\sin^2(\alpha/2)-1\right)} =: f(\alpha).  
\end{equation}
  Let   $\psi_1(\alpha)$ be the number of pairs $(p,q)$ 
such that $p,q$ are coprime,  $0<p\le q$ and 
 $$p^2+pq+q^2 <  f(\alpha).$$
Hence 
\begin{equation}\label{1_estim}
   \psi_1(\alpha) <  N(\alpha). 
\end{equation} 
 Let $c_1(\alpha)$ be the number of pairs $(p,q)$ such that $p,q$ are coprime,
$0<p\le q$ and $$ p+q<\sqrt{ f(\alpha) }.$$ 
Consider $(p,q)$ as a coordinates of the Euclidean plane. 
  The curve $p^2+pq+q^2 =  f(\alpha)$ is an ellipse 
   with the focal points  $\left(\mp\sqrt{2f(\alpha)/3}, \pm\sqrt{2f(\alpha)/3} \right)$.
 The ellipse intersects axes $p=0$, $q=0$ 
at the points $(\pm\sqrt{f(\alpha)}, 0)$, $(0, \pm\sqrt{f(\alpha)})$.
  Since the domain $p+q\leq \sqrt{ f(\alpha)}$ and $p,q, \geq 0$ lies inside ellipse,
 it follows  
\begin{equation}\label{psi_1_psi_1}
   c_1(\alpha)<\psi_1(\alpha) <  N(\alpha). 
\end{equation}
Euler's  function $\varphi(n)$ is equal to the number of
positive integers not greater than  $n $ and prime to $n \in \mathbb N$.
From \cite[Th. 330.]{Hard} it's known
\begin{equation}\label{phi}
 \sum\limits_{n=1}^{x}  \varphi(n) = \frac{3}{\pi^2}x^2+O(x\ln x) \approx   \frac{3}{\pi^2}x^2.
\end{equation}
The error term $O(x\ln x) < C x\ln x $, when $x \rightarrow +\infty$.
In \cite{Mon} there is an improvment on the error term. 

If $(p, q)=1$ and $p+q=n$, then $(p, n)=1$ and $(q, n)=1$.
Thus the set of  integers not greater than and prime to $n$ are separated
 into the  pairs of coprime integers  
$(p,q)$ such that $p<q$ and $p+q=n$.
 It follows that  $\varphi(n)$ is even  and   
\begin{equation}\label{psi_1}
c_1(\alpha) = \frac{1}{2} \sum\limits_{n=1}^{\sqrt{ f(\alpha)}} \varphi(n)  \approx \frac{3}{2\pi^2}f(\alpha).\notag
\end{equation}
Combining  (\ref{psi_1_psi_1}) and (\ref{phi})
we have 
$$c_1 (\alpha) \approx \frac{3 \cos^2(\alpha))}{ 8\sin^2(\alpha/2) (4\sin^2(\alpha/2)-1)} < \psi_1(\alpha)  < N(\alpha). $$

To find estimation on $N(\alpha)$ from above
we use the estimation from the Theorem~\ref{thm_tetr_necess}. 
If geodesic $\gamma_{p,q}$  on a tetrahedron exists,
 the pair of coprime integers $(p,q)$ satisfies
 \begin{equation} 
 p^2+pq+q^2 <  \frac{\pi^2 \sin^2(\alpha/2)}{ 4\sin^2(\alpha/2) - 1} =: g(\alpha).  
\end{equation}
Let   $\psi_2(\alpha)$ be the number of pairs $(p,q)$ 
such that $p,q$ are coprime,  $0<p\le q$ and 
 $$p^2+pq+q^2 <  g(\alpha).$$
Notice that $\psi(\alpha)$ does not include the pair $(0,1)$.
Hence 
\begin{equation}\label{2_estim}
    N(\alpha) < \psi_2(\alpha)+1. 
\end{equation} 
The tangent line  to ellipse 
at the point $\left( \sqrt{ g(\alpha)/3},  \sqrt{ g(\alpha)/3} \right)$ satisfies
$$ p+q=2\sqrt{\frac{ g(\alpha)}{3} }.$$
 Let $c_2(\alpha)$ be the number of pairs $(p,q)$ such that $p,q$ are coprime,
$0<p\le q$ and $$ p+q<2\sqrt{ g(\alpha)/3 }.$$
Since the ellipse lies inside  the half space $p+q\leq 2\sqrt{ g(\alpha)/3 }, $ 
it follows  
\begin{equation}\label{psi_2_psi_2}
   \psi_2(\alpha) <c_2(\alpha).
\end{equation}
Using (\ref{phi}) we have
\begin{equation}\label{psi_2}
c_2(\alpha) = \frac{1}{2} \sum\limits_{n=1}^{2\sqrt{ g(\alpha)/3 } } \varphi(n) \approx
\frac{2}{\pi^2}g(\alpha) .\notag
\end{equation}
And from (\ref{2_estim}) and (\ref{psi_2_psi_2}) follows
$$ N(\alpha) <  c_2(\alpha)+1 \approx \frac{2 \sin^2(\alpha/2)}{ 4\sin^2(\alpha/2) - 1}  +1. $$
 \end{proof}


 \section{Octahedron}

Label vertices of an octahedron with $A_1, \dots, A_6$, 
where $A_1A_2A_3A_4$ is a square of symmetry, $A_5$
 is a top vertex and $A_6$ is a bottom vertex. 
For any $\alpha \in (\pi/3, \pi/2)$ 
a geodesic $\gamma$ divides the surface of the octahedron with a planar angle $\alpha$
into two closed domains $D_1$ and $D_2$. 
Since $\gamma$ doesn't go through a vertex of the octahedron, 
each vertex belongs to only one of the domains. 

\begin{lemma}\label{edge_inside_oct}
If a domain $D_i$, $i=1,2$ contains a vertex of the octahedron, then 
$D_i$ contains at least one  edge, coming out of this vertex.
\end{lemma}

\begin{proof}
Since the planar angle of the octahedron $\alpha > \pi/3$, the geodesic 
$\gamma$ can not  sequentially intersect more than three edges coming from one vertex.  
Therefore, the domain $D_i$ contains at least two vertices of the octahedron.

Without loss of generality we can assume the domain $D_1$ has at most three 
vertices of the octahedron. If not, we can consider $D_2$. 
Suppose $D_1$ contains a vertex $A_5$ inside, and does not have 
any edge coming out of it.
Then $\gamma$ intersect all edges $A_5A_i$, $i=1, \dots, 4$ at least once.
Let $X_i$ be the closest point of $\gamma$ to $A_5$ on the edge $A_5A_i$, $i=1, \dots, 4$. 

First, assume there is no $X_iX_{i+1}$ segments of $\gamma$
inside the facets $A_5A_iA_{i+1}$,
$i=1,\dots, 4$, if $i+1 >4$ then take $i+1 \mod 5$.
Consider the facet $A_5A_1A_2$. 
The geodesic $\gamma$ goes from the  point $X_1$  to a point $Y_1$ and
from the point  $X_2$ to a point $Y_2$.
If   $Y_1$ is on the edge $A_5A_2$  
then $\dist(A_5, Y_1)> \dist(A_5X_2)$  
since $X_2$ is the closest point of $\gamma$ to $A_5$ on $A_5A_2$.
 But then  $X_2Y_2$  intersects $X_1Y_1$, that leads to a contradiction. 
Thus  $Y_1$ and  $Y_2$  lie  on the edge $A_1A_2$ and  
\begin{equation}\label{A1Y1Y2}
    \dist(A_1, Y_1) < \dist (A_1, Y_2).
\end{equation}
Similarly, on the facet $A_5A_2A_3$ the segment of $\gamma$ coming out of $X_2$
goes to a point $Q_2$ on the edge $A_2A_3$
and on the facet $A_5A_4A_1$ the segment of $\gamma$ coming out of $X_1$
goes to a point $Q_1$ on the edge $A_1A_4$.
Consider now the facet $A_1A_2A_6$. 
The segment of $\gamma$ coming out of $Y_1$  goes to a point $Z_1$
and $Z_1$ should be on the edge $A_6A_2$. 
Otherwise $\gamma$  crosses all  edges  coming out of the vertex $A_1$
at the points $Q_1$, $X_1$, $Y_1$, $Z_1$ sequentially 
that leads to a contradiction. 
By the same argument the segment of $\gamma$ goes from $Y_2$ to the point $Z_2$ at 
the edge $A_1A_6$. 
From (\ref{A1Y1Y2}) follows $Y_1Z_1$ and $Y_2Z_2$ intersect that leads to a contradiction. 

Hence we can assume there is a segment $X_1X_2$ of $\gamma$ inside $A_5A_1A_2$. 
Assume there are no geodesic segments $X_2X_3$, $X_3X_4$ or $X_4X_1$ within their facets.
Similarly to above once can show starting at $X_4$ $\gamma$ crosses the edge $A_3A_4$ 
and then the edge $A_3A_6$, and starting at $X_3$ $\gamma$ goes to $A_3A_4$ and $A_4A_6$. 
This leads to a contradiction  (see Figure~\ref{around_A_5_oct_1}).


\begin{figure}[h]
\centering
\begin{tikzpicture}[line cap=round,line join=round,x=0.6cm,y=0.6cm]

\clip(-2,-2) rectangle (15,13);

\draw [line width=0.7pt] (3,8)-- (6,5);
\draw [line width=0.7pt] (6,5)-- (9,8);
\draw [line width=0.7pt] (6,5)-- (9,2);
\draw [line width=0.7pt] (6,5)-- (3,2);
\draw [line width=0.7pt] (3,8)-- (9,8);
\draw [line width=0.7pt] (9,8)-- (9,2);
\draw [line width=0.7pt] (9,2)-- (3,2);
\draw [line width=0.7pt] (3,2)-- (3,8);
\draw [line width=0.7pt] (7.945,6.945)-- (9,5.89);
\draw [line width=0.7pt] (7.975,3.025)-- (9,4.17);
\draw [line width=0.7pt] (9,8)-- (12,5);
\draw [line width=0.7pt] (12,5)-- (9,2);
\draw [line width=0.7pt] (7,8)-- (7.945,6.945);
\draw [line width=0.7pt] (7.975,3.025)-- (7,2);
\draw [line width=0.7pt] (4.035,3.035)-- (5.24,2);
\draw [line width=0.7pt] (4.035,3.035)-- (3,4);
\draw [line width=0.7pt] (4.035,6.965)-- (3,5.69);
\draw [line width=0.7pt] (4.035,6.965)-- (5,8);
\draw [line width=0.7pt] (9,5.89)-- (10.855,3.855);
\draw [line width=0.7pt] (9,4.17)-- (10.675,6.325);
\draw [line width=0.7pt] (3,8)-- (6,11);
\draw [line width=0.7pt] (6,11)-- (9,8);
\draw [line width=0.7pt] (9,2)-- (6,-1);
\draw [line width=0.7pt] (6,-1)-- (3,2);
\draw [line width=0.7pt] (3,2)-- (0,5);
\draw [line width=0.7pt] (0,5)-- (3,8);
\draw [line width=0.7pt] (3,5.69)-- (1.125,3.875);
\draw [line width=0.7pt] (3,4)-- (0.885,5.885);

\begin{scriptsize}
\draw [fill=black] (3,8) circle (1pt);
\draw[color=black] (2.5,8.4) node {$A_4$};
\draw [fill=black] (3,2) circle (1pt);
\draw[color=black] (2.62,1.92) node {$A_3$};
\draw [fill=black] (9,2) circle (1pt);
\draw[color=black] (9.2,1.86) node {$A_2$};
\draw [fill=black] (9,8) circle (1pt);
\draw[color=black] (9.22,8.4) node {$A_1$};
\draw [fill=black] (6,5) circle (1pt);
\draw[color=black] (6.48,5.1) node {$A_5$};

\draw [fill=black] (7.945,6.945) circle (1pt);
\draw[color=black] (7.38,7.0) node {$X_1$};
\draw [fill=black] (7.975,3.025) circle (1pt);
\draw[color=black] (8.02,3.72) node {$X_2$};
\draw [fill=black] (4.035,3.035) circle (1pt);
\draw[color=black] (3.87,3.65) node {$X_3$};
\draw [fill=black] (4.035,6.965) circle (1pt);
\draw[color=black] (4.55,7.16) node {$X_4$};

\draw [fill=black] (9,5.89) circle (1pt);
\draw[color=black] (8.58,5.92) node {$Y_1$};
\draw [fill=black] (9,4.17) circle (1pt);
\draw[color=black] (8.58,4.54) node {$Y_2$};
\draw [fill=black] (12,5) circle (1pt);

\draw[color=black] (12.22,5.4) node {$A_6$};
\draw [fill=black] (7,8) circle (1pt);
\draw[color=black] (6.96,8.52) node {$Q_1$};
\draw [fill=black] (7,2) circle (1pt);
\draw[color=black] (6.76,1.75) node {$Q_{2}$};

\draw [fill=black] (5.24,2) circle (1pt);
\draw [fill=black] (3,4) circle (1pt);
\draw [fill=black] (3,5.69) circle (1pt);
\draw [fill=black] (5,8) circle (1pt);
\draw [fill=black] (10.855,3.855) circle (1pt);

\draw[color=black] (11.18,3.68) node {$Z_1$};
\draw [fill=black] (10.675,6.325) circle (1pt);
\draw[color=black] (10.9,6.74) node {$Z_2$};
\draw [fill=black] (6,11) circle (1pt);
\draw[color=black] (6.3,11.54) node {$A_6$};
\draw [fill=black] (0,5) circle (1pt);
\draw[color=black] (-0.34,5.58) node {$A_6$};
\draw [fill=black] (6,-1) circle (1pt);
\draw[color=black] (6.58,-1.04) node {$A_6$};
\draw [fill=black] (1.125,3.875) circle (1pt);
\draw [fill=black] (0.885,5.885) circle (1pt);
\end{scriptsize}
\end{tikzpicture}
\caption{\;}
\label{around_A_5_oct_1}
\end{figure}

Therefore there exist one of the segments  $X_2X_3$, $X_3X_4$ or $X_4X_1$. 
If there are two of them, then $\gamma$ crosses all four edges 
around $A_5$, that is contradiction. 

\textit{Case 1.} Assume  that the points $X_3$ and $X_4$ are connected. 
Then on the facet $A_5A_2A_3$ the geodesic $\gamma$ goes 
from the point $X_2$ to a point $Y_2$ at the edge $A_2A_3$
and $\gamma$ goes from the point $Y_3$ to a point $Y_3$ on $A_2A_3$
and $\dist(A_2, Y_2) < \dist(A_2,Y_3)$. 
Note that there is no points of $\gamma$ between $Y_2$ and $Y_3$. 
If there is a point $Q$ of $\gamma$ between $Y_2$ and $Y_3$ on the edge $A_2A_3$,
then on the facet $A_2A_3A_5$ the geodesic should have a segment $QP$
coming out of $Q$. 
Since $X_2$ and $X_3$ are the closest point of $\gamma$ to $A_5$ on there edges, 
and since $QP$ can not intersect $X_2Y_2$ or $X_3Y_3$, we get a contradiction.

Denote by $\gamma_1$ the part of the geodesic, that starts at $Y_2$
on the facet $A_6A_2A_3$, continues going on the octahedron surface 
and comes to the point $Y_3$. 
Denote by $\sigma_1$ the segment $Y_2Y_3$ of the edge $A_2A_3$.
Then $\gamma_1$ and $\sigma_1$ 
bound  the domain $G_1$ on the octahedron.

Similarly on the facet $A_5A_4A_1$ the geodesic comes from the point 
$X_1$ to the point $Y_1$ on the edge $A_1A_4$ and 
from the point $X_4$ to the point $Y_4$ also on $A_1A_4$
and $\dist(A_4, Y_4) < \dist(A_4,Y_1)$. 
As before one can show that there is no other points of the geodesic 
between $Y_1$ and $Y_4$. 
Denote by $\gamma_2$ the part of the geodesic  that starts at $Y_1$ on $A_1A_4$
in opposite direction to $X_1$, goes along the surface and comes to $Y_4$.
Denote by $\sigma_2$ the segment $Y_1Y_4$ on the edge $A_1A_4$. 
The segments $\gamma_2$ and $\sigma_2$ bound  the domain $G_2$ on the octahedron 
(see Figure~\ref{around_A_5_oct_2}).

\textit{Case 2.} Consider now the case, when there one of the segments
$X_1X_4$ or $X_2X_3$ on the facet 
$A_5A_4A_1$ or $A_5A_3A_2$ respectively and there is no $X_3X_4$ on $A_5A_3A_4$.  
In this case $\gamma$ crosses three edges coming out of $A_5$ sequentially. 
Assume there is $X_1X_4$. 

Similarly to the Case 1,
on the facet $A_5A_2A_3$ the geodesic $\gamma$
goes from the point $X_2$ to a point $Y_2$ at the edge $A_2A_3$
and $\gamma$ goes from the point $Y_3$ to a point $Y_3$ on $A_2A_3$
and $\dist(A_2, Y_2) < \dist(A_2,Y_3)$. 
As before let $\gamma_1$ be the part of the geodesic that starts at $Y_2$
on the facet $A_6A_2A_3$, goes along the surface and comes to the point $Y_3$. 
And let $\sigma_1$ be the segment $Y_2Y_3$ of the edge $A_2A_3$.
Then $\gamma_1$ and $\sigma_1$ bound the domain $G_1$ on the octahedron.

On a facet $A_5A_4A_3$ the geodesic has the segment $X_4Y_4$. 
If $Y_4$ lies on the edge $A_5A_3$, then $\dist(A_5, Y_4) > \dist(A_5, X_3)$.
But then the segment of $\gamma$ coming out of $X_3$ intersects $X_4Y_4$,
since $X_4$ is the closest to $A_5$ point of $\gamma$ on $A_5A_4$.
Hence $Y_4$ lies on $A_4A_3$. 
The geodesic goes from the point $X_3$ to a point $Y_5$ on the edge $A_4A_3$  
and $\dist(A_4, Y_4) < \dist(A_4,Y_5)$.
We denote by $\gamma_2$ the part of the geodesic  that starts at $Y_4$ on $A_4A_3$
in opposite direction to $X_4$, goes along the surface and comes to $Y_5$.
Denote by $\sigma_2$ the segment $Y_4Y_5$ on the edge $A_4A_3$. 
The segments $\gamma_2$ and $\sigma_2$ bound the domain $G_2$ on the octahedron
(see Figure~\ref{around_A_5_oct_3}). 

 \begin{minipage}{0.4\linewidth}
    \begin{figure}[H]
        \begin{tikzpicture}[line cap=round,line join=round,x=0.6cm,y=0.6cm]

\clip(-2,-2) rectangle (15,13);

\draw [line width=0.5pt] (3,8)-- (6,5);
\draw [line width=0.5pt] (6,5)-- (9,8);
\draw [line width=0.5pt] (6,5)-- (9,2);
\draw [line width=0.5pt] (6,5)-- (3,2);
\draw [line width=0.5pt] (3,8)-- (9,8);
\draw [line width=0.5pt] (9,8)-- (9,2);
\draw [line width=0.5pt] (9,2)-- (3,2);
\draw [line width=0.5pt] (3,2)-- (3,8);
\draw [line width=0.5pt] (7.465,6.465)-- (7.625,3.375);
\draw [line width=0.5pt] (4.665,3.665)-- (4.585,6.415);
\draw [line width=0.5pt] (4.585,6.415)-- (4.54,8.99);
\draw [line width=0.5pt] (7.465,6.465)-- (7.4,9.13);
\draw [line width=0.5pt] (4.665,3.665)-- (4.7,0.41);
\draw [line width=0.5pt] (7.625,3.375)-- (7.72,0.45);
\draw [shift={(5.972588331963846,9.007124075595726)},line width=0.5pt]  plot[domain=0.08587133913966076:3.153545326748522,variable=\t]({1*1.432690672421637*cos(\t r)+0*1.432690672421637*sin(\t r)},{0*1.432690672421637*cos(\t r)+1*1.432690672421637*sin(\t r)});
\draw [shift={(6.209923221140969,0.43579680385679886)},line width=0.5pt]  plot[domain=-3.1245094708321606:0.009405334571825368,variable=\t]({1*1.5101435722572691*cos(\t r)+0*1.5101435722572691*sin(\t r)},{0*1.5101435722572691*cos(\t r)+1*1.5101435722572691*sin(\t r)});
\draw (7.54,0.05) node[anchor=north west] {$ \gamma_1$};
\draw (6.54,11.01) node[anchor=north west] {$ \gamma_2$};
\draw (5.65,0.63) node[anchor=north west] {$G_1$};
\draw (5.5,9.63) node[anchor=north west] {$G_2$};

\begin{scriptsize}
\draw [fill=black] (3,8) circle (1pt);
\draw[color=black] (2.65,8.4) node {$A_4$};
\draw [fill=black] (3,2) circle (1pt);
\draw[color=black] (2.6,1.9) node {$A_3$};
\draw [fill=black] (9,2) circle (1pt);
\draw[color=black] (9.25,1.9) node {$A_2$};
\draw [fill=black] (9,8) circle (1pt);
\draw[color=black] (9.15,8.3) node {$A_1$};
\draw [fill=black] (6,5) circle (1pt);
\draw[color=black] (6.4,5.0) node {$A_5$};

\draw [fill=black] (7.465,6.465) circle (1pt);
\draw[color=black] (7.9,6.4) node {$X_1$};
\draw [fill=black] (7.625,3.375) circle (1pt);
\draw[color=black] (7.95,3.7) node {$X_2$};

\draw [fill=black] (4.585,6.415) circle (1pt);
\draw[color=black] (4.15,6.45) node {$X_4$};
\draw [fill=black] (4.665,3.665) circle (1pt);
\draw[color=black] (4.15,3.85) node {$X_3$};

\draw [fill=black] (4.557300970873785,8) circle (1pt);
\draw[color=black] (4.22,8.3) node {$Y_4$};
\draw [fill=black] (7.427560975609755,8) circle (1pt);
\draw[color=black] (7.74,8.3) node {$Y_1$};

\draw [fill=black] (4.68,2)  circle (1pt);
\draw[color=black] (4.35,1.7) node {$Y_3$};
\draw [fill=black] (7.68,2) circle (1pt);
\draw[color=black] (7.97,1.7) node {$Y_2$};
\end{scriptsize}
\end{tikzpicture}
        \caption{\;}
        \label{around_A_5_oct_2}
    \end{figure}
 \end{minipage}
      \hspace{0.05\linewidth}
 \begin{minipage}{0.4\linewidth}
     \begin{figure}[H]
        \begin{tikzpicture}[line cap=round,line join=round,x=0.6cm,y=0.6cm]

\clip(-2,-2) rectangle (15,13);

\draw [line width=0.5pt] (3,8)-- (6,5);
\draw [line width=0.5pt] (6,5)-- (9,8);
\draw [line width=0.5pt] (6,5)-- (9,2);
\draw [line width=0.5pt] (6,5)-- (3,2);
\draw [line width=0.5pt] (3,8)-- (9,8);
\draw [line width=0.5pt] (9,8)-- (9,2);
\draw [line width=0.5pt] (9,2)-- (3,2);
\draw [line width=0.5pt] (3,2)-- (3,8);
\draw (6.15,0) node[anchor=north west] {$ \gamma_1$};
\draw (-0.1,4.0) node[anchor=north west] {$ \gamma_2$};
\draw (5.25,1.5) node[anchor=north west] {$G_1$};
\draw (1.3,4.6) node[anchor=north west] {$G_2$};

\draw [shift={(5.4,4.23)},line width=0.5pt]  plot[domain=-0.36687262900805084:1.9376689558029478,variable=\t]({1*2.383621194737117*cos(\t r)+0*2.383621194737117*sin(\t r)},{0*2.383621194737117*cos(\t r)+1*2.383621194737117*sin(\t r)});
\draw [shift={(3.1838461336743773,2.0930136694341765)},line width=0.5pt]  plot[domain=-0.0511699343457952:1.6720204509481746,variable=\t]({1*1.818534137395039*cos(\t r)+0*1.818534137395039*sin(\t r)},{0*1.818534137395039*cos(\t r)+1*1.818534137395039*sin(\t r)});

\draw [line width=0.5pt] (7.625,3.375)-- (7,2);
\draw [line width=0.5pt] (4.545,6.455)-- (3,5.45);
\draw [line width=0.5pt] (3,5.45)-- (0.9866548870665336,4.2611671462630625);
\draw [line width=0.5pt] (3,3.903034652322576)-- (1.4127707165118844,3.2423278344446373);
\draw [line width=0.5pt] (5,2)-- (5,0);
\draw [line width=0.5pt] (7,2)-- (6.2788958679371065,-0.16249189925207103);
\draw [shift={(1.350593590695177,3.8143573919583442)},line width=0.5pt]  plot[domain=2.25433276809074:4.820659571245146,variable=\t]({1*0.5762727969813778*cos(\t r)+0*0.5762727969813778*sin(\t r)},{0*0.5762727969813778*cos(\t r)+1*0.5762727969813778*sin(\t r)});
\draw [shift={(5.664108228304564,0.10906699469489697)},line width=0.5pt]  plot[domain=3.304370321124309:5.867245272182152,variable=\t]({1*0.6730047163531648*cos(\t r)+0*0.6730047163531648*sin(\t r)},{0*0.6730047163531648*cos(\t r)+1*0.6730047163531648*sin(\t r)});

\begin{scriptsize}
\draw [fill=black] (3,8) circle (1pt);
\draw[color=black] (2.65,8.4) node {$A_4$};
\draw [fill=black] (3,2) circle (1pt);
\draw[color=black] (2.6,1.9) node {$A_3$};
\draw [fill=black] (9,2) circle (1pt);
\draw[color=black] (9.3,1.9) node {$A_2$};
\draw [fill=black] (9,8) circle (1pt);
\draw[color=black] (9.15,8.3) node {$A_1$};
\draw [fill=black] (6,5) circle (1pt);
\draw[color=black] (6.4,5.0) node {$A_5$};

\draw [fill=black] (7,6) circle (1pt);
\draw[color=black] (7.6,6) node {$X_1$};
\draw [fill=black] (7.625,3.375) circle (1pt);
\draw[color=black] (8,3.5) node {$X_2$};
\draw [fill=black] (4.423525450646084,3.4235254506460837) circle (1pt);
\draw[color=black] (4.9,3.4) node {$X_3$};
\draw [fill=black] (4.545,6.455) circle (1pt);
\draw[color=black] (4.6,6.9) node {$X_4$};

\draw [fill=black] (3,5.45) circle (1pt);
\draw[color=black] (2.6,5.8) node {$Y_4$};
\draw [fill=black] (3,3.9) circle (1pt);
\draw[color=black] (2.75,3.5) node {$Y_5$};

\draw [fill=black] (5,2) circle (1pt);
\draw[color=black] (4.68,1.7) node {$Y_3$};
\draw [fill=black] (7,2) circle (1pt);
\draw[color=black] (7.25,1.7) node {$Y_2$};

\end{scriptsize}
\end{tikzpicture}
        \caption{\;}
        \label{around_A_5_oct_3}
    \end{figure}
 \end{minipage}

In both cases we obtain two domains $G_1$ and $G_2$ bounded by the segment of $\gamma_i$
and the segment of the octahedron's edge $\sigma_i$.
The domains $G_i$, $i=1,2$ are locally convex domains homeomorphic to a disc.
Moreover, on the octahedron the domains $G_1$ and $G_2$ do not intersect.
Otherwise, $\gamma_1$ and $\gamma_2$ intersect. 
If $G_i$ does not contain any vertex of the octahedron  
then $G_i$  forms a lune on a unite sphere. 
Hence the length of $\sigma_i$ and and the length of $\gamma_i$ are
equal to $\pi$. Since $\sigma_i$ is part of the edge, then from 
(\ref{a_triangle}) follows the length of $\sigma_i$ is less than $\pi/2$.
We get a contradiction.

Hence $G_i$ should have at least one octahedron's vertex, $i=1,2$. 
Since $G_i$ is a part of the domain $D_1$, 
and $D_i$ has at most three octahedron's vertices, 
then $G_i$ contains exactly one vertex
$A_{s_i}$, $s_i \in \{1,2,3,4,6\}$ of the octahedron and
has no edges coming out of  $A_{s_i}$. 
Since there is no points of the geodesic on $\sigma_i$, then 
$A_{s_1}$ and $A_{s_2}$ are different vertices of the octahedron,
otherwise $\gamma_1$ and $\gamma_2$ intersect. 

Consider the point $A_{s_1}$ inside $G_1$. 
The behaviour of the geodesic around $A_{s_1}$  is similar to the behaviour around $A_5$.
Hence, the domain $G_1$ is divided  
into  three parts $G'_1$, $G'_2$ and 
$G_1\backslash (G'_1 \cup G'_2)$ between them
with two segments  $\sigma'_1$ and $\sigma'_2$ of the edges.
The domain $G_1\backslash (G'_1 \cup G'_2)$ contains the vertex $A_{s_1}$ and 
$G'_1$, $G'_2$ do not intersect and do not have any vertex of the octahedron.
Since there is no points of geodesic on $\sigma_1$, then $\sigma_1$ belong to the boundary 
of one of the domains $G'_1$ or $G'_2$, for example to $G'_2$ 
(see Figure~\ref{oct_sec_division}).
Then the perimeter of  $G'_1$ is a subarc $\gamma'_1$ of $\gamma_1$ and the subarc of the
octahedon's edge $\sigma'_1$.
Since $G'_1$ does not have any vertex of octahedron, then 
$G'_1$ form a lune on a unite sphere. 
Hence the length of $\sigma'_1$ equals $\pi$,  that leads to a contradiction. 

\begin{figure}[h]
\centering
\begin{tikzpicture}[line cap=round,line join=round,x=1.2cm,y=1.2cm]
\clip(4.,4.5) rectangle (19,8);

\draw [line width=1pt] (7,7)-- (16,7);
\draw [line width=1pt] (7,5)-- (16,5);
\draw [shift={(7,6)},line width=1pt]  plot[domain=1.5707963267948966:4.71238898038469,variable=\t]({1*1*cos(\t r)+0*1*sin(\t r)},{0*1*cos(\t r)+1*1*sin(\t r)});
\draw [shift={(16,6)},line width=1pt]  plot[domain=-1.5707963267948966:1.5707963267948966,variable=\t]({1*1*cos(\t r)+0*1*sin(\t r)},{0*1*cos(\t r)+1*1*sin(\t r)});

\draw [shift={(16,6)},line width=0.8pt,dotted]  plot[domain=2.7331294361508927:3.5777598548344827,variable=\t]({1*3.725724320420394*cos(\t r)+0*3.725724320420394*sin(\t r)},{0*3.725724320420394*cos(\t r)+1*3.725724320420394*sin(\t r)});
\draw [shift={(10,6)},line width=0.8pt,dotted]  plot[domain=-0.4100922358073227:0.3844281212275174,variable=\t]({1*3.981895251527355*cos(\t r)+0*3.981895251527355*sin(\t r)},{0*3.981895251527355*cos(\t r)+1*3.981895251527355*sin(\t r)});

\draw [shift={(11.04707587228614,5.800711415525862)},line width=0.8pt,dotted]  plot[domain=2.586662298546335:3.5469872050411624,variable=\t]({1*3.43787231520742*cos(\t r)+0*3.43787231520742*sin(\t r)},{0*3.43787231520742*cos(\t r)+1*3.43787231520742*sin(\t r)});
\draw [shift={(6.5,5.800711415525862)},line width=0.8pt,dotted]  plot[domain=-0.3906070436976865:0.516403890772842,variable=\t]({1*3.6461445651289233*cos(\t r)+0*3.6461445651289233*sin(\t r)},{0*3.6461445651289233*cos(\t r)+1*3.6461445651289233*sin(\t r)});
\draw [shift={(11.04707587228614,5.800711415525862)},line width=1pt]  plot[domain=2.7852530064860614:3.376660538089388,variable=\t]({1*3.437872315207418*cos(\t r)+0*3.437872315207418*sin(\t r)},{0*3.437872315207418*cos(\t r)+1*3.437872315207418*sin(\t r)});

\draw [shift={(6.5,5.800711415525862)},line width=1pt]  plot[domain=-0.221409522908254:0.3351593392491988,variable=\t]({1*3.646144565128921*cos(\t r)+0*3.646144565128921*sin(\t r)},{0*3.646144565128921*cos(\t r)+1*3.646144565128921*sin(\t r)});

\draw [shift={(10,6)},line width=1pt]  plot[domain=-0.2538544034885071:0.25385440348850735,variable=\t]({1*3.981895251527356*cos(\t r)+0*3.981895251527356*sin(\t r)},{0*3.981895251527356*cos(\t r)+1*3.981895251527356*sin(\t r)});
\draw [shift={(16,6)},line width=1pt]  plot[domain=2.869856607016582:3.4133287001630044,variable=\t]({1*3.7257243204203947*cos(\t r)+0*3.7257243204203947*sin(\t r)},{0*3.7257243204203947*cos(\t r)+1*3.7257243204203947*sin(\t r)});

\draw (11.8,6.3) node[anchor=north west] {$\sigma_1$};
\draw (14,6.3) node[anchor=north west] {$\sigma_2$};
\draw (16.3,6.3) node[anchor=north west] {$\gamma_2$};
\draw (15.0,6.5) node[anchor=north west] {$G_2$};

\draw (10.7,6.5) node[anchor=north west] {$ G'_2$};
\draw (6.2,6.5) node[anchor=north west] {$ G'_1$};
\draw (7.7,6.9) node[anchor=north west] 
{$G_1\backslash (G'_1 \cup G'_2)$};

\draw (7,6.3) node[anchor=north west] {$\sigma'_1$};
\draw (10.2,6.3) node[anchor=north west] {$\sigma'_2$};
\draw (5.5,6.3) node[anchor=north west] {$\gamma'_1$};

\begin{scriptsize}
\draw [fill=black] (13,6) circle (1.5pt);
\draw[color=black] (13.25,6.1) node {$A_5$};
\draw [fill=black] (9,6) circle (1.5pt);
\draw[color=black] (9.3,6.1) node {$A_{s_1}$};

\draw [fill=black] (7.825170698818176,7) circle (1.5pt);
\draw [fill=black] (9.947135853529666,7) circle (1.5pt);
\draw [fill=black] (7.703750090266653,5) circle (1.5pt);
\draw [fill=black] (10.061008479712822,5) circle (1.5pt);

\draw [fill=black] (12.410985969409982,7) circle (1.5pt);
\draw[color=black] (12.2,7.2) node {$Y_2$};
\draw [fill=black] (13.854282007603505,7) circle (1.5pt);
\draw[color=black] (14.0,7.2) node {$Y_1$};

\draw [fill=black] (12.410985969409982,5) circle (1.5pt);
\draw[color=black] (12.2,4.8) node {$Y_3$};
\draw [fill=black] (13.854282007603505,5) circle (1.5pt);
\draw[color=black] (14.0,4.8) node {$Y_4$};
\end{scriptsize}

\end{tikzpicture}
\caption{\;}
\label{oct_sec_division}
\end{figure}

\end{proof}

\begin{lemma}\label{oct_lemma}
A domain $D_i$, $i=1,2$ contains three vertices connected with
two adjacent edges on the octahedron. 
\end{lemma}

\begin{proof}
By Lemma~\ref{edge_inside_oct} if the domain $D_i$ contains only two vertices,
then $D_i$ has an edge connecting these vertices.

Suppose the domain $D_1$ contains just one edge of the octahedron, for example  $A_4A_5$. 
Denote vertices of $\gamma$ as $X_1$, $X_2$, $X_3$, $X_4$, $X_5$ and $X_6$ on the edges 
$A_5A_1$, $A_5A_2$, $A_5A_3$, $A_3A_4$, $A_4A_6$ and $A_1A_4$ respectively. 

Consider an isometry $r$ of the octahedron 
that is a rotation on an angle $\pi$ around a line $A_1A_3$.
The geodesic $\gamma$ maps into a geodesic $\sigma$
that enclose a domain with the edge $r(A_4A_5)=A_2A_6$. 
Denote by $Y_k$ the vertices of $\sigma$ such that 
$Y_k=r(X_k)$ (see Figure~\ref{intersection}).

Since $r(A_6A_4)=A_2A_5$, $r(A_1A_6)=A_1A_5$ and $r(A_3A_6)=A_3A_5$ then 
$\dist(A_5,X_2) = \dist(A_6,Y_2)$ and $\dist(A_4,X_5) = \dist(A_2,Y_5)$.
Hence if $\gamma$ and $\sigma$ intersect, then they intersect in four points 
$Z_1$, $Z_2$, $Z_3$ and $Z_4$ on the facets $A_1A_2A_5$, $A_2A_3A_5$, $A_3A_4A_6$
and $A_1A_4A_6$ respectively.

\begin{figure}[h]
\centering
\begin{tikzpicture}[line cap=round,line join=round,x=1cm,y=1cm]

\clip(3,2.5) rectangle (20,12);

\draw [shift={(19.52727272727274,4.404545454545459)},line width=0.6pt,dotted]  plot[domain=2.5608828750136485:2.7807322313850245,variable=\t]({1*10.199190739246786*cos(\t r)+0*10.199190739246786*sin(\t r)},{0*10.199190739246786*cos(\t r)+1*10.199190739246786*sin(\t r)});
\draw [shift={(18.272727272727284,1.3681818181818222)},line width=0.6pt]  plot[domain=2.2709485956750033:2.7180030412862193,variable=\t]({1*11.287198372822111*cos(\t r)+0*11.287198372822111*sin(\t r)},{0*11.287198372822111*cos(\t r)+1*11.287198372822111*sin(\t r)});
\draw [shift={(12,-2)},line width=0.6pt,dotted]  plot[domain=1.373400766945016:1.7681918866447774,variable=\t]({1*10.19803902718557*cos(\t r)+0*10.19803902718557*sin(\t r)},{0*10.19803902718557*cos(\t r)+1*10.19803902718557*sin(\t r)});
\draw [shift={(5.490909090909096,-1.268181818181814)},line width=0.6pt]  plot[domain=0.8280723168235737:1.116062582183234,variable=\t]({1*12.58188468847372*cos(\t r)+0*12.58188468847372*sin(\t r)},{0*12.58188468847372*cos(\t r)+1*12.58188468847372*sin(\t r)});
\draw [shift={(7,13)},line width=0.6pt]  plot[domain=5.332638466367511:5.662935821196765,variable=\t]({1*8.602325267042627*cos(\t r)+0*8.602325267042627*sin(\t r)},{0*8.602325267042627*cos(\t r)+1*8.602325267042627*sin(\t r)});
\draw [shift={(13.6909090909091,2.331818181818186)},line width=0.6pt,dotted]  plot[domain=2.1479885406683588:2.569045057772295,variable=\t]({1*6.763955576533774*cos(\t r)+0*6.763955576533774*sin(\t r)},{0*6.763955576533774*cos(\t r)+1*6.763955576533774*sin(\t r)});
\draw [shift={(10,15)},line width=0.6pt]  plot[domain=4.493720034510748:4.931057926258632,variable=\t]({1*9.219544457292889*cos(\t r)+0*9.219544457292889*sin(\t r)},{0*9.219544457292889*cos(\t r)+1*9.219544457292889*sin(\t r)});
\draw [shift={(3.3818181818181863,11.004545454545458)},line width=0.6pt]  plot[domain=5.473061313919317:6.007431755606862,variable=\t]({1*11.050314119913528*cos(\t r)+0*11.050314119913528*sin(\t r)},{0*11.050314119913528*cos(\t r)+1*11.050314119913528*sin(\t r)});
\draw [shift={(2.236363636363641,7.622727272727277)},line width=0.6pt]  plot[domain=5.797788614245702:6.118489630866475,variable=\t]({1*9.90812443159986*cos(\t r)+0*9.90812443159986*sin(\t r)},{0*9.90812443159986*cos(\t r)+1*9.90812443159986*sin(\t r)});
\draw [shift={(16,11)},line width=0.6pt]  plot[domain=3.7001919689333556:4.153789665041127,variable=\t]({1*9.433981132056603*cos(\t r)+0*9.433981132056603*sin(\t r)},{0*9.433981132056603*cos(\t r)+1*9.433981132056603*sin(\t r)});
\draw [shift={(21.490909090909117,7.6954545454545515)},line width=0.6pt,dotted]  plot[domain=3.115095691849769:3.5624269997287135,variable=\t]({1*11.494944091617933*cos(\t r)+0*11.494944091617933*sin(\t r)},{0*11.494944091617933*cos(\t r)+1*11.494944091617933*sin(\t r)});
\draw [shift={(1.1834710743801737,5.379752066115707)},line width=0.6pt]  plot[domain=0.05727987679613458:0.4399013705715978,variable=\t]({1*10.83429764674472*cos(\t r)+0*10.83429764674472*sin(\t r)},{0*10.83429764674472*cos(\t r)+1*10.83429764674472*sin(\t r)});

\draw [shift={(12.248975583147116,11.941262602658366)},line width=0.6pt]  plot[domain=4.0474483268920025:4.655618833966844,variable=\t]({1*5.324991533454713*cos(\t r)+0*5.324991533454713*sin(\t r)},{0*5.324991533454713*cos(\t r)+1*5.324991533454713*sin(\t r)});
\draw [shift={(10.483452686392834,2.892524572283593)},line width=0.6pt]  plot[domain=1.0492924055973782:1.1967477363299066,variable=\t]({1*4.0111713628110675*cos(\t r)+0*4.0111713628110675*sin(\t r)},{0*4.0111713628110675*cos(\t r)+1*4.0111713628110675*sin(\t r)});
\draw [shift={(4.788883035786868,3.1012406536401897)},line width=0.6pt,dotted]  plot[domain=0.18581436237430643:0.7095463041060693,variable=\t]({1*7.671565026997815*cos(\t r)+0*7.671565026997815*sin(\t r)},{0*7.671565026997815*cos(\t r)+1*7.671565026997815*sin(\t r)});
\draw [shift={(7.431404958677697,3.2805785123967)},line width=0.6pt,dotted]  plot[domain=0.9875080455032237:1.0832204234087117,variable=\t]({1*5.778212874070339*cos(\t r)+0*5.778212874070339*sin(\t r)},{0*5.778212874070339*cos(\t r)+1*5.778212874070339*sin(\t r)});
\draw [shift={(10.594619344949914,6.092904817129524)},line width=0.6pt,dotted]  plot[domain=1.7710656178240451:2.348104884050531,variable=\t]({1*2.326563106634456*cos(\t r)+0*2.326563106634456*sin(\t r)},{0*2.326563106634456*cos(\t r)+1*2.326563106634456*sin(\t r)});
\draw [shift={(15.125732968663854,8.852101481660558)},line width=0.6pt]  plot[domain=3.0586261204039653:3.7920749603505595,variable=\t]({1*4.9383706050854315*cos(\t r)+0*4.9383706050854315*sin(\t r)},{0*4.9383706050854315*cos(\t r)+1*4.9383706050854315*sin(\t r)});

\draw [shift={(14,11)},line width=0.6pt]  plot[domain=4.212382479256491:4.349490462033025,variable=\t]({1*5.859678784599839*cos(\t r)+0*5.859678784599839*sin(\t r)},{0*5.859678784599839*cos(\t r)+1*5.859678784599839*sin(\t r)});
\draw [shift={(10.880355533126247,9.201714544488034)},line width=0.6pt]  plot[domain=4.987019331362803:5.435148201406076,variable=\t]({1*3.827382980966542*cos(\t r)+0*3.827382980966542*sin(\t r)},{0*3.827382980966542*cos(\t r)+1*3.827382980966542*sin(\t r)});
\draw [shift={(10.3774598419359,10.333229849666317)},line width=0.6pt,dotted]  plot[domain=4.4169963802552275:4.573200534907478,variable=\t]({1*2.7520972265241483*cos(\t r)+0*2.7520972265241483*sin(\t r)},{0*2.7520972265241483*cos(\t r)+1*2.7520972265241483*sin(\t r)});
\draw [shift={(8.289389334049218,-2.3345048989460677)},line width=0.6pt,dotted]  plot[domain=1.038482653531245:1.4008502021011957,variable=\t]({1*10.071714569338848*cos(\t r)+0*10.071714569338848*sin(\t r)},{0*10.071714569338848*cos(\t r)+1*10.071714569338848*sin(\t r)});
\draw [shift={(16.952333708376326,5.649353057649744)},line width=0.6pt,dotted]  plot[domain=2.6501282573268643:2.870384621088868,variable=\t]({1*7.653876519013346*cos(\t r)+0*7.653876519013346*sin(\t r)},{0*7.653876519013346*cos(\t r)+1*7.653876519013346*sin(\t r)});
\draw [shift={(6.6423235133374625,5.940068350158391)},line width=0.6pt]  plot[domain=-0.24498077260447548:0.07357722505858331,variable=\t]({1*5.861066998156201*cos(\t r)+0*5.861066998156201*sin(\t r)},{0*5.861066998156201*cos(\t r)+1*5.861066998156201*sin(\t r)});

\begin{scriptsize}

\draw [fill=black] (12,6) circle (1pt);
\draw[color=black] (11.8,6.2) node {$A_1$};
\draw [fill=black] (8,6) circle (1pt);
\draw[color=black] (7.7944047008218424,5.92717782798725) node {$A_2$};
\draw [fill=black] (10,8) circle (1pt);
\draw [fill=black] (14,8) circle (1pt);
\draw[color=black] (14.217754211025825,7.950190040275692) node {$A_4$};
\draw [fill=black] (11,10) circle (1pt);
\draw[color=black] (11.120373931194367,10.236079545686362) node {$A_5$};
\draw [fill=black] (11,3) circle (1pt);
\draw[color=black] (11.0,2.7) node {$A_6$};

\draw [fill=black] (11.946556700218574,6.619910648365977) circle (1pt);
\draw[color=black] (12.15,6.8) node {$X_1$};
\draw [fill=black] (8.963389612681583,7.750742388884849) circle (1pt);
\draw[color=black] (8.75,7.9) node {$X_2$};
\draw [fill=black] (10.131788592434809,8.372967005684668) circle (1pt);
\draw[color=black] (10.4,8.5) node {$X_3$};
\draw [fill=black] (10.613879893007766,8.103398984945208) circle (1pt);
\draw[color=black] (10.95,7.95) node {$X_4$};
\draw [fill=black] (12.328390705775924,4.518538807746152) circle (1pt);
\draw[color=black] (12.5,4.4) node {$X_5$};
\draw [fill=black] (12.481756916883775,6.370494656145386) circle (1pt);
\draw[color=black] (12.7,6.25) node {$X_6$};

\draw [fill=black] (11.91830795920883,5.517760792853574) circle (1pt);
\draw[color=black] (11.7,5.3) node {$Y_1$};
\draw [fill=black] (13.401068903147277,6.3436381252900444) circle (1pt);
\draw[color=black] (13.6,6.25) node {$Y_2$};
\draw [fill=black] (9.996266130032357,7.612250751139701) circle (1pt);
\draw[color=black] (10.2,7.4) node {$Y_3$};
\draw [fill=black] (9.57628176888508,7.700331667148717) circle (1pt);
\draw[color=black] (9.65,7.5) node {$Y_4$};
\draw [fill=black] (10.204349119050937,9.26135108513134) circle (1pt);
\draw[color=black] (10.05,9.45) node {$Y_5$};
\draw [fill=black] (11.190686911615897,5.857666343952076) circle (1pt);
\draw[color=black] (11.2,5.6) node {$Y_6$};

\draw [fill=black] (10.598937464284033,6.878367281040195) circle (1pt);
\draw[color=black] (10.75,7.1) node {$Z_1$};
\draw [fill=black] (9.758982641303804,8.264219439350827) circle (1pt);
\draw[color=black] (9.65,8.45) node {$Z_2$};
\draw [fill=black] (11.205646404014786,7.3057689077519665) circle (1pt);
\draw[color=black] (11.4,7.5) node {$Z_3$};
\draw [fill=black] (12.499760000479473,5.733806002714081) circle (1pt);
\draw[color=black] (12.7,5.6) node {$Z_4$};
\end{scriptsize}
\end{tikzpicture}
\caption{\;}
\label{intersection}
\end{figure}

The segments $Z_1Y_5Z_2$ and $Z_1X_2Z_2$ form a lune 
and they do not enclose any vertex inside.
Since the facets of the octahedron are the domains of the unite sphere,
then the length of $Z_1X_2Z_2$  is equal to $\pi$.
The same holds for the segments $Z_4Y_2Z_3$ and $Z_4X_5Z_3$.  
In this case, the length of the geodesic $\gamma$ is greater than $2\pi$ 
that contradicts Lemma~\ref{length}.

Thus $\gamma$ and $\sigma$ do not intersect. 
From the Gauss-Bonnet theorem follows that the integral of the curvature 
over the area of a domain enclosed by $\gamma$ and $\sigma$ equals zero.
Since the octahedron has positive curvature in each point, it follows 
the geodesics  $\gamma$ and $\sigma$  coincide after the rotation $r$. 
It leads to a contradiction 
since $\gamma$ can not pass through the vertex of the octahedron.

Therefore, the domain $D_1$ contains exactly three vertices of the octahedron. 
By Lemma~\ref{edge_inside_oct} there is at least one edge inside $D_1$ 
coming out of each vertex. 
Hence $D_1$ contains at least two edges of the octahedron connecting three vertices. 
The same is true for the domain $D_2$. 
\end{proof}

\begin{theorem}\label{geod_oct}
There are only two different simple closed geodesics on  regular spherical octahedra.
\end{theorem}

\begin{proof}
Consider a regular spherical octahedron with  a planar angle $\alpha \in (\pi/3, \pi/2)$.
A simple closed geodesic $\gamma$ divides the surface of the octahedron
into two closed domains $D_1$ and $D_2$. 
From Lemma~\ref{oct_lemma} follows each domain $D_i$ contains three vertices 
of the octahedron connecting with two adjacent edges. 
There are two possibilities for two edges coming out of one vertex: 
1) their endpoints are connected with an edge, and 2) they are not connected with an edge.
In the first case, these edges belong to one facet.

\textbf{Type 1.} Denote  by $X_1, \dots, X_6$ the midpoint of the edges 
$A_1A_2$, $A_2A_5$, $A_5A_3$, $A_3A_4$, $A_4A_6$, $A_6A_1$ respectively.  
Connect $X_1, \dots,  X_6$ consecutively with the segments through the facets 
($X_6$ is connected with $X_1$). 
The broken line $X_1X_2X_3X_4X_5X_6$ forms a simple closed geodesic $\gamma$ 
on the octahedron with a planar angle $\alpha \in (\pi/3, \pi/2)$
(see Figure~\ref{geod_bound_face}). 
The domain $D_1$ contains the facet $A_1A_4A_5$ and $D_2$ has $A_2A_3A_6$. 

If we consider any simple closed broken line $\sigma$ on the octahedron, 
that does not intersect faces $A_1A_4A_5$ and $A_2A_3A_6$,
then this broken line is equivalent to $\gamma$.
In fact, assume $\sigma$ starts at the edge $A_1A_2$. 
Then $\sigma$ can only intersect the edge $A_2A_5$, 
and can not go to $A_1A_5$. 
After it $\sigma$ goes to the edge $A_3A_5$, since it can not go to $A_2A_3$. 
Continuing similarly one can deduce, that $\sigma$ should intersect same edges as $\gamma$
in the same order.
Therefore if there is another geodesic, that encloses the domain with $A_1A_4A_5$
from one side and with $A_2A_3A_6$ from another,
then by Lemma~\ref{uniqness} this geodesic is coincide with $\gamma$.

 In general there are four geodesics that bound a domain with a  facet of an octahedron. 
 They can be constructed similarly if we fix other two non-connected facets
 of the octahedron.
These geodesics can be mapped into each other with the symmetries of the octahedron.

 \begin{figure}[h]
 \begin{tikzpicture}[line cap=round,line join=round,x=0.7cm,y=0.7cm]
 
\clip(5,2) rectangle (15,11);

\draw [shift={(19.52727272727274,4.404545454545459)},line width=0.8pt,dotted]  plot[domain=2.5608828750136485:2.7807322313850245,variable=\t]({1*10.199190739246786*cos(\t r)+0*10.199190739246786*sin(\t r)},{0*10.199190739246786*cos(\t r)+1*10.199190739246786*sin(\t r)});
\draw [shift={(18.272727272727284,1.3681818181818222)},line width=0.8pt]  plot[domain=2.2709485956750033:2.7180030412862193,variable=\t]({1*11.287198372822111*cos(\t r)+0*11.287198372822111*sin(\t r)},{0*11.287198372822111*cos(\t r)+1*11.287198372822111*sin(\t r)});
\draw [shift={(12,-2)},line width=0.8pt,dotted]  plot[domain=1.373400766945016:1.7681918866447774,variable=\t]({1*10.19803902718557*cos(\t r)+0*10.19803902718557*sin(\t r)},{0*10.19803902718557*cos(\t r)+1*10.19803902718557*sin(\t r)});
\draw [shift={(5.490909090909096,-1.268181818181814)},line width=0.8pt]  plot[domain=0.8280723168235737:1.116062582183234,variable=\t]({1*12.58188468847372*cos(\t r)+0*12.58188468847372*sin(\t r)},{0*12.58188468847372*cos(\t r)+1*12.58188468847372*sin(\t r)});
\draw [shift={(7,13)},line width=0.8pt]  plot[domain=5.332638466367511:5.662935821196765,variable=\t]({1*8.602325267042627*cos(\t r)+0*8.602325267042627*sin(\t r)},{0*8.602325267042627*cos(\t r)+1*8.602325267042627*sin(\t r)});
\draw [shift={(13.6909090909091,2.331818181818186)},line width=0.8pt,dotted]  plot[domain=2.1479885406683588:2.569045057772295,variable=\t]({1*6.763955576533774*cos(\t r)+0*6.763955576533774*sin(\t r)},{0*6.763955576533774*cos(\t r)+1*6.763955576533774*sin(\t r)});
\draw [shift={(10,15)},line width=0.8pt]  plot[domain=4.493720034510748:4.931057926258632,variable=\t]({1*9.219544457292889*cos(\t r)+0*9.219544457292889*sin(\t r)},{0*9.219544457292889*cos(\t r)+1*9.219544457292889*sin(\t r)});
\draw [shift={(3.3818181818181863,11.004545454545458)},line width=0.8pt]  plot[domain=5.473061313919317:6.007431755606862,variable=\t]({1*11.050314119913528*cos(\t r)+0*11.050314119913528*sin(\t r)},{0*11.050314119913528*cos(\t r)+1*11.050314119913528*sin(\t r)});
\draw [shift={(2.236363636363641,7.622727272727277)},line width=0.8pt]  plot[domain=5.797788614245702:6.118489630866475,variable=\t]({1*9.90812443159986*cos(\t r)+0*9.90812443159986*sin(\t r)},{0*9.90812443159986*cos(\t r)+1*9.90812443159986*sin(\t r)});
\draw [shift={(16,11)},line width=0.8pt]  plot[domain=3.7001919689333556:4.153789665041127,variable=\t]({1*9.433981132056603*cos(\t r)+0*9.433981132056603*sin(\t r)},{0*9.433981132056603*cos(\t r)+1*9.433981132056603*sin(\t r)});
\draw [shift={(21.490909090909117,7.6954545454545515)},line width=0.8pt,dotted]  plot[domain=3.115095691849769:3.5624269997287135,variable=\t]({1*11.494944091617933*cos(\t r)+0*11.494944091617933*sin(\t r)},{0*11.494944091617933*cos(\t r)+1*11.494944091617933*sin(\t r)});
\draw [shift={(1.1834710743801737,5.379752066115707)},line width=0.8pt]  plot[domain=0.05727987679613458:0.4399013705715978,variable=\t]({1*10.83429764674472*cos(\t r)+0*10.83429764674472*sin(\t r)},{0*10.83429764674472*cos(\t r)+1*10.83429764674472*sin(\t r)});

\draw [shift={(16.571900826446296,8.900413223140495)},line width=0.8pt]  plot[domain=3.262195808322075:3.5717393210745323,variable=\t]({1*7.487288414843506*cos(\t r)+0*7.487288414843506*sin(\t r)},{0*7.487288414843506*cos(\t r)+1*7.487288414843506*sin(\t r)});
\draw [shift={(13.266115702479352,4.123553719008269)},line width=0.8pt,dotted]  plot[domain=2.0987727766296196:2.38755440724099,variable=\t]({1*5.64528893694586*cos(\t r)+0*5.64528893694586*sin(\t r)},{0*5.64528893694586*cos(\t r)+1*5.64528893694586*sin(\t r)});
\draw [shift={(7.894214876033069,1.4789256198347187)},line width=0.8pt,dotted]  plot[domain=1.005256267662155:1.2465566725912016,variable=\t]({1*7.956762366144609*cos(\t r)+0*7.956762366144609*sin(\t r)},{0*7.956762366144609*cos(\t r)+1*7.956762366144609*sin(\t r)});
\draw [shift={(6.257851239669432,5.280578512396698)},line width=0.8pt,dotted]  plot[domain=0.0019911326581333448:0.45904763128406056,variable=\t]({1*6.584189027925571*cos(\t r)+0*6.584189027925571*sin(\t r)},{0*6.584189027925571*cos(\t r)+1*6.584189027925571*sin(\t r)});
\draw [shift={(9.695867768595054,9.016115702479338)},line width=0.8pt]  plot[domain=5.123134542806779:5.414084142909398,variable=\t]({1*4.892803101701065*cos(\t r)+0*4.892803101701065*sin(\t r)},{0*4.892803101701065*cos(\t r)+1*4.892803101701065*sin(\t r)});
\draw [shift={(13.596694214876047,9.511983471074378)},line width=0.8pt]  plot[domain=3.9150180825666276:4.33978393869016,variable=\t]({1*5.337666606468009*cos(\t r)+0*5.337666606468009*sin(\t r)},{0*5.337666606468009*cos(\t r)+1*5.337666606468009*sin(\t r)});

\begin{scriptsize}
\draw [fill=black] (12,6) circle (1.5pt);
\draw[color=black] (12.2,5.7) node {$A_1$};
\draw [fill=black] (8,6) circle (1.5pt);
\draw[color=black] (7.7,5.8) node {$A_2$};
\draw [fill=black] (10,8) circle (1.5pt);
\draw[color=black] (10.4,7.8) node {$A_3$};
\draw [fill=black] (14,8) circle (1.5pt);
\draw[color=black] (14.5,8) node {$A_4$};
\draw [fill=black] (11,10) circle (1.5pt);
\draw[color=black] (11.1,10.4) node {$A_5$};
\draw [fill=black] (11,3) circle (1.5pt);
\draw[color=black] (11.1,2.7) node {$A_6$};

\draw [fill=black] (9.777476859661014,5.78314134577221) circle (1.5pt);
\draw[color=black] (9.5,5.2) node {$X_1$};
\draw [fill=black] (9.138998200698301,7.999610046415217) circle (1.5pt);
\draw[color=black] (8.9,8.3) node {$X_2$};
\draw [fill=black] (10.422096058815583,9.00011392555324) circle (1.5pt);
\draw[color=black] (10.6,8.7) node {$X_3$};
\draw [fill=black] (12.158019970073584,8.1968146834714) circle (1.5pt);
\draw[color=black] (12.3,8.5) node {$X_4$};
\draw [fill=black] (12.842027215730912,5.293688497534858) circle (1.5pt);
\draw[color=black] (13.3,5.3) node {$X_5$};
\draw [fill=black] (11.649529738558233,4.530279426027069) circle (1.5pt);
\draw[color=black] (11.2,4.5) node {$X_6$};
\end{scriptsize} 
\end{tikzpicture}
\includegraphics[width=0.35\textwidth]{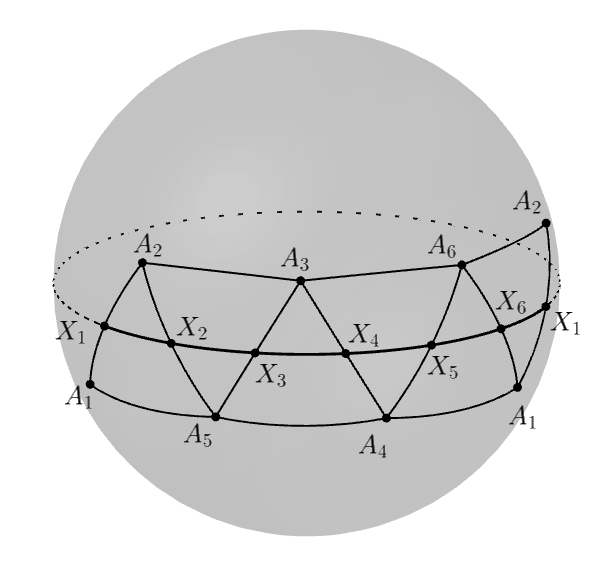}
  \caption{\,}
  \label{geod_bound_face}
\end{figure}

\textbf{Type 2.}  Denote by $X_1$, $X_2$, $X_3$ and $X_4$ the midpoint 
of the edges $A_1A_2$, $A_2A_3$, $A_3A_4$ and $A_4A_1$. 
Unfold two adjacent facets  $A_1A_2A_6$ and $A_2A_6A_3$ into a sphere and 
draw a geodesic line segment $X_1X_2$.
Since $\alpha < \pi/2$, then the segment  $X_1X_2$ is contained inside the development
and intersects the edge  $A_2A_6$  at right angle.
Then unfold another  two adjacent facets $A_2A_3A_5$ and $A_3A_5A_4$ 
and construct the segment $X_2X_3$.
In the same way connect the points  $X_3$ and $X_4$ within the edge  $A_4A_6$,
and connect  $X_4$ and $X_1$ within $A_1A_5$.
The broken line $X_1X_2X_3X_4$ form a geodesic on the spherical octahedron
with a planar angle $\alpha \in (\pi/3, \pi/2)$ 
(see Figure~\ref{geod_not_bound_face}).
The domain $D_1$ contains the edges $A_2A_5$ and $A_5A_4$  that don't bound a facet. 
The complementary domain $D_2$ contains the edges $A_1A_6$ and $A_6A_3$.  
 
\begin{figure}[h]
\centering
 \begin{tikzpicture}[line cap=round,line join=round,x=0.7cm,y=0.7cm]

\clip(5,2) rectangle (15,11);

\draw [shift={(19.52727272727274,4.404545454545459)},line width=0.8pt,dotted]  plot[domain=2.5608828750136485:2.7807322313850245,variable=\t]({1*10.199190739246786*cos(\t r)+0*10.199190739246786*sin(\t r)},{0*10.199190739246786*cos(\t r)+1*10.199190739246786*sin(\t r)});
\draw [shift={(18.272727272727284,1.3681818181818222)},line width=0.8pt]  plot[domain=2.2709485956750033:2.7180030412862193,variable=\t]({1*11.287198372822111*cos(\t r)+0*11.287198372822111*sin(\t r)},{0*11.287198372822111*cos(\t r)+1*11.287198372822111*sin(\t r)});
\draw [shift={(12,-2)},line width=0.8pt,dotted]  plot[domain=1.373400766945016:1.7681918866447774,variable=\t]({1*10.19803902718557*cos(\t r)+0*10.19803902718557*sin(\t r)},{0*10.19803902718557*cos(\t r)+1*10.19803902718557*sin(\t r)});
\draw [shift={(5.490909090909096,-1.268181818181814)},line width=0.8pt]  plot[domain=0.8280723168235737:1.116062582183234,variable=\t]({1*12.58188468847372*cos(\t r)+0*12.58188468847372*sin(\t r)},{0*12.58188468847372*cos(\t r)+1*12.58188468847372*sin(\t r)});
\draw [shift={(7,13)},line width=0.8pt]  plot[domain=5.332638466367511:5.662935821196765,variable=\t]({1*8.602325267042627*cos(\t r)+0*8.602325267042627*sin(\t r)},{0*8.602325267042627*cos(\t r)+1*8.602325267042627*sin(\t r)});
\draw [shift={(13.6909090909091,2.331818181818186)},line width=0.8pt,dotted]  plot[domain=2.1479885406683588:2.569045057772295,variable=\t]({1*6.763955576533774*cos(\t r)+0*6.763955576533774*sin(\t r)},{0*6.763955576533774*cos(\t r)+1*6.763955576533774*sin(\t r)});
\draw [shift={(10,15)},line width=0.8pt]  plot[domain=4.493720034510748:4.931057926258632,variable=\t]({1*9.219544457292889*cos(\t r)+0*9.219544457292889*sin(\t r)},{0*9.219544457292889*cos(\t r)+1*9.219544457292889*sin(\t r)});
\draw [shift={(3.3818181818181863,11.004545454545458)},line width=0.8pt]  plot[domain=5.473061313919317:6.007431755606862,variable=\t]({1*11.050314119913528*cos(\t r)+0*11.050314119913528*sin(\t r)},{0*11.050314119913528*cos(\t r)+1*11.050314119913528*sin(\t r)});
\draw [shift={(2.236363636363641,7.622727272727277)},line width=0.8pt]  plot[domain=5.797788614245702:6.118489630866475,variable=\t]({1*9.90812443159986*cos(\t r)+0*9.90812443159986*sin(\t r)},{0*9.90812443159986*cos(\t r)+1*9.90812443159986*sin(\t r)});
\draw [shift={(16,11)},line width=0.8pt]  plot[domain=3.7001919689333556:4.153789665041127,variable=\t]({1*9.433981132056603*cos(\t r)+0*9.433981132056603*sin(\t r)},{0*9.433981132056603*cos(\t r)+1*9.433981132056603*sin(\t r)});
\draw [shift={(21.490909090909117,7.6954545454545515)},line width=0.8pt,dotted]  plot[domain=3.115095691849769:3.5624269997287135,variable=\t]({1*11.494944091617933*cos(\t r)+0*11.494944091617933*sin(\t r)},{0*11.494944091617933*cos(\t r)+1*11.494944091617933*sin(\t r)});
\draw [shift={(1.1834710743801737,5.379752066115707)},line width=0.8pt]  plot[domain=0.05727987679613458:0.4399013705715978,variable=\t]({1*10.83429764674472*cos(\t r)+0*10.83429764674472*sin(\t r)},{0*10.83429764674472*cos(\t r)+1*10.83429764674472*sin(\t r)});

\draw [shift={(15.067768595041336,3.6938016528925672)},line width=0.8pt,dotted]  plot[domain=2.291862570920127:2.6303444967581973,variable=\t]({1*7.047038974897471*cos(\t r)+0*7.047038974897471*sin(\t r)},{0*7.047038974897471*cos(\t r)+1*7.047038974897471*sin(\t r)});
\draw [shift={(7.183471074380177,3.925206611570253)},line width=0.8pt,dotted]  plot[domain=0.7843889226740518:1.002536281886817,variable=\t]({1*6.02779947585806*cos(\t r)+0*6.02779947585806*sin(\t r)},{0*6.02779947585806*cos(\t r)+1*6.02779947585806*sin(\t r)});

\draw [shift={(6.42314049586778,4.024380165289261)},line width=0.8pt,dotted]  plot[domain=0.19105187646110797:0.6911721126675402,variable=\t]({1*6.518388990675975*cos(\t r)+0*6.518388990675975*sin(\t r)},{0*6.518388990675975*cos(\t r)+1*6.518388990675975*sin(\t r)});
\draw [shift={(6.571900826446292,7.065702479338844)},line width=0.8pt]  plot[domain=6.002295407145787:6.26308273361386,variable=\t]({1*6.506002718784147*cos(\t r)+0*6.506002718784147*sin(\t r)},{0*6.506002718784147*cos(\t r)+1*6.506002718784147*sin(\t r)});

\draw [shift={(8.059504132231416,3.660743801652898)},line width=0.8pt]  plot[domain=0.5760171584082338:0.954672474053413,variable=\t]({1*6.010248406928756*cos(\t r)+0*6.010248406928756*sin(\t r)},{0*6.010248406928756*cos(\t r)+1*6.010248406928756*sin(\t r)});
\draw [shift={(15.613223140495883,4.0574380165289305)},line width=0.8pt]  plot[domain=2.305254223149517:2.853264035990184,variable=\t]({1*6.0824220025026845*cos(\t r)+0*6.0824220025026845*sin(\t r)},{0*6.0824220025026845*cos(\t r)+1*6.0824220025026845*sin(\t r)});

\draw [shift={(17.18347107438018,6.73512396694215)},line width=0.8pt,dotted]  plot[domain=3.0929532528026713:3.422115278362287,variable=\t]({1*8.262860348957709*cos(\t r)+0*8.262860348957709*sin(\t r)},{0*8.262860348957709*cos(\t r)+1*8.262860348957709*sin(\t r)});
\draw [shift={(13.728925619834724,3.3136363636363693)},line width=0.8pt]  plot[domain=2.580954542922314:2.896697184403031,variable=\t]({1*4.643494774235665*cos(\t r)+0*4.643494774235665*sin(\t r)},{0*4.643494774235665*cos(\t r)+1*4.643494774235665*sin(\t r)});

\begin{scriptsize}
\draw [fill=black] (12,6) circle (1.5pt);
\draw[color=black] (12.2,5.7) node {$A_1$};
\draw [fill=black] (8,6) circle (1.5pt);
\draw[color=black] (7.7,5.8) node {$A_2$};
\draw [fill=black] (10,8) circle (1.5pt);
\draw[color=black] (10.4,7.8) node {$A_3$};
\draw [fill=black] (14,8) circle (1.5pt);
\draw[color=black] (14.5,8) node {$A_4$};
\draw [fill=black] (11,10) circle (1.5pt);
\draw[color=black] (11.1,10.4) node {$A_5$};
\draw [fill=black] (11,3) circle (1.5pt);
\draw[color=black] (11.1,2.7) node {$A_6$};

\draw [fill=black] (9.796275678732934,5.782706666221134) circle (1.5pt);
\draw[color=black] (9.9,5.3) node {$X_1$};
\draw [fill=black] (8.930382900217605,7.136866093032252) circle (1.5pt);
\draw[color=black] (9.3,6.8) node {$X_2$};
\draw [fill=black] (11.450068472545215,8.183200641994208) circle (1.5pt);
\draw[color=black] (11.4,7.7) node {$X_3$};
\draw [fill=black] (13.099931056521378,6.934454591243587) circle (1.5pt);
\draw[color=black] (13.1,7.4) node {$X_4$};

\draw [fill=black] (10.415412788712523,8.98684839036264) circle (1pt);
\draw [fill=black] (11.536875710531728,8.5717761791784) circle (1pt);
\draw [fill=black] (9.219595800748415,4.440570230974992) circle (1pt);
\draw [fill=black] (12.822927631313716,5.262168376069461) circle (1pt);
\end{scriptsize}
\end{tikzpicture}
\includegraphics[width=0.4\textwidth]{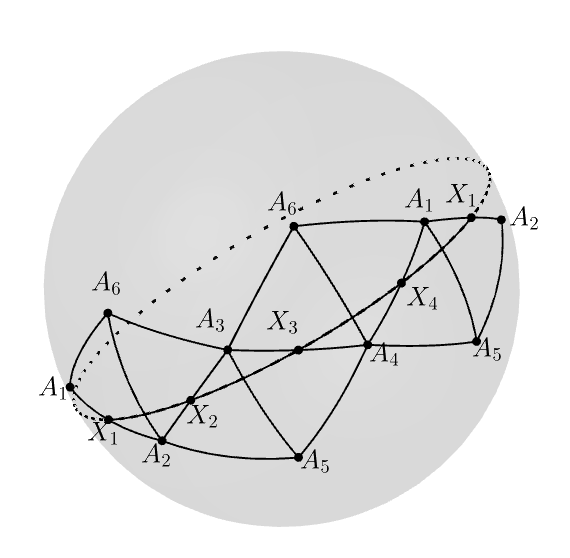}
\caption{ }
\label{geod_not_bound_face}
\end{figure}

Any simple closed broken line $\sigma$ on the octahedron, 
that enclosed the domain with the edges $A_2A_5$ and $A_5A_4$  from one side
and the domain with  $A_1A_6$ and $A_6A_3$ from other,
 is equivalent to $\gamma$.
 If $\sigma$ starts at the edge $A_1A_2$, 
then $\sigma$ can only go to  the edge $A_2A_6$, 
and can not intersect $A_1A_6$. 
After, $\sigma$ goes to the edge $A_2A_3$, since it can not go to $A_3A_6$.
After, $\sigma$ intersects $A_3A_5$ and can not intersect $A_2A_5$.
Continuing similarly we see  that $\sigma$ intersects same edges as $\gamma$
in the same order.
By Lemma~\ref{uniqness} follows that $\gamma$ is a unique geodesic, that bound
the domain with $A_2A_5$ and $A_5A_4$
from one side and with $A_1A_6$ and $A_6A_3$  from another.

In general there are six geodesics of such type and 
they  coincide up to octahedral symmetry.
\end{proof}

\section{Cube}

For any $\alpha \in (\pi/2, 2\pi/3)$ 
a geodesic $\gamma$ divides the surface of the cube with a planar angle $\alpha$
into two domains $D_1$ and $D_2$.

\begin{lemma}\label{edge_inside_cube}
If the domain $D_i$  contains a vertex of a cube, then 
$D_i$ contains at least one edge coming out of this vertex.
\end{lemma}

\begin{proof}

 Since the plane angle $\alpha$ of a cube is greater than $\pi/2$,
 then a simple geodesic cannot sequentially cross all three edges
 coming out of a vertex of a cube. 
Therefore, a domain $D_i$ contains at least two vertices of the cube.

Without loss of generality, we can assume the domain $D_1$ has at most $4$
vertices of the cube. If not, we can consider $D_2$. 
For proof of this lemma, we will use the following labeling of the vertices of the cube.
Fix the vertex $A_4$.
Let $A_1$, $A_2$ and $A_3$ be vertices such that
there is an edge $A_4A_i$ on a cube. 
Let $A'_4$ be a vertex such that $A_4$ and $A'_4$ do not belong to any 
facet of the cube and 
let $A'_1$, $A'_2$ and $A'_3$ be the endpoints of the edges coming out of  $A'_4$.
For certainty choose $A'_1$ such that $A'_1$ form a facet with $A_1$, $A_3$ and $A_4$.

Suppose $D_1$ contains the vertex $A_4$ inside 
and does not have any edge coming out of it.
Then $\gamma$ intersects all edges $A_4A_i$, $i=1,2,3$ at least once.
Let $X_i$ be the closest point of $\gamma$ to $A_4$ on their edge $A_4A_i$, $i=1,2,3$.

\textbf{Case 1.}
Assume there is no $X_iX_{i+1}$ segments of $\gamma$
inside the facets $A_4A_iA'_{i+1}A_{i+1}$,
$i=1,2,3$, when $i+1>3$   take $i+1 \mod 4$.
On the facet $A_4A_1A'_2A_2$ the geodesic $\gamma$ goes
from the point $X_1$ to the point $Y_1$.
If $Y_1$   lies on the edge $A_2A_4$,  then $\dist(A_4, X_2)< \dist(A_4, Y_1)$
and the segment of $\gamma$ coming out of $X_2$ intersects $X_1Y_1$, that 
leads to a contradiction. 
Hence, $Y_1$ belongs to the edge $A'_2A_2$ or $A_1A'_2$. 

\textit{Subcase 1.1.} If $Y_1$ belongs to the edge $A'_2A_2$, 
then a segment of $\gamma$ coming out of $X_2$ on this facet goes to the point $Y_2$
on the edge $A_2A'_2$ and 
$\dist(A_2, Y_2)< \dist(A_2, Y_1)$.

Now consider the facet $A_4A_2A'_3A_3$. 
The segment of $\gamma$ coming out of $X_2$ goes to the point $Z_2$,
and from the point $X_3$ goes to the point $Y_3$.
If $Z_2$ lies on the segment $A_4A_3$, then $\dist(A_4, X_3)< \dist(A_4, Z_2)$. 
In this case the segment $X_3Y_3$ intersects $X_2Z_2$. 
If $Z_2$ lies on the edge $A_2A'_3$, then $\gamma$ intersects all three edges 
coming out of $A_2$.
Hence $Z_2$ should lie on the edge $A_3A'_3$. 
The point  $Y_3$ also belong to the edge $A_3A'_3$ and $\dist(A_3, Y_3)<\dist(A_3, Z_2)$. 
On the facet $A_4A_3A'_1A_1$ the segment of $\gamma$ goes to 
a point $Z_3$  starting at $X_3$  and goes to the point $Z_1$ starting at $X_1$. 
As before $Z_3$ can not be on the edge $A_1A_4$ and on the edge $A'_1A_3$. 
Hence both points $Z_3$  and $Z_1$ lie on $A_1A'_1$ 
and $\dist(A_1, Z_1) < \dist(A_1, Z_3)$. 

\textit{Subcase 1.2.} 
If $Y_1$ belongs to the edge $A_1A'_2$, then on the facet $A_4A_3A'_1A_1$
the point $Z_1$ can be only on the edge $A'_1A_3$.
Hence $Z_3$ should also be on the edge $A'_1A_3$.
From this follows on the facet $A_4A_2A'_3A_3$
the point $Y_3$ lies on $A_2A'_3$ and thus $Z_2$ as well.
But then $Y_2$ can be only on $A_1A'_2$. 

Notice, that the case, when $Y_1$ and $Y_2$ do not belong to the same edge 
is not possible. 
In fact, if $Y_1$ lies on $A_1A'_2$ and $Y_2$ lies on $A_2A'_2$, 
then $Z_1$ should belong to $A'_1A_3$ and $Z_2$ belongs to $A_3A'_3$.
But then $Z_3$ is also on $A'_1A_3$ and  $Q_3$ is on $A_3A'_3$.
Hence the segment $Z_3X_3Y_3$ of $\gamma$ intersects all three edges coming out of 
$A_3$ that leads to a contradiction.

In both Subcases $1.1$ and $1.2$ the segments $X_1Y_1$ and $X_2Y_2$ intersect
the diagonal $A_1A_2$ at the points $Q_1$ and $P_1$ respectively.
The segments  $X_2Z_2$ and $X_3Y_3$ intersect
the diagonal $A_2A_3$ at the points $Q_2$ and $P_2$ respectively.
And the segments $X_3Z_3$ and $X_1Z_1$ intersect
the diagonal $A_1A_3$ at the points $Q_3$ and $P_3$ respectively.

 There is no points of $\gamma$ between $Q_i$ and $P_i$, $i=1,2,3$.
 Assume it's not true and there is a point $R$ on the segment $P_1Q_1$.
Since there is no points of $\gamma$ on the segments $A_4X_i$,
then the part of a geodesic that starts at $R$
and goes inside the triangle $A_1A_3A_4$ should intersect $P_1Q_1$ again at a point $S$.
The segment $RS$ does not intersect any edge of the cube.
Then the segment $RS$ of $\gamma$ and segment of the diagonal form a lune on a sphere.
Hence the length of the segment of the diagonal $A_1A_3$ between $R$ and $S$ equals $\pi$
that contradicts to (\ref{diagonal}).

Consider the domain of the cube outside the triangles $A_4A_iA_{i+1}$,
$i=1,2,3$, when $i+1>3$ then take $i+1\mod 4$.
Consider the segment $\gamma_i$ of the geodesic $\gamma$ coming out of $P_i$. 
Since the domain $D_i$ is homeomorphic to a disc, 
then it's easy to show, that $\gamma_i$ comes to the point $Q_i$,
and does not have other points $P_j$, $Q_k$, $j,k\neq i$ on it. 
Denote by $\sigma_i$ the part of the diagonal 
of the facet between the points $Q_i$ and $P_i$.
The segments $\gamma_i$ and $\sigma_i$ enclose a domain $G_i$ inside $D_1$.
Each $G_i$ is locally convex and homeomorphic to a disc, $i=1,2,3$.

 \begin{minipage}{0.4\linewidth}
     \begin{figure}[H]
         \begin{tikzpicture}[line cap=round,line join=round,x=0.65cm,y=0.65cm]

\clip(2,2) rectangle (15,12);

\draw [line width=0.8pt] (9,8)-- (5,8);
\draw [line width=0.8pt] (9,8)-- (12,10);
\draw [line width=0.8pt] (9,8)-- (9,4);
\draw [line width=0.8pt] (9,4)-- (12,6);
\draw [line width=0.8pt] (12,6)-- (12,10);
\draw [line width=0.8pt] (5,8)-- (5,4);
\draw [line width=0.8pt] (5,4)-- (9,4);
\draw [line width=0.8pt] (5,8)-- (8,10);
\draw [line width=0.8pt] (8,10)-- (12,10);
\draw [line width=0.8pt] (7.575854681749002,8)-- (9.684881515908215,10);
\draw [line width=0.8pt] (5,6.422194535186725)-- (7.575854681749002,8);
\draw [line width=0.8pt] (5,6.422194535186725)-- (2.9906881920375037,5.368025827084078);
\draw [line width=0.8pt] (5,5.377355715963261)-- (8.321928810384462,7.2269397135808475);
\draw [line width=0.8pt] (10.072992318329009,4.715328212219339)-- (9.329058802140915,6.758668053241739);

\draw [shift={(8.656111766030888,6.616994993008048)},line width=0.8pt]  plot[domain=0.20749622643520388:2.0720185650167093,variable=\t]({1*0.6876983127834834*cos(\t r)+0*0.6876983127834834*sin(\t r)},{0*0.6876983127834834*cos(\t r)+1*0.6876983127834834*sin(\t r)});
\draw [line width=0.8pt] (5,5.377355715963261)-- (3.479273457942146,4.571812801165403);
\draw [shift={(3.2078371991062338,4.9518235635356795)},line width=0.8pt]  plot[domain=2.0516836069902338:5.332638466367512,variable=\t]({1*0.4694443688593419*cos(\t r)+0*0.4694443688593419*sin(\t r)},{0*0.4694443688593419*cos(\t r)+1*0.4694443688593419*sin(\t r)});
\draw [line width=0.8pt] (10.949940231491189,5.299960154327459)-- (9.975647919414985,7.449037144826071);
\draw [shift={(12.346191246581956,8.263345921333807)},line width=0.8pt]  plot[domain=2.3767545561551477:3.472475374205845,variable=\t]({1*2.508092758006147*cos(\t r)+0*2.508092758006147*sin(\t r)},{0*2.508092758006147*cos(\t r)+1*2.508092758006147*sin(\t r)});
\draw [line width=0.8pt] (10.072992318329009,4.715328212219339)-- (10.783267594403467,3.2792811035810665);
\draw [line width=0.8pt] (10.949940231491189,5.299960154327459)-- (11.803318728910131,3.9746530317263957);
\draw [shift={(11.281826870841563,3.6452448277749863)},line width=0.8pt]  plot[domain=-2.5083827766654605:0.5633778416984945,variable=\t]({1*0.6184584056736241*cos(\t r)+0*0.6184584056736241*sin(\t r)},{0*0.6184584056736241*cos(\t r)+1*0.6184584056736241*sin(\t r)});
\draw [line width=0.8pt] (9.684881515908215,10)-- (10.627094940621177,11.104378763816355);
\draw [line width=0.8pt] (10.536616187675872,10)-- (11.230833797235123,10.772094805867756);
\draw [shift={(10.919779109955785,10.920216085524583)},line width=0.8pt]  plot[domain=-0.4444192099010982:2.579964619369558,variable=\t]({1*0.3445213084347943*cos(\t r)+0*0.3445213084347943*sin(\t r)},{0*0.3445213084347943*cos(\t r)+1*0.3445213084347943*sin(\t r)});

\draw [line width=0.8pt] (5,8)-- (12,10);
\draw [line width=0.8pt] (12,10)-- (9,4);
\draw [line width=0.8pt] (9,4)-- (5,8);

\draw (11.3,11.5) node[anchor=north west] {$\gamma_1$};
\draw (11.8,3.3) node[anchor=north west] {$\gamma_2$};
\draw (2.3,4.8) node[anchor=north west] {$\gamma_3$};

\begin{scriptsize}

\draw [fill=black] (9,8) circle (1.5pt);
\draw[color=black] (9.3,7.9) node {$A_4$};

\draw [fill=black] (5,8) circle (1.5pt);
\draw[color=black] (4.6,8.0) node {$A_1$};
\draw [fill=black] (12,10) circle (1.5pt);
\draw[color=black] (12.35,10.0) node {$A_2$};
\draw [fill=black] (9,4) circle (1.5pt);
\draw[color=black] (9.2,3.7) node {$A_3$};

\draw [fill=black] (5,4) circle (1.5pt);
\draw[color=black] (4.7,3.9) node {$A'_1$};
\draw [fill=black] (8,10) circle (1.5pt);
\draw[color=black] (8.0,10.3) node {$A'_2$};
\draw [fill=black] (12,6) circle (1.5pt);
\draw[color=black] (12.25,5.9) node {$A'_3$};

\draw [fill=black] (7.575854681749002,8) circle (1.5pt);
\draw[color=black] (7.7,7.65) node {$X_1$};
\draw [fill=black] (9.684881515908215,10) circle (1.5pt);
\draw[color=black] (9.5,10.35) node {$Y_1$};
\draw [fill=black] (5,6.422194535186725) circle (1.5pt);
\draw[color=black] (4.655,6.6) node {$Z_1$};

\draw [fill=black] (9.857087578753067,8.571391719168712) circle (1.5pt);
\draw[color=black] (10.2,8.5) node {$X_2$};
\draw [fill=black] (10.536616187675872,10) circle (1.5pt);
\draw[color=black] (11.15,10.3) node {$Y_2$};
\draw [fill=black] (10.949940231491189,5.299960154327459) circle (1.5pt);
\draw[color=black] (11.3,5.2) node {$Z_2$};

\draw [fill=black] (9,7.212536638858999) circle (1.5pt);
\draw[color=black] (8.65,7.0) node {$X_3$};
\draw [fill=black] (10.072992318329009,4.715328212219339) circle (1.5pt);
\draw[color=black] (9.85,4.2) node {$Y_3$};
\draw [fill=black] (5,5.377355715963261) circle (1.5pt);
\draw[color=black] (4.6,5.5) node {$Z_3$};

\draw [fill=black] (10.147891763968762,9.47082621827679) circle (1.5pt);
\draw[color=black] (10.317157115915963,9.2) node {$P_1$};
\draw [fill=black] (8.686583868047734,9.053309676585068) circle (1.5pt);
\draw[color=black] (8.558811626295716,9.471505254209116) node {$Q_1$};

\draw [fill=black] (9.771590190827649,5.543180381655296) circle (1.5pt);
\draw[color=black] (10.12,5.58) node {$P_2$};
\draw [fill=black] (10.331762693434765,6.663525386869529) circle (1.5pt);
\draw[color=black] (10.7,6.7) node {$Q_2$};

\draw [fill=black] (5.97846175421745,7.02153824578255) circle (1.5pt);
\draw[color=black] (5.87,6.7) node {$P_3$};
\draw [fill=black] (6.684659388844798,6.315340611155201) circle (1.5pt);
\draw[color=black] (6.6,6.0) node {$Q_3$};

\end{scriptsize}
\end{tikzpicture}
         \caption{\;  }
         \label{around_A_1_cube_1}
     \end{figure}
 \end{minipage}
      \hspace{0.05\linewidth}
 \begin{minipage}{0.4\linewidth}
     \begin{figure}[H]
        \begin{tikzpicture}[line cap=round,line join=round,x=0.65cm,y=0.65cm]

\clip(2,2) rectangle (15,12);

\draw [line width=0.8pt] (9,8)-- (5,8);
\draw [line width=0.8pt] (9,8)-- (12,10);
\draw [line width=0.8pt] (9,8)-- (9,4);
\draw [line width=0.8pt] (9,4)-- (12,6);
\draw [line width=0.8pt] (12,6)-- (12,10);
\draw [line width=0.8pt] (5,8)-- (5,4);
\draw [line width=0.8pt] (5,4)-- (9,4);
\draw [line width=0.8pt] (5,8)-- (8,10);
\draw [line width=0.8pt] (8,10)-- (12,10);
\draw [line width=0.8pt] (5,8)-- (12,10);
\draw [line width=0.8pt] (12,10)-- (9,4);
\draw [line width=0.8pt] (9,4)-- (5,8);

\draw [line width=0.8pt] (7.575854681749002,8)-- (9.770337449674232,8.51355829978282);
\draw [line width=0.8pt] (9.770337449674232,8.51355829978282)-- (12,8.51271133494717);
\draw [line width=0.8pt] (7.575854681749002,8)-- (5,5.2381907478039595);
\draw [line width=0.8pt] (5,5.2381907478039595)-- (3.09688496309929,3.6520447917313748);
\draw [line width=0.8pt] (6.007332671137244,4)-- (4.013139490774205,2.2110341934399305);

\draw [shift={(3.5550122269367472,2.9315394925856526)},line width=0.8pt]  plot[domain=2.1371539107662105:5.278746564356004,variable=\t]({1*0.8538199318171613*cos(\t r)+0*0.8538199318171613*sin(\t r)},{0*0.8538199318171613*cos(\t r)+1*0.8538199318171613*sin(\t r)});
\draw [line width=0.8pt] (12,8.51271133494717)-- (14.093941162202173,8.431969365790762);
\draw [line width=0.8pt] (12,6.939752162011551)-- (13.819418467070383,7.027059102469253);
\draw [shift={(13.956679814636278,7.729514234130008)},line width=0.8pt]  plot[domain=-1.7637671367604666:1.3778255168293263,variable=\t]({1*0.7157400991506161*cos(\t r)+0*0.7157400991506161*sin(\t r)},{0*0.7157400991506161*cos(\t r)+1*0.7157400991506161*sin(\t r)});
\draw [line width=0.8pt] (6.007332671137244,4)-- (7.583934037470551,5.416065962529449);
\draw [line width=0.8pt] (10.402846384014591,6.805692768029183)-- (12,6.939752162011551);
\draw [shift={(10.86426239594582,2.3924700728914012)},line width=0.8pt]  plot[domain=1.6749709144578102:2.396897814290708,variable=\t]({1*4.437278365501279*cos(\t r)+0*4.437278365501279*sin(\t r)},{0*4.437278365501279*cos(\t r)+1*4.437278365501279*sin(\t r)});

\draw (2.,2.5) node[anchor=north west] {$\gamma_1$};
\draw (13.5,8) node[anchor=north west] {$\gamma_2$};

\begin{scriptsize}
\draw [fill=black] (9,8) circle (1.5pt);
\draw[color=black] (9.3,7.8) node {$A_4$};
\draw [fill=black] (5,8) circle (1.5pt);
\draw[color=black] (4.5,8.15) node {$A_1$};
\draw [fill=black] (12,10) circle (1.5pt);
\draw[color=black] (12.4,10.0) node {$A_2$};
\draw [fill=black] (9,4) circle (1.5pt);
\draw[color=black] (9.2,3.75) node {$A_3$};

\draw [fill=black] (5,4) circle (1.5pt);
\draw[color=black] (4.65,3.8) node {$A'_1$};
\draw [fill=black] (8,10) circle (1.5pt);
\draw[color=black] (8.12,10.4) node {$A'_2$};
\draw [fill=black] (12,6) circle (1.5pt);
\draw[color=black] (12.3,5.85) node {$A'_3$};

\draw [fill=black] (7.575854681749002,8) circle (1.5pt);
\draw[color=black] (7.7,7.7) node {$X_1$};
\draw [fill=black] (9.770337449674232,8.51355829978282) circle (1.5pt);
\draw[color=black] (9.5,8.9) node {$X_2$};
\draw [fill=black] (9,6.419126876246768) circle (1.5pt);
\draw[color=black] (8.5,6.65) node {$X_3$};

\draw [fill=black] (12,8.51271133494717) circle (1.5pt);
\draw[color=black] (12.25,8.85) node {$Y_2$};
\draw [fill=black] (12,6.939752162011551) circle (1.5pt);
\draw[color=black] (12.25,7.3) node {$Y_3$};

\draw [fill=black] (11.256496881891017,8.512993763782031) circle (1.5pt);
\draw[color=black] (11.1,8.95) node {$P_2$};
\draw [fill=black] (10.402846384014591,6.805692768029183) circle (1.5pt);
\draw[color=black] (10.55,6.5) node {$Q_2$};

\draw [fill=black] (5,5.2381907478039595) circle (1.5pt);
\draw[color=black] (4.6,5.56) node {$Y_1$};
\draw [fill=black] (6.007332671137244,4) circle (1.5pt);
\draw[color=black] (6.1,3.65) node {$Z_3$};

\draw [fill=black] (6.332796403146523,6.667203596853477) circle (1.5pt);
\draw[color=black] (6.3,6.25) node {$Q_1$};
\draw [fill=black] (7.583934037470551,5.416065962529449) circle (1.5pt);
\draw[color=black] (7.13,5.5) node {$P_1$};

\end{scriptsize}
\end{tikzpicture}
        \caption{\;}
        \label{around_A_1_cube_2}
    \end{figure}
 \end{minipage}

\textbf{ Case 2.}
Assume there is only one of the three segments $X_iX_{i+1}$, $i=1,2,3$ of $\gamma$
inside the facets $A_4A_iA'_{i+1}A_{i+1}$, for example $X_1X_2$.
If there are two of them, then $\gamma$ intersects sequentially
three edges coming out of $A_4$.

On the facet $A_4A_2A'_3A_3$ the geodesic $\gamma$ goes
 from the point $X_2$ to the point $Y_2$ and from the point $X_3$ to $Y_3$.
On the facet $A_4A_3A'_1A_1$ $\gamma$  goes
from the point $X_1$ to the point $Y_1$ and from the point $X_3$ to $Z_3$.

If $Y_2$ is on the edge $A_3A'_3$, then $Y_3$ should also be on this edge.
In this case $Z_3$ should be on the edge $A_1A'_1$.
If $Y_2$ is on the edge $A_2A'_3$, then   $Y_3$ can be either on
$A_2A'_3$ or on $A_3A'_3$ as long as $X_2Y_2$ and $X_3Y_3$ do  not intersect.
Similar conditions hold  for $Y_1$ and $Z_1$.

In any case, $X_2Y_2$ crosses the diagonal of $A_2A_3$ at the point $P_2$
and $X_3Y_3$ intersects $A_2A_3$ at $Q_2$.
On the facet  $A_4A_3A'_1A_1$ the segment $X_3Z_3$ intersects $A_1A_3$ at $P_1$ 
and $X_1Y_1$ intersects $A_1A_3$ at $Q_1$.
There is no other points of the geodesic between $Q_i$ and $P_i$, $i=1,2$.

Consider  the domain of the cube outside the triangles $A_4A_iA_{i+1}$,
$i=1,2,3$, when $i+1>3$ then take $i+1\mod 4$.
Consider the segment $\gamma_i$ of the geodesic $\gamma$ coming out of $P_i$, $i=1,2$. 
Since the domain $D_i$ is homeomorphic to a disc, 
then it's easy to show, that $\gamma_i$ comes to the point $Q_i$,
and does not have other points $P_j$, $Q_j$, $j\neq i$ on it. 
Denote by $\sigma_i$ the parts of the diagonals 
of the facets between the points $Q_i$ and $P_j$.
The segments $\gamma_i$ and $\sigma_i$ enclose a domain $G_i$ inside $D_1$.
Each $G_i$ is locally convex and homeomorphic to a disc, $i=1,2$. \\

Let $l=2$ or $3$.
In both Cases 1 and 2 we obtain $l$ domains $G_i$ 
bounded by the segment of $\gamma_i$
and the segment  $\sigma_i$ of the diagonal of the facet of the cube.
On the cube these domains  do not intersect. 
If $G_i$ does not contain any vertex of the cube
then $G_i$ is locally isometric to a unite sphere and
 $G_i$ form a lune on a unite sphere. 
 Hence the length of $\sigma_i$ is equal to $\pi$. 
Since $\sigma_i$ is part of the diagonal of the cube's facet, 
then from (\ref{diagonal}) follows the length of $\sigma_i$ is less than $\pi$.
We get a contradiction.

Therefore each $G_i$ should have at least one cube's vertex. 
If we are in the Case 1 and there are three domains $G_i$ around $A_4$,
then each domain should have exactly one vertex $A_{s_i}$ of the cube.
Moreover, geodesic should intersect all edges coming out of $A_{s_i}$.
Otherwise $D_1$ has more then four vertices, 
that contradicts to the assumption. 
Since domains $G_i$ do not intersect, then all $A_{s_i}$ are different 
cube's vertices.

If we are in the Case 2, then there are only two domains $G_1$ and $G_2$.
Since $D_1$ has at most four vertices of the cube,
then $G_1$ and $G_2$ together have up to three cube's vertices.
Hence one of the domains, for example $G_1$,
has exactly one vertex of a cube  $A_{s_1}$, and $A_{s_1}$ 
does not lie at $G_2$.
In this case geodesic also intersects all edges coming out of $A_{s_i}$.

In both cases let us apply again the same construction 
to the vertex $A_{s_1}$  inside $G_1$. 
The domain  $G_1$ is divided  
into  $l+1$  parts $G'_k$, $k=1,\dots,l$ and 
$G_1\backslash (\cup_{k=1}^{l} G'_k)$ between them
with $l$ segments  $\sigma'_1$ and $\sigma'_2$ of the diagonals.
The domain $G_1\backslash (\cup_{k=1}^{l} G'_k)$ contains
the vertex $A_{s_1}$ and 
$G'_k$ do not intersect and do not have any vertex of the cube.
Since there is no points of the geodesic on $\sigma_1$,
then $\sigma_1$ belong to the boundary 
of one of the domains $G'_k$, for example to $G'_2$.
Then the perimeter of  $G'_1$ is a subarc $\gamma'_1$ of $\gamma_1$ and 
the subarc $\sigma'_1$ of the diagonal of the cube's facet.
Since $G'_1$ does not have any vertex of the cube, then 
$G'_1$ form a lune on a unite sphere
that again leads to a contradiction. 
\end{proof}

In what following we  label the front facet of the cube as $A_1A_2A_3A_4$ and 
 the back one as $A'_1A'_2A'_3A'_4$.

\begin{lemma}\label{cube_lemma}
A domain $D_i$, $i=1,2$ contains four vertices of the cube
connected with at least three edges. 
\end{lemma}
\begin{proof}
By Lemma~\ref{edge_inside_cube} the domain $D_i$ contains the vertex of a cube 
together with at least one edge coming out of this vertex. 

\textbf{Case 1.}
Assume first that the domain $D_1$ has only one edge, for instance $A_1A_4$.
Then geodesic intersects the edges $A_1A'_1$, $A_4A'_4$, $A_4A_3$ and $A_1A_2$
at the points $X_1$, $X_2$, $X_3$ and $X_4$ respectively.

Rotation $r_0$ on the angle $\pi$ along the line passing through the
midpoints of the edges $A_1A_4$ and $A'_2A'_3$ maps $\gamma$ into 
a geodesic $\Tilde{\gamma}$. 
Since the edge $A_1A_4$ maps into itself,
and $r_0(A_1A'_1)=A_3A_4$, $r_0(A_1A_2) = A_4A'_4$,
then $\Tilde{\gamma}$ is equivalent to $\gamma$.
By Lemma~\ref{uniqness} $\gamma$ and $\Tilde{\gamma}$ coincide. 
Hence $\dist(A_1, X_4)=\dist(A_4, X_2)$ and $\dist(A_1, X_1)=\dist(A_4, X_3)$.

 \begin{minipage}{0.4\linewidth}
    \begin{figure}[H]
        \begin{tikzpicture}[line cap=round,line join=round,x=0.7cm,y=0.7cm]
\clip(-1,-0.7) rectangle (10,7);

\draw [shift={(5,2)},line width=0.5pt]  plot[domain=2.677945044588987:3.6052402625905993,variable=\t]({1*4.47213595499958*cos(\t r)+0*4.47213595499958*sin(\t r)},{0*4.47213595499958*cos(\t r)+1*4.47213595499958*sin(\t r)});
\draw [shift={(-2,2)},line width=0.5pt]  plot[domain=-0.27829965900511144:0.27829965900511133,variable=\t]({1*7.280109889280519*cos(\t r)+0*7.280109889280519*sin(\t r)},{0*7.280109889280519*cos(\t r)+1*7.280109889280519*sin(\t r)});
\draw [shift={(3,-4)},line width=0.5pt]  plot[domain=1.3258176636680326:1.8157749899217608,variable=\t]({1*8.246211251235321*cos(\t r)+0*8.246211251235321*sin(\t r)},{0*8.246211251235321*cos(\t r)+1*8.246211251235321*sin(\t r)});
\draw [shift={(3,8)},line width=0.5pt]  plot[domain=4.4674103172578254:4.957367643511554,variable=\t]({1*8.246211251235321*cos(\t r)+0*8.246211251235321*sin(\t r)},{0*8.246211251235321*cos(\t r)+1*8.246211251235321*sin(\t r)});
\draw [shift={(4,3)},line width=0.5pt]  plot[domain=1.892546881191539:2.819842099193151,variable=\t]({1*3.1622776601683795*cos(\t r)+0*3.1622776601683795*sin(\t r)},{0*3.1622776601683795*cos(\t r)+1*3.1622776601683795*sin(\t r)});
\draw [shift={(9.130809898184463,1.8950343697549739)},line width=0.5pt]  plot[domain=2.049591997369353:2.6703128472203277,variable=\t]({1*4.6250506643380485*cos(\t r)+0*4.6250506643380485*sin(\t r)},{0*4.6250506643380485*cos(\t r)+1*4.6250506643380485*sin(\t r)});
\draw [shift={(3,4)},line width=0.5pt]  plot[domain=5.176036589385496:5.81953769817878,variable=\t]({1*4.47213595499958*cos(\t r)+0*4.47213595499958*sin(\t r)},{0*4.47213595499958*cos(\t r)+1*4.47213595499958*sin(\t r)});
\draw [shift={(3,4)},line width=0.5pt]  plot[domain=-0.46364760900080615:0.4636476090008061,variable=\t]({1*4.47213595499958*cos(\t r)+0*4.47213595499958*sin(\t r)},{0*4.47213595499958*cos(\t r)+1*4.47213595499958*sin(\t r)});
\draw [shift={(4.289620338326439,-1.2517388441527422)},line width=0.5pt,dotted]  plot[domain=1.9483631122020275:2.7779986343945624,variable=\t]({1*3.498131805349394*cos(\t r)+0*3.498131805349394*sin(\t r)},{0*3.498131805349394*cos(\t r)+1*3.498131805349394*sin(\t r)});
\draw [shift={(5,-4)},line width=0.5pt,dotted]  plot[domain=1.2490457723982544:1.892546881191539,variable=\t]({1*6.324555320336759*cos(\t r)+0*6.324555320336759*sin(\t r)},{0*6.324555320336759*cos(\t r)+1*6.324555320336759*sin(\t r)});
\draw [shift={(8,4)},line width=0.5pt,dotted]  plot[domain=2.761086276477428:3.522099030702158,variable=\t]({1*5.385164807134505*cos(\t r)+0*5.385164807134505*sin(\t r)},{0*5.385164807134505*cos(\t r)+1*5.385164807134505*sin(\t r)});
\draw [shift={(4.98,0.87)},line width=0.5pt]  plot[domain=1.1956789173869815:1.9391452482038598,variable=\t]({1*5.513374647164838*cos(\t r)+0*5.513374647164838*sin(\t r)},{0*5.513374647164838*cos(\t r)+1*5.513374647164838*sin(\t r)});

\draw [shift={(6.0960685469709635,1.723202058592476)},line width=0.5pt]  plot[domain=1.820543495446873:2.183527339436593,variable=\t]({1*2.9509046702148245*cos(\t r)+0*2.9509046702148245*sin(\t r)},{0*2.9509046702148245*cos(\t r)+1*2.9509046702148245*sin(\t r)});
\draw [shift={(0.258266793241362,2.187572652639149)},line width=0.5pt]  plot[domain=-0.3277963434058764:0.4383994466587939,variable=\t]({1*5.651871847040474*cos(\t r)+0*5.651871847040474*sin(\t r)},{0*5.651871847040474*cos(\t r)+1*5.651871847040474*sin(\t r)});
\draw [shift={(13.426490067279156,1.9056333633965261)},line width=0.5pt]  plot[domain=2.9002883701243602:3.3626821339243267,variable=\t]({1*9.288115293779883*cos(\t r)+0*9.288115293779883*sin(\t r)},{0*9.288115293779883*cos(\t r)+1*9.288115293779883*sin(\t r)});
\draw [shift={(-0.28902712117078816,13.050527620516677)},line width=0.5pt,dotted]  plot[domain=5.050061747528929:5.147699150766648,variable=\t]({1*13.976203033001996*cos(\t r)+0*13.976203033001996*sin(\t r)},{0*13.976203033001996*cos(\t r)+1*13.976203033001996*sin(\t r)});

\draw [shift={(8.600352822008377,3.7465310755101253)},line width=0.5pt,dotted]  plot[domain=2.8591918555817566:3.467234994493685,variable=\t]({1*6.665714654387071*cos(\t r)+0*6.665714654387071*sin(\t r)},{0*6.665714654387071*cos(\t r)+1*6.665714654387071*sin(\t r)});
\draw [shift={(12.116301605504615,4.3269943180684685)},line width=0.5pt,dotted]  plot[domain=2.916794720284866:3.3907839243522626,variable=\t]({1*8.592005155909213*cos(\t r)+0*8.592005155909213*sin(\t r)},{0*8.592005155909213*cos(\t r)+1*8.592005155909213*sin(\t r)});
\draw [shift={(0.39094410582612565,9.302393539997103)},line width=0.5pt,dotted]  plot[domain=4.95357780822651:5.15836436709901,variable=\t]({1*7.9184313694159085*cos(\t r)+0*7.9184313694159085*sin(\t r)},{0*7.9184313694159085*cos(\t r)+1*7.9184313694159085*sin(\t r)});
\draw [shift={(-0.8694903637291297,15.438719247042423)},line width=0.5pt]  plot[domain=5.014793679777069:5.177056078032471,variable=\t]({1*10.302187213244611*cos(\t r)+0*10.302187213244611*sin(\t r)},{0*10.302187213244611*cos(\t r)+1*10.302187213244611*sin(\t r)});

\begin{scriptsize}
\draw [fill=black] (5,4) circle (1.5pt);
\draw[color=black] (5.3,4.0) node {$A_1$};
\draw [fill=black] (1,4) circle (1.5pt);
\draw[color=black] (0.65,4.1) node {$A_2$};
\draw [fill=black] (1,0) circle (1.5pt);
\draw[color=black] (0.8,-0.2) node {$A_3$};
\draw [fill=black] (5,0) circle (1.5pt);
\draw[color=black] (5.1,-0.3) node {$A_4$};

\draw [fill=black] (7,6) circle (1.5pt);
\draw[color=black] (7.35,6.1) node {$A'_1$};
\draw [fill=black] (3,6) circle (1.5pt);
\draw[color=black] (2.9,6.35) node {$A'_2$};
\draw [fill=black] (3,2) circle (1.5pt);
\draw[color=black] (2.6,2.25) node {$A'_3$};
\draw [fill=black] (7,2) circle (1.5pt);
\draw[color=black] (7.35,2.0) node {$A'_4$};

\draw [fill=black] (5.366725958851461,4.582554740905522) circle (1.5pt);
\draw[color=black] (5.2,4.9) node {$X_1$};
\draw [fill=black] (5.609199781045852,0.36791017421233363) circle (1.5pt);
\draw[color=black] (5.9,0.2) node {$X_2$};
\draw [fill=black] (4.34117966937726,-0.13641426516921484) circle (1.5pt);
\draw[color=black] (4.3,-0.4) node {$X_3$};
\draw [fill=black] (4.407478249930408,4.125207995859112) circle (1.5pt);
\draw[color=black] (4.3,4.5) node {$X_4$};

\draw [fill=black] (2.198672489793531,5.604013451696816) circle (1.5pt);
\draw [fill=black] (3.7404784346948423,6.242232896025958) circle (1.5pt);
\draw [fill=black] (2.282318544629908,1.613163534047901) circle (1.5pt);
\draw [fill=black] (3.7835158436577667,2.2064616568032003) circle (1.5pt);
\end{scriptsize}
\end{tikzpicture}
        \caption{\;}
        \label{cube_geod_1_edge}
    \end{figure}
 \end{minipage}
      \hspace{0.05\linewidth}
 \begin{minipage}{0.4\linewidth}
     \begin{figure}[H]
        \begin{tikzpicture}[line cap=round,line join=round,x=0.7cm,y=0.7cm]
\clip(0,-1.5) rectangle (10,7.5);

\draw [shift={(5,2)},line width=0.5pt]  plot[domain=2.677945044588987:3.6052402625905993,variable=\t]({1*4.47213595499958*cos(\t r)+0*4.47213595499958*sin(\t r)},{0*4.47213595499958*cos(\t r)+1*4.47213595499958*sin(\t r)});
\draw [shift={(-2,2)},line width=0.5pt]  plot[domain=-0.27829965900511144:0.27829965900511133,variable=\t]({1*7.280109889280519*cos(\t r)+0*7.280109889280519*sin(\t r)},{0*7.280109889280519*cos(\t r)+1*7.280109889280519*sin(\t r)});
\draw [shift={(3,-4)},line width=0.5pt]  plot[domain=1.3258176636680326:1.8157749899217608,variable=\t]({1*8.246211251235321*cos(\t r)+0*8.246211251235321*sin(\t r)},{0*8.246211251235321*cos(\t r)+1*8.246211251235321*sin(\t r)});
\draw [shift={(3,8)},line width=0.5pt]  plot[domain=4.4674103172578254:4.957367643511554,variable=\t]({1*8.246211251235321*cos(\t r)+0*8.246211251235321*sin(\t r)},{0*8.246211251235321*cos(\t r)+1*8.246211251235321*sin(\t r)});
\draw [shift={(4,3)},line width=0.5pt]  plot[domain=1.892546881191539:2.819842099193151,variable=\t]({1*3.1622776601683795*cos(\t r)+0*3.1622776601683795*sin(\t r)},{0*3.1622776601683795*cos(\t r)+1*3.1622776601683795*sin(\t r)});
\draw [shift={(9.130809898184463,1.8950343697549739)},line width=0.5pt]  plot[domain=2.049591997369353:2.6703128472203277,variable=\t]({1*4.6250506643380485*cos(\t r)+0*4.6250506643380485*sin(\t r)},{0*4.6250506643380485*cos(\t r)+1*4.6250506643380485*sin(\t r)});
\draw [shift={(3,4)},line width=0.5pt]  plot[domain=5.176036589385496:5.81953769817878,variable=\t]({1*4.47213595499958*cos(\t r)+0*4.47213595499958*sin(\t r)},{0*4.47213595499958*cos(\t r)+1*4.47213595499958*sin(\t r)});
\draw [shift={(3,4)},line width=0.5pt]  plot[domain=-0.46364760900080615:0.4636476090008061,variable=\t]({1*4.47213595499958*cos(\t r)+0*4.47213595499958*sin(\t r)},{0*4.47213595499958*cos(\t r)+1*4.47213595499958*sin(\t r)});
\draw [shift={(4.289620338326439,-1.2517388441527422)},line width=0.5pt,dash pattern=on 1pt off 1pt]  plot[domain=1.9483631122020275:2.7779986343945624,variable=\t]({1*3.498131805349394*cos(\t r)+0*3.498131805349394*sin(\t r)},{0*3.498131805349394*cos(\t r)+1*3.498131805349394*sin(\t r)});
\draw [shift={(5,-4)},line width=0.5pt,dash pattern=on 1pt off 1pt]  plot[domain=1.2490457723982544:1.892546881191539,variable=\t]({1*6.324555320336759*cos(\t r)+0*6.324555320336759*sin(\t r)},{0*6.324555320336759*cos(\t r)+1*6.324555320336759*sin(\t r)});
\draw [shift={(8,4)},line width=0.5pt,dash pattern=on 1pt off 1pt]  plot[domain=2.761086276477428:3.522099030702158,variable=\t]({1*5.385164807134505*cos(\t r)+0*5.385164807134505*sin(\t r)},{0*5.385164807134505*cos(\t r)+1*5.385164807134505*sin(\t r)});
\draw [shift={(4.98,0.87)},line width=0.5pt]  plot[domain=1.1956789173869815:1.9391452482038598,variable=\t]({1*5.513374647164838*cos(\t r)+0*5.513374647164838*sin(\t r)},{0*5.513374647164838*cos(\t r)+1*5.513374647164838*sin(\t r)});

\draw [shift={(3.2626301530298987,-3.7699363622007547)},line width=0.5pt]  plot[domain=1.3243310373201933:1.8031497755867771,variable=\t]({1*8.6065568953819*cos(\t r)+0*8.6065568953819*sin(\t r)},{0*8.6065568953819*cos(\t r)+1*8.6065568953819*sin(\t r)});
\draw [shift={(-0.9940692726172548,2.2031741221750742)},line width=0.5pt]  plot[domain=-0.27968857107729495:0.3573456838281194,variable=\t]({1*6.805990593324959*cos(\t r)+0*6.805990593324959*sin(\t r)},{0*6.805990593324959*cos(\t r)+1*6.805990593324959*sin(\t r)});
\draw [shift={(14.470794366528253,0.4695989528591009)},line width=0.5pt,dash pattern=on 1pt off 1pt]  plot[domain=2.8373400416768675:3.100582711379334,variable=\t]({1*13.826538046209452*cos(\t r)+0*13.826538046209452*sin(\t r)},{0*13.826538046209452*cos(\t r)+1*13.826538046209452*sin(\t r)});
\draw [shift={(2.0954706330943873,10.338962540549048)},line width=0.5pt,dash pattern=on 1pt off 1pt]  plot[domain=4.7397674520491515:5.044328926041472,variable=\t]({1*10.566147411201364*cos(\t r)+0*10.566147411201364*sin(\t r)},{0*10.566147411201364*cos(\t r)+1*10.566147411201364*sin(\t r)});
\draw [shift={(3.4857635906646283,3.078543762126704)},line width=0.5pt]  plot[domain=3.762587576332066:4.390514534116842,variable=\t]({1*3.5080266211266506*cos(\t r)+0*3.5080266211266506*sin(\t r)},{0*3.5080266211266506*cos(\t r)+1*3.5080266211266506*sin(\t r)});

\draw [shift={(8.068581117470234,3.8337646279673265)},line width=0.5pt,dotted]  plot[domain=2.848607700610918:3.5039261709082856,variable=\t]({1*6.125462656024561*cos(\t r)+0*6.125462656024561*sin(\t r)},{0*6.125462656024561*cos(\t r)+1*6.125462656024561*sin(\t r)});
\draw [shift={(4.97904121175859,10.201649655850753)},line width=0.5pt,dotted]  plot[domain=4.41582676347546:4.902618646642531,variable=\t]({1*8.915923125029972*cos(\t r)+0*8.915923125029972*sin(\t r)},{0*8.915923125029972*cos(\t r)+1*8.915923125029972*sin(\t r)});
\draw [shift={(-2.229885234901915,5.3956986914104315)},line width=0.5pt]  plot[domain=5.864709046513695:6.172136634165562,variable=\t]({1*9.73637770798381*cos(\t r)+0*9.73637770798381*sin(\t r)},{0*9.73637770798381*cos(\t r)+1*9.73637770798381*sin(\t r)});
\draw [shift={(5.888739072884796,-2.8087461693126903)},line width=0.5pt]  plot[domain=1.5967833403254617:1.983656667654308,variable=\t]({1*9.15424304828646*cos(\t r)+0*9.15424304828646*sin(\t r)},{0*9.15424304828646*cos(\t r)+1*9.15424304828646*sin(\t r)});
\draw [shift={(3.5715841436010627,2.5979486656826722)},line width=0.5pt,dotted]  plot[domain=0.41575323842441336:1.0638945343409396,variable=\t]({1*4.25162369782233*cos(\t r)+0*4.25162369782233*sin(\t r)},{0*4.25162369782233*cos(\t r)+1*4.25162369782233*sin(\t r)});

\begin{scriptsize}
\draw [fill=black] (1,0) circle (1.5pt);
\draw[color=black] (0.7,-0.14) node {$A_3$};
\draw [fill=black] (5,0) circle (1.5pt);
\draw[color=black] (5.2,-0.3) node {$A_4$};
\draw [fill=black] (1,4) circle (1.5pt);
\draw[color=black] (0.5,4.1) node {$A_2$};
\draw [fill=black] (5,4) circle (1.5pt);
\draw[color=black] (5.3,4.07) node {$A_1$};

\draw [fill=black] (3,6) circle (1.5pt);
\draw[color=black] (3.1,6.3) node {$A'_2$};
\draw [fill=black] (7,6) circle (1.5pt);
\draw[color=black] (7.3,6.2) node {$A'_1$};
\draw [fill=black] (7,2) circle (1.5pt);
\draw[color=black] (7.4,2) node {$A'_4$};
\draw [fill=black] (3,2) circle (1.5pt);
\draw[color=black] (3.2,1.8) node {$A'_3$};

\draw [fill=black] (5.547450441530492,0.32433730492770785) circle (1.5pt);
\draw[color=black] (6.1,0.3) node {$X_5$};
\draw [fill=black] (5.362437213325659,4.576537857670639) circle (1.5pt);
\draw[color=black] (5.7,4.6) node {$X_1$};
\draw [fill=black] (1.2792937037686034,4.611755952260403) circle (1.5pt);
\draw[color=black] (0.8,4.8) node {$X_2$};
\draw [fill=black] (0.6326856489941486,1.0374173492637961) circle (1.5pt);
\draw[color=black] (0.2,1) node {$X_3$};
\draw [fill=black] (2.3847194615437193,-0.22322502787057097) circle (1.5pt);
\draw[color=black] (2.3,-0.6) node {$X_4$};

\draw [fill=black] (2.2041485969838095,5.60286721487762) circle (1.5pt);
\draw[color=black] (2,6) node {$Y_2$};
\draw [fill=black] (2.3735033073376064,1.6749353179554267) circle (1.5pt);
\draw[color=black] (2.45,1.3) node {$Y_1$};
\draw [fill=black] (6.666332658536589,1.4391398247955767) circle (1.5pt);
\draw[color=black] (7.1,1.35) node {$Y_5$};
\draw [fill=black] (7.461022064221398,4.315090689373487) circle (1.5pt);
\draw[color=black] (7.9,4.4) node {$Y_4$};
\draw [fill=black] (5.650874409690287,6.342406008916253) circle (1.5pt);
\draw[color=black] (5.837246741122936,6.65) node {$Y_3$};
\end{scriptsize}
\end{tikzpicture}
 
        \caption{\;}
        \label{cube_geod_2_edge}
    \end{figure}
 \end{minipage}

Apply an isometry $r$ of a cube such that $r(A_1A_4)=A'_2A'_3$.
The action $r$ is a rotation on the angle $\pi$ 
along the line passing through the centers of the facets 
$A'_1A_1A_2A'_2$ and $A'_4A_4A_3A'_3$.
The image of a geodesic $\gamma$ is a geodesic $\sigma$ that enclose a domain
containing the edge $A'_2A'_3$ (see Figure~\ref{cube_geod_1_edge}).
One can easily check that $\sigma$ is also fixed by $r_0$ action.

If $\gamma$ and $\sigma$ intersect, then they intersect inside the facets 
$A'_1A_1A_2A'_2$ and $A'_4A_4A_3A'_3$.
But then it would mean that the length of $\gamma$ is greater $2\pi$ 
that contradicts  Lemma~\ref{length}.
Thus $\gamma$ and $\sigma$ do not intersect. 
In this case from the Gauss-Bonnet theorem follows that the integral of the curvature 
over the area of a  domain enclosed by $\gamma$ and $\sigma$ equals zero.
This leads to a contradiction.

\textbf{Case 2.} Assume now the domain $D_1$ contains three vertices of a cube.
Since each vertex lies in the domain with at least one edge
coming out of it, then  this vertices belong to the boundary of some facet 
and connected by two edges. 
For example $D_1$ contains $A_1A_2$ and $A_1A_4$. 
Denote with $X_1, \dots, X_5$ the vertices of the geodesic $\gamma$ on 
the edges $A_1A'_1$, $A_2A'_2$, $A_2A_3$, $A_3A_4$ and $A_4A'_4$ respectively. 

A reflection $s$ through the plane $A_1A'_1A'_3A_3$ maps the geodesic~$\gamma$ 
into a geodesic $\Tilde{\gamma}$. 
Since $s$ fixes the edges $A_1A'_1$ and $A_3A'_3$ and $s(A_2A'_2)=A_4A'_4$,
then $\gamma$ and $\Tilde{\gamma}$ are equivalent geodesics
and hence by Lemma~\ref{uniqness} they coincide. 
It follows 
\begin{equation}\label{distance}
\dist(A_4,X_5) = \dist(A_2,X_2)
\end{equation}
and 
$\angle(X_5X_1A_1)=\angle(X_2X_1A_1)$. 
Since $\gamma$ is a geodesic line, then  $\gamma$ intersects an edge $A_1A'_1$
orthogonal.

Consider two isometries of the cube, namely,
the rotation $r_1$  on an angle $\pi$ along the 
line passing through the midpoints of the edges $A_1A'_1$ and $A_3A'_3$
and the rotation $r_2$ on the angle $\pi$ along the line
passing through the centers of the facets $A_1A_2A_3A_4$ and $A'_1A'_2A'_3A'_4$.
An isometry $r_1 \circ r_2$ maps $\gamma$ into a geodesic $\sigma$ with the 
vertices $Y_k=(r_1\circ r_2)(X_k)$, $k=1, \dots, 5$ on the edges 
$A_3A'_3$, $A_2A'_2$, $A'_1A'_2$, $A'_1A'_4$ and $A_4A'_4$ respectively
(see Figure~\ref{cube_geod_2_edge}.).
If $\gamma$ and $\sigma$ do not intersect, then by the Gauss-Bonnet theorem
the integral of the curvature over the area of a  domain 
enclosed by $\gamma$ and $\sigma$ should equal zero.
This leads to a contradiction.
Hence $\gamma$ and $\sigma$ intersect.

The intersection is possible when $\dist(A'_2, X_2) \leq \dist(A_2,X_2)$. 
If $\dist(A'_2, X_2) < \dist(A_2,X_2)$, 
then $\gamma$ and $\sigma$ intersect at the points $Z_1$ and $Z_2$ 
on the facets $A_1A_4A'_4A'_1$ and $A_3A_4A'_4A'_3$ respectively. 
Thus the length of the segment $Z_1X_5Z_2$ of $\gamma$ equals $\pi$. 
From~(\ref{distance}) follows that
$\gamma$ and $\sigma$ also intersect at the points $Z_3$ and $Z_4$ 
on the facets $A_1A_2A'_2A'_1$ and $A_3A_2A'_2A'_3$ respectively. 
Then the length of the segment $Z_3X_2Z_4$ of $\gamma$ is also equal to $\pi$.
Hence the length of $\gamma$ is greater then $2\pi$ that leads to a contradiction.

 \begin{figure}[H]
 \centering
    \begin{tikzpicture}[line cap=round,line join=round,x=2cm,y=2cm]

\clip(2,-1.2) rectangle (8,1.5);

\draw[line width=0.5pt,fill=black,fill opacity=0.1] (5,-0.1159352657910958) -- (5.115935265791096,-0.11593526579109582) -- (5.115935265791096,0) -- (5,0) -- cycle; 

\draw[line width=0.5pt,fill=black,fill opacity=0.1] (5,-0.5969365940220236) -- (5.115935265791096,-0.5969365940220237) -- (5.115935265791096,-0.48100132823092784) -- (5,-0.48100132823092784) -- cycle; 

\draw [shift={(7.121666574555207,0)},line width=0.8pt]  plot[domain=2.9035715551167094:3.379613752062877,variable=\t]({1*4.241242194429087*cos(\t r)+0*4.241242194429087*sin(\t r)},{0*4.241242194429087*cos(\t r)+1*4.241242194429087*sin(\t r)});
\draw [shift={(4,-3)},line width=0.8pt]  plot[domain=1.3258176636680326:1.8157749899217608,variable=\t]({1*4.123105625617661*cos(\t r)+0*4.123105625617661*sin(\t r)},{0*4.123105625617661*cos(\t r)+1*4.123105625617661*sin(\t r)});
\draw [shift={(6,-3)},line width=0.8pt]  plot[domain=1.3258176636680326:1.8157749899217608,variable=\t]({1*4.123105625617661*cos(\t r)+0*4.123105625617661*sin(\t r)},{0*4.123105625617661*cos(\t r)+1*4.123105625617661*sin(\t r)});
\draw [shift={(4,3)},line width=0.8pt]  plot[domain=4.4674103172578254:4.957367643511554,variable=\t]({1*4.123105625617661*cos(\t r)+0*4.123105625617661*sin(\t r)},{0*4.123105625617661*cos(\t r)+1*4.123105625617661*sin(\t r)});
\draw [shift={(6,3)},line width=0.8pt]  plot[domain=4.4674103172578254:4.957367643511554,variable=\t]({1*4.123105625617661*cos(\t r)+0*4.123105625617661*sin(\t r)},{0*4.123105625617661*cos(\t r)+1*4.123105625617661*sin(\t r)});

\draw [shift={(3,0)},line width=0.8pt]  plot[domain=-0.24497866312686423:0.24497866312686414,variable=\t]({1*4.123105625617661*cos(\t r)+0*4.123105625617661*sin(\t r)},{0*4.123105625617661*cos(\t r)+1*4.123105625617661*sin(\t r)});
\draw [shift={(5,4.5)},line width=0.8pt]  plot[domain=4.2728407396702295:5.152962592611362,variable=\t]({1*4.981001328230928*cos(\t r)+0*4.981001328230928*sin(\t r)},{0*4.981001328230928*cos(\t r)+1*4.981001328230928*sin(\t r)});

\draw [line width=0.8pt] (5,1)-- (5,-1);
\draw [line width=0.8pt] (2.880431059684847,-0.00752723109795116)-- (5,0);
\draw [line width=0.8pt] (5,0)-- (7.121666574555207,0);

\begin{scriptsize}
\draw [fill=black] (5,-1) circle (1.5pt);
\draw[color=black] (5.1,-0.9) node {$A_1$};
\draw [fill=black] (5,1) circle (1.5pt);
\draw[color=black] (5.05,1.15) node {$A'_1$};
\draw [fill=black] (3,-1) circle (1.5pt);
\draw[color=black] (3.1,-0.9) node {$A_2$};
\draw [fill=black] (3,1) circle (1.5pt);
\draw[color=black] (3.1,1.15) node {$A'_2$};
\draw [fill=black] (7,-1) circle (1.5pt);
\draw[color=black] (7.15,-0.95) node {$A_4$};
\draw [fill=black] (7,1) circle (1.5pt);
\draw[color=black] (7.1,1.1) node {$A'_4$};

\draw [fill=black] (5,0) circle (1.5pt);
\draw[color=black] (5.1,0.1) node {$M$};

\draw [fill=black] (2.880431059684847,-0.00752723109795116) circle (1.5pt);
\draw[color=black] (3.0,0.1) node {$X_2$};
\draw [fill=black] (5,-0.48100132823092784) circle (1.5pt);
\draw[color=black] (5.1,-0.4) node {$X_1$};
\draw [fill=black] (7.121666574555207,0) circle (1.5pt);
\draw[color=black] (7.25,0.1) node {$X_5$};
\end{scriptsize}
\end{tikzpicture}
    \caption{\;}
    \label{two_square}
\end{figure}

If $\dist(A'_4, X_5) = \dist(A_4, X_5)$, 
then from~(\ref{distance}) follows $\dist(A'_2, X_2) = \dist(A_2, X_2)$ as well.
Consider the geodesic segment $X_2MX_5$ where $M$ is a middle point of the edge $A_1A'_1$. 
The segments $X_2X_1X_5$ and $X_2MX_5$ form a geodesic lune on a sphere,
thus the length of the $X_2MX_5$ equals $\pi$.
From the other hand the length of $X_2MX_5$ equal $2h$,
where $h$ is the length of arc, connecting the midpoints of opposite edges of 
a spherical square. 
To find $h$ consider a quadrilateral $A_1A_2X_2M$. 
From the triangle $A_1A_2X_2$ using (\ref{cos_rule}) we have
$$ \cos A_1X_2 = \cos a_s \, \cos (a_s/2) + \sin a_s\, \sin (a_s/2) \, \cos \alpha. $$
From the triangle $A_1X_2M$ follows 
$$ \cos A_1X_2 = \cos  (a_s/2) \, \cos h. $$
Combining this two formulas we obtain
\begin{equation}\label{vspom}
   \cos h = \cos a_s + 2 \sin^2  (a_s/2) \, \cos \alpha.
\end{equation} 
From~\ref{a_cube} one can compute
$$ \sin  (a_s/2)  = \frac{\sqrt{- \cos \alpha}}{\sqrt{2}\sin(\alpha/2)}. $$
Implementing it into (\ref{vspom}) we have
$$ \cos h = \cot^2 (\alpha/2)-  \frac{ \cos^2 \alpha }{\sin^2(\alpha/2)}  = $$
$$   = \frac{\cos^2(\alpha/2) - \cos^2\alpha}{\sin^2(\alpha/2)} 
= \frac{\cos \alpha - \cos 2\alpha}{2 \sin^2(\alpha/2)} 
= \frac{2\sin (\alpha/2) \sin (3\alpha/2)}{2 \sin^2(\alpha/2)}$$   
Thus
\begin{equation}\label{h}
   h= \arccos \left( \frac{\sin (3\alpha/2)}{\sin (\alpha/2)} \right).
\end{equation}
From (\ref{h}) follows that $X_2MX_5$ equals $\pi$ if and only if $\alpha=2\pi/3$
that leads to a contradiction. 

\textbf{Case 3.}
Assume  the domain  $D_1$ contains four vertices of a cube
and there are only two edges inside $D_i$. 
Remind, that by Lemma~\ref{edge_inside_cube}
$D_1$ can not have a vertex  without an edge coming from it. 
Since $D_1$ has two edges and four vertices, thus this edges does not share a vertex. 
Assume, for example, $D_1$ has $A_1A'_1$ and any other edge of a cube,
that does not come from the edge $A_1$ or $A'_1$.
Since $\gamma$ bound an area, containing $A_1A'_1$, then there is a segment of
$\gamma$ that intersects edges 
$A_1A_2$, $A'_1A'_2$, $A'_1A'_4$ and $A_1A_4$ consequently at the points 
$X_1$, $X_2$, $X_3$ and $X_4$.
Note, that the segment $X_1X_2X_3X_4$ is just a part of the geodesic. 
The points $X_1$ and $X_4$ are not connected on the facet $A_1A_2A_3A_4$ 
(see Figure~\ref{cube_2edge_4vert}).

Consider the development of the facets $A_1A_2A'_2A'_1$, $A'_1A'_2A'_3A'_4$ and 
$A'_1A'_4A_4A_1$ to a  unite sphere and take the diagonals $A'_1A_2$ and $A'_1A_4$.
Then $A'_1A_2$ and $X_1X_2$ intersect at $Z_1$ and $A'_1A_4$ intersect at $Z_2$
(see Figure~\ref{cube_2edge_4vert_dev}).
The angle $\angle Z_1A'_1Z_2=2\alpha > \pi$. 
It means, that the shortest path  connecting $Z_1$ and $Z_2$ lies 
from the opposite side of the vertex $A'_1$ on a sphere,
and thus the segment $\gamma_1$ of the geodesic  $\gamma$
between the points $Z_1$ and $Z_2$ does not lie inside the development.
Unless the length of $A'_1Z_1$ and the length of $A'_1Z_2$ are equal to $\pi$
and the length of $Z_1Z_2$ is greater than $\pi$. 
From (\ref{diagonal}) follows 
the length of $Y'_1Z_k$ is less than $\pi$ for all $\alpha \in (\pi/3, 2\pi/3)$.
This leads to a contradiction.

 \begin{minipage}{0.4\linewidth}
    \begin{figure}[H]
        \begin{tikzpicture}[line cap=round,line join=round,x=0.7cm,y=0.7cm]
\clip(-1,-1) rectangle (10,7);

\draw [shift={(5,2)},line width=0.5pt]  plot[domain=2.677945044588987:3.6052402625905993,variable=\t]({1*4.47213595499958*cos(\t r)+0*4.47213595499958*sin(\t r)},{0*4.47213595499958*cos(\t r)+1*4.47213595499958*sin(\t r)});
\draw [shift={(-2,2)},line width=0.5pt]  plot[domain=-0.27829965900511144:0.27829965900511133,variable=\t]({1*7.280109889280519*cos(\t r)+0*7.280109889280519*sin(\t r)},{0*7.280109889280519*cos(\t r)+1*7.280109889280519*sin(\t r)});
\draw [shift={(3,-4)},line width=0.5pt]  plot[domain=1.3258176636680326:1.8157749899217608,variable=\t]({1*8.246211251235321*cos(\t r)+0*8.246211251235321*sin(\t r)},{0*8.246211251235321*cos(\t r)+1*8.246211251235321*sin(\t r)});
\draw [shift={(3,8)},line width=0.5pt]  plot[domain=4.4674103172578254:4.957367643511554,variable=\t]({1*8.246211251235321*cos(\t r)+0*8.246211251235321*sin(\t r)},{0*8.246211251235321*cos(\t r)+1*8.246211251235321*sin(\t r)});
\draw [shift={(4,3)},line width=0.5pt]  plot[domain=1.892546881191539:2.819842099193151,variable=\t]({1*3.1622776601683795*cos(\t r)+0*3.1622776601683795*sin(\t r)},{0*3.1622776601683795*cos(\t r)+1*3.1622776601683795*sin(\t r)});
\draw [shift={(9.130809898184463,1.8950343697549739)},line width=0.5pt]  plot[domain=2.049591997369353:2.6703128472203277,variable=\t]({1*4.6250506643380485*cos(\t r)+0*4.6250506643380485*sin(\t r)},{0*4.6250506643380485*cos(\t r)+1*4.6250506643380485*sin(\t r)});
\draw [shift={(3,4)},line width=0.5pt]  plot[domain=5.176036589385496:5.81953769817878,variable=\t]({1*4.47213595499958*cos(\t r)+0*4.47213595499958*sin(\t r)},{0*4.47213595499958*cos(\t r)+1*4.47213595499958*sin(\t r)});
\draw [shift={(3,4)},line width=0.5pt]  plot[domain=-0.46364760900080615:0.4636476090008061,variable=\t]({1*4.47213595499958*cos(\t r)+0*4.47213595499958*sin(\t r)},{0*4.47213595499958*cos(\t r)+1*4.47213595499958*sin(\t r)});
\draw [shift={(4.289620338326439,-1.2517388441527422)},line width=0.5pt,dotted]  plot[domain=1.9483631122020275:2.7779986343945624,variable=\t]({1*3.498131805349394*cos(\t r)+0*3.498131805349394*sin(\t r)},{0*3.498131805349394*cos(\t r)+1*3.498131805349394*sin(\t r)});
\draw [shift={(5,-4)},line width=0.5pt,dotted]  plot[domain=1.2490457723982544:1.892546881191539,variable=\t]({1*6.324555320336759*cos(\t r)+0*6.324555320336759*sin(\t r)},{0*6.324555320336759*cos(\t r)+1*6.324555320336759*sin(\t r)});
\draw [shift={(8,4)},line width=0.5pt,dotted]  plot[domain=2.761086276477428:3.522099030702158,variable=\t]({1*5.385164807134505*cos(\t r)+0*5.385164807134505*sin(\t r)},{0*5.385164807134505*cos(\t r)+1*5.385164807134505*sin(\t r)});
\draw [shift={(4.98,0.87)},line width=0.5pt]  plot[domain=1.1956789173869815:1.9391452482038598,variable=\t]({1*5.513374647164838*cos(\t r)+0*5.513374647164838*sin(\t r)},{0*5.513374647164838*cos(\t r)+1*5.513374647164838*sin(\t r)});

\draw [shift={(7.411698614654287,3.020963674114645)},line width=0.5pt]  plot[domain=2.022584869200487:2.816045995279755,variable=\t]({1*3.667453393212879*cos(\t r)+0*3.667453393212879*sin(\t r)},{0*3.667453393212879*cos(\t r)+1*3.667453393212879*sin(\t r)});
\draw [shift={(2.6527432093403887,7.900568653084396)},line width=0.5pt]  plot[domain=5.2005195474737675:5.752146943224357,variable=\t]({1*5.419547115805906*cos(\t r)+0*5.419547115805906*sin(\t r)},{0*5.419547115805906*cos(\t r)+1*5.419547115805906*sin(\t r)});
\draw [shift={(4.46248681417807,3.0209636741146455)},line width=0.5pt]  plot[domain=0.6427294951801987:1.1829187874873526,variable=\t]({1*3.5686804021821508*cos(\t r)+0*3.5686804021821508*sin(\t r)},{0*3.5686804021821508*cos(\t r)+1*3.5686804021821508*sin(\t r)});

\begin{scriptsize}
\draw [fill=black] (5,4) circle (1.5pt);
\draw[color=black] (5.3,4.0) node {$A_1$};
\draw [fill=black] (1,4) circle (1.5pt);
\draw[color=black] (0.75,4.3) node {$A_2$};
\draw [fill=black] (1,0) circle (1.5pt);
\draw[color=black] (0.8,-0.2) node {$A_3$};
\draw [fill=black] (5,0) circle (1.5pt);
\draw[color=black] (5.1,-0.2) node {$A_4$};

\draw [fill=black] (7,6) circle (1.5pt);
\draw[color=black] (7.197210483710562,6.2717994087304785) node {$A'_1$};
\draw [fill=black] (3,6) circle (1.5pt);
\draw[color=black] (2.974475405755977,6.432665506938273) node {$A'_2$};
\draw [fill=black] (3,2) circle (1.5pt);
\draw[color=black] (2.6,2.15) node {$A'_3$};
\draw [fill=black] (7,2) circle (1.5pt);
\draw[color=black] (7.35,2.0) node {$A'_4$};

\draw [fill=black] (3.9417892392320977,4.192254453376471) circle (1.5pt);
\draw[color=black] (3.65,4.55) node {$X_1$};
\draw [fill=black] (5.810578956206633,6.320453063508272) circle (1.5pt);
\draw[color=black] (5.8,6.7) node {$X_2$};
\draw [fill=black] (7.3190824712881515,5.1599683643149215) circle (1.5pt);
\draw[color=black] (7.65,5.3) node {$X_3$};
\draw [fill=black] (5.194378595072868,3.113964375892409) circle (1.5pt);
\draw[color=black] (5.55,2.9) node {$X_4$};

\end{scriptsize}
\end{tikzpicture}
        \caption{\;}
        \label{cube_2edge_4vert}
    \end{figure}
 \end{minipage}
      \hspace{0.05\linewidth}
 \begin{minipage}{0.4\linewidth}
     \begin{figure}[H]
        \begin{tikzpicture}[line cap=round,line join=round,x=0.6cm,y=0.6cm]
\clip(-1,-2.5) rectangle (12,9);

\draw [shift={(4,-5)},line width=0.5pt]  plot[domain=1.3258176636680326:1.8157749899217608,variable=\t]({1*8.246211251235321*cos(\t r)+0*8.246211251235321*sin(\t r)},{0*8.246211251235321*cos(\t r)+1*8.246211251235321*sin(\t r)});
\draw [shift={(4,7)},line width=0.5pt]  plot[domain=4.4674103172578254:4.957367643511554,variable=\t]({1*8.246211251235321*cos(\t r)+0*8.246211251235321*sin(\t r)},{0*8.246211251235321*cos(\t r)+1*8.246211251235321*sin(\t r)});
\draw [shift={(8,7)},line width=0.5pt]  plot[domain=4.4674103172578254:4.957367643511554,variable=\t]({1*8.246211251235321*cos(\t r)+0*8.246211251235321*sin(\t r)},{0*8.246211251235321*cos(\t r)+1*8.246211251235321*sin(\t r)});
\draw [line width=0.5pt] (6,3)-- (6,-1);
\draw [shift={(2,1)},line width=0.5pt]  plot[domain=-0.24497866312686423:0.24497866312686414,variable=\t]({1*8.246211251235321*cos(\t r)+0*8.246211251235321*sin(\t r)},{0*8.246211251235321*cos(\t r)+1*8.246211251235321*sin(\t r)});

\draw [line width=0.5pt] (6,3)-- (10,3);

\draw [shift={(2,5)},line width=0.5pt]  plot[domain=-0.24497866312686423:0.24497866312686414,variable=\t]({1*8.246211251235321*cos(\t r)+0*8.246211251235321*sin(\t r)},{0*8.246211251235321*cos(\t r)+1*8.246211251235321*sin(\t r)});
\draw [shift={(13,5)},line width=0.5pt]  plot[domain=2.8632929945846817:3.4198923125949046,variable=\t]({1*7.280109889280519*cos(\t r)+0*7.280109889280519*sin(\t r)},{0*7.280109889280519*cos(\t r)+1*7.280109889280519*sin(\t r)});
\draw [shift={(8,-1)},line width=0.5pt]  plot[domain=1.3258176636680326:1.8157749899217608,variable=\t]({1*8.246211251235321*cos(\t r)+0*8.246211251235321*sin(\t r)},{0*8.246211251235321*cos(\t r)+1*8.246211251235321*sin(\t r)});
\draw [shift={(10,1)},line width=0.5pt]  plot[domain=2.896613990462929:3.3865713167166573,variable=\t]({1*8.246211251235321*cos(\t r)+0*8.246211251235321*sin(\t r)},{0*8.246211251235321*cos(\t r)+1*8.246211251235321*sin(\t r)});
\draw [shift={(9,-4)},line width=0.5pt]  plot[domain=1.9756881130799802:2.7367008673047097,variable=\t]({1*7.615773105863909*cos(\t r)+0*7.615773105863909*sin(\t r)},{0*7.615773105863909*cos(\t r)+1*7.615773105863909*sin(\t r)});
\draw [shift={(14,-1)},line width=0.5pt]  plot[domain=2.0344439357957027:2.677945044588987,variable=\t]({1*8.94427190999916*cos(\t r)+0*8.94427190999916*sin(\t r)},{0*8.94427190999916*cos(\t r)+1*8.94427190999916*sin(\t r)});
\draw [shift={(2.6191859749404527,6.168606069104442)},line width=0.5pt]  plot[domain=-1.7126492143850864:0.18478361596982498,variable=\t]({1*5.917153097281364*cos(\t r)+0*5.917153097281364*sin(\t r)},{0*5.917153097281364*cos(\t r)+1*5.917153097281364*sin(\t r)});

\begin{scriptsize}
\draw [fill=black] (2,3) circle (1.5pt);
\draw[color=black] (1.7,3.45) node {$A_1$};
\draw [fill=black] (2,-1) circle (1.5pt);
\draw[color=black] (1.5,-1.0) node {$A_4$};

\draw [fill=black] (6,3) circle (1.5pt);
\draw[color=black] (5.5,3.6) node {$A'_1$};
\draw [fill=black] (6,-1) circle (1.5pt);
\draw[color=black] (5.9,-1.5) node {$A'_4$};

\draw [fill=black] (10,3) circle (1.5pt);
\draw[color=black] (10.5,3.1) node {$A'_2$};
\draw [fill=black] (10,-1) circle (1.5pt);
\draw[color=black] (10.3,-1.15) node {$A'_3$};

\draw [fill=black] (6,7) circle (1.5pt);
\draw[color=black] (5.75,7.45) node {$A_1$};
\draw [fill=black] (10,7) circle (1.5pt);
\draw[color=black] (10.2,7.37) node {$A_2$};

\draw [fill=black] (8.44349956568432,7.257262815037879) circle (1.5pt);
\draw[color=black] (8.44,7.7) node {$X_1$};
\draw [fill=black] (7.6164488613114445,3) circle (1.5pt);
\draw[color=black] (8.15,3.35) node {$X_2$};
\draw [fill=black] (6,1.3123938306796115) circle (1.5pt);
\draw[color=black] (6.4,1.2) node {$X_3$};
\draw [fill=black] (1.7826328760294619,0.310886402769353) circle (1.5pt);
\draw[color=black] (1.3,0.5) node {$X_4$};

\draw [fill=black] (8.535691317804156,6.081054344211682) circle (1.5pt);
\draw[color=black] (8.9,5.9) node {$Z_1$};
\draw [fill=black] (2.681586453306882,0.25178200910687565) circle (1.5pt);
\draw[color=black] (2.9,-0.05) node {$Z_2$};

\end{scriptsize}
\end{tikzpicture}
        \caption{\;}
        \label{cube_2edge_4vert_dev}
    \end{figure}
 \end{minipage}

\end{proof}
    
\begin{theorem}\label{geod_cube}
There are only three different simple closed geodesics on  spherical cubes.
\end{theorem}

\begin{proof}
Consider a spherical cube with a planar angle $\alpha \in (\pi/2, 2\pi/3)$.
A simple closed geodesic $\gamma$ divides the cube's surface 
into two closed domains $D_1$ and $D_2$. 
From Lemma~\ref{cube_lemma} follows each domain $D_i$ contains 
four vertices of the cube connecting with at least three edges. 
There are three possibilities for four vertices of a cube
that are consequently connected with edges: 
1) they form a facet of a cube;
 2) three edges share a cube's vertex; and 
 3) three edges form a broken line, that does not form a facet.

\textbf{Type 1.}
Take the midpoints $X_k$ of the edges $A_kA'_k$, $k=1, \dots 4$. 
Join these points within the facets (see Figure~\ref{cube_geod_type_1}). 
The broken line $X_1X_2X_3X_4$ forms a simple closed geodesic $\gamma$
on the cube with the planar angle $\alpha \in (\pi/2, 2\pi/3)$. 
The geodesic $\gamma$ encloses the facet  $A_1A_2A_3A_4$ from one side 
and $A'_1A'_2A'_3A'_4$ from another.

If there is another broken line $\sigma$ on the cube, 
that does not intersect facets $A_1A_2A_3A_4$ and $A'_1A'_2A'_3A'_4$, then 
$\sigma$ is equivalent to $\gamma$. 
By Lemma~\ref{uniqness}, it follows that 
$\gamma$ is a unique geodesic that bounds the domain containing the facet of the cube. 

In general, there are three geodesics of such type on the cube
with the planar angle $\alpha \in (\pi/2, 2\pi/3)$. 
They lie on the planes of symmetry of the cube,
and can be mapped into each other with the cube's symmetries.
  
\begin{figure}[h]
 \centering
\begin{tikzpicture}[line cap=round,line join=round,x=0.7cm,y=0.7cm]

\clip(-2,-2) rectangle (8,7);

\draw [shift={(5,2)},line width=0.3pt]  plot[domain=2.677945044588987:3.6052402625905993,variable=\t]({1*4.47213595499958*cos(\t r)+0*4.47213595499958*sin(\t r)},{0*4.47213595499958*cos(\t r)+1*4.47213595499958*sin(\t r)});
\draw [shift={(-2,2)},line width=0.3pt]  plot[domain=-0.27829965900511144:0.27829965900511133,variable=\t]({1*7.280109889280519*cos(\t r)+0*7.280109889280519*sin(\t r)},{0*7.280109889280519*cos(\t r)+1*7.280109889280519*sin(\t r)});
\draw [shift={(3,-4)},line width=0.3pt]  plot[domain=1.3258176636680326:1.8157749899217608,variable=\t]({1*8.246211251235321*cos(\t r)+0*8.246211251235321*sin(\t r)},{0*8.246211251235321*cos(\t r)+1*8.246211251235321*sin(\t r)});
\draw [shift={(3,8)},line width=0.3pt]  plot[domain=4.4674103172578254:4.957367643511554,variable=\t]({1*8.246211251235321*cos(\t r)+0*8.246211251235321*sin(\t r)},{0*8.246211251235321*cos(\t r)+1*8.246211251235321*sin(\t r)});
\draw [shift={(4,3)},line width=0.3pt]  plot[domain=1.892546881191539:2.819842099193151,variable=\t]({1*3.1622776601683795*cos(\t r)+0*3.1622776601683795*sin(\t r)},{0*3.1622776601683795*cos(\t r)+1*3.1622776601683795*sin(\t r)});
\draw [shift={(9.130809898184463,1.8950343697549739)},line width=0.3pt]  plot[domain=2.049591997369353:2.6703128472203277,variable=\t]({1*4.6250506643380485*cos(\t r)+0*4.6250506643380485*sin(\t r)},{0*4.6250506643380485*cos(\t r)+1*4.6250506643380485*sin(\t r)});
\draw [shift={(3,4)},line width=0.3pt]  plot[domain=5.176036589385496:5.81953769817878,variable=\t]({1*4.47213595499958*cos(\t r)+0*4.47213595499958*sin(\t r)},{0*4.47213595499958*cos(\t r)+1*4.47213595499958*sin(\t r)});
\draw [shift={(3,4)},line width=0.3pt]  plot[domain=-0.46364760900080615:0.4636476090008061,variable=\t]({1*4.47213595499958*cos(\t r)+0*4.47213595499958*sin(\t r)},{0*4.47213595499958*cos(\t r)+1*4.47213595499958*sin(\t r)});
\draw [shift={(4.289620338326439,-1.2517388441527422)},line width=0.3pt,dotted]  plot[domain=1.9483631122020275:2.7779986343945624,variable=\t]({1*3.498131805349394*cos(\t r)+0*3.498131805349394*sin(\t r)},{0*3.498131805349394*cos(\t r)+1*3.498131805349394*sin(\t r)});
\draw [shift={(5,-4)},line width=0.3pt,dotted]  plot[domain=1.2490457723982544:1.892546881191539,variable=\t]({1*6.324555320336759*cos(\t r)+0*6.324555320336759*sin(\t r)},{0*6.324555320336759*cos(\t r)+1*6.324555320336759*sin(\t r)});
\draw [shift={(8,4)},line width=0.3pt,dotted]  plot[domain=2.761086276477428:3.522099030702158,variable=\t]({1*5.385164807134505*cos(\t r)+0*5.385164807134505*sin(\t r)},{0*5.385164807134505*cos(\t r)+1*5.385164807134505*sin(\t r)});
\draw [shift={(4.98,0.87)},line width=0.3pt]  plot[domain=1.1956789173869815:1.9391452482038598,variable=\t]({1*5.513374647164838*cos(\t r)+0*5.513374647164838*sin(\t r)},{0*5.513374647164838*cos(\t r)+1*5.513374647164838*sin(\t r)});

\draw [shift={(4.178138750361743,-2.8803681915244637)},line width=0.3pt]  plot[domain=1.3485794790452956:1.8736638900838034,variable=\t]({1*8.416828737402566*cos(\t r)+0*8.416828737402566*sin(\t r)},{0*8.416828737402566*cos(\t r)+1*8.416828737402566*sin(\t r)});
\draw [shift={(-0.7965292437640612,2.758454105529259)},line width=0.3pt]  plot[domain=-0.2473935355962933:0.3600434203309359,variable=\t]({1*7.307367566660957*cos(\t r)+0*7.307367566660957*sin(\t r)},{0*7.307367566660957*cos(\t r)+1*7.307367566660957*sin(\t r)});
\draw [shift={(7.550670630351407,3.148250972187613)},line width=0.3pt,dotted] 
plot[domain=2.814914342161928:3.4848073193821114,variable=\t]({1*6.207038900071783*cos(\t r)+0*6.207038900071783*sin(\t r)},{0*6.207038900071783*cos(\t r)+1*6.207038900071783*sin(\t r)});
\draw [shift={(4.167643197128606,9.140961529534101)}, line width=0.3pt,,dotted] 
plot[domain=4.419312527683169:4.9663003317821435,variable=\t]({1*8.431681605797332*cos(\t r)+0*8.431681605797332*sin(\t r)},{0*8.431681605797332*cos(\t r)+1*8.431681605797332*sin(\t r)});

\begin{scriptsize}
\draw [fill=black] (5,4) circle (1.5pt);
\draw[color=black] (5.5,4.07) node {$A_1$};
\draw [fill=black] (1,4) circle (1.5pt);
\draw[color=black] (0.5,4.1) node {$A_2$};
\draw [fill=black] (1,0) circle (1.5pt);
\draw[color=black] (0.7,-0.2) node {$A_3$};
\draw [fill=black] (5,0) circle (1.5pt);
\draw[color=black] (5.2,-0.5) node {$A_4$};

\draw [fill=black] (7,6) circle (1.5pt);
\draw[color=black] (7.5,6.2) node {$A'_1$};
\draw [fill=black] (3,6) circle (1.5pt);
\draw[color=black] (2.5,6.3) node {$A'_2$};
\draw [fill=black] (3,2) circle (1.5pt);
\draw[color=black] (3.2,1.7) node {$A'_3$};
\draw [fill=black] (7,2) circle (1.5pt);
\draw[color=black] (7.3,1.8) node {$A'_4$};

\draw [fill=black] (6.0331446874492745,5.3295014157309195) circle (1.5pt);
\draw[color=black] (6.6,5.3) node {$X_1$};
\draw [fill=black] (1.671900339964754,5.140082235089057) circle (1.5pt);
\draw[color=black] (1.1,5.2) node {$X_2$};
\draw [fill=black] (1.73174,1.06881) circle (1.5pt);
\draw[color=black] (1.2,1.15) node {$X_3$};
\draw [fill=black] (6.28835805918443,0.9690428451400361) circle (1.5pt);
\draw[color=black] (6.7,0.8) node {$X_4$};
\end{scriptsize}

\end{tikzpicture}
\includegraphics[width=0.4\textwidth]{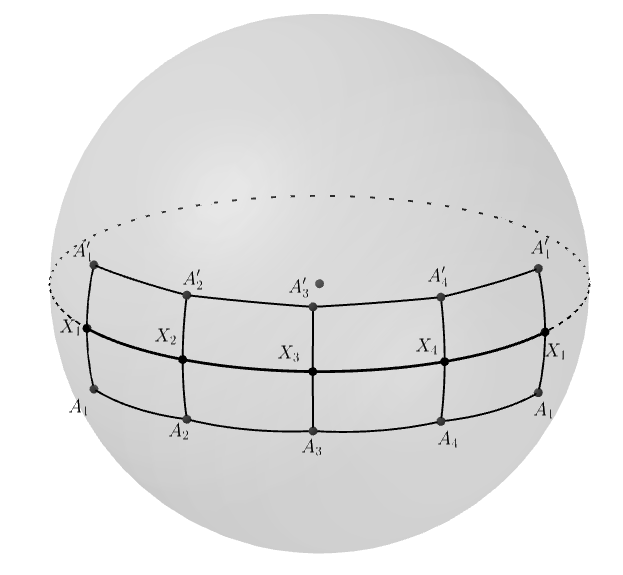}
\caption{ }
\label{cube_geod_type_1}
\end{figure}

\textbf{Type 2.} 
Take the midpoints $X_1, \dots, X_6$ of the edges $A'_1A'_2$,
$A'_2A_2$, $A_2A_3$, $A_3A_4$, $A_4A'_4$, $A'_4A'_1$. 
Join these points within the facets (see Figure~\ref{cube_geod_type_2}). 
Since the triangles $X_2A'_2X_1$ and $X_2A_2X_3$ are equal,
then the segments $X_1X_2$ and $X_2X_3$ form a geodesic segment on the cube.
The same is true for any adjacent segments of the closed broken line $X_1X_2X_3X_4X_5X_6$. 
Thus, this broken line forms a simple closed geodesic $\gamma$
on a cube with a planar angle $\alpha \in (\pi/2, 2\pi/3)$. 

This geodesic $\gamma$ bounds a domain containing three edges 
coming out of the vertex $A_1$ from one side and 
domain with three edges coming out of $A'_3$ from another.
If there is a broken line $\sigma$ that bounds the same domains, 
then $\sigma$ is equivalent to $\gamma$. 
Hence from Lemma~\ref{uniqness} follows $\gamma$ is unique simple closed geodesic 
of such type on the cube.

In general, there are four geodesics that bound  the domains
containing three edges sharing a vertex on a cube 
with the planar angle $\alpha\in (\pi/2, 2\pi/3)$.
They lie in the plane passing through the center of the cube orthogonal to the diagonal
and can be mapped into each other with symmetries of the cube.

\begin{figure}[h]
 \centering
\begin{tikzpicture}[line cap=round,line join=round,x=0.7cm,y=0.7cm]

\clip(-2,-2) rectangle (8.5,7);

\draw [shift={(5,2)},line width=0.3pt]  plot[domain=2.677945044588987:3.6052402625905993,variable=\t]({1*4.47213595499958*cos(\t r)+0*4.47213595499958*sin(\t r)},{0*4.47213595499958*cos(\t r)+1*4.47213595499958*sin(\t r)});
\draw [shift={(-2,2)},line width=0.3pt]  plot[domain=-0.27829965900511144:0.27829965900511133,variable=\t]({1*7.280109889280519*cos(\t r)+0*7.280109889280519*sin(\t r)},{0*7.280109889280519*cos(\t r)+1*7.280109889280519*sin(\t r)});
\draw [shift={(3,-4)},line width=0.3pt]  plot[domain=1.3258176636680326:1.8157749899217608,variable=\t]({1*8.246211251235321*cos(\t r)+0*8.246211251235321*sin(\t r)},{0*8.246211251235321*cos(\t r)+1*8.246211251235321*sin(\t r)});
\draw [shift={(3,8)},line width=0.3pt]  plot[domain=4.4674103172578254:4.957367643511554,variable=\t]({1*8.246211251235321*cos(\t r)+0*8.246211251235321*sin(\t r)},{0*8.246211251235321*cos(\t r)+1*8.246211251235321*sin(\t r)});
\draw [shift={(4,3)},line width=0.3pt]  plot[domain=1.892546881191539:2.819842099193151,variable=\t]({1*3.1622776601683795*cos(\t r)+0*3.1622776601683795*sin(\t r)},{0*3.1622776601683795*cos(\t r)+1*3.1622776601683795*sin(\t r)});
\draw [shift={(9.130809898184463,1.8950343697549739)},line width=0.3pt]  plot[domain=2.049591997369353:2.6703128472203277,variable=\t]({1*4.6250506643380485*cos(\t r)+0*4.6250506643380485*sin(\t r)},{0*4.6250506643380485*cos(\t r)+1*4.6250506643380485*sin(\t r)});
\draw [shift={(3,4)},line width=0.3pt]  plot[domain=5.176036589385496:5.81953769817878,variable=\t]({1*4.47213595499958*cos(\t r)+0*4.47213595499958*sin(\t r)},{0*4.47213595499958*cos(\t r)+1*4.47213595499958*sin(\t r)});
\draw [shift={(3,4)},line width=0.3pt]  plot[domain=-0.46364760900080615:0.4636476090008061,variable=\t]({1*4.47213595499958*cos(\t r)+0*4.47213595499958*sin(\t r)},{0*4.47213595499958*cos(\t r)+1*4.47213595499958*sin(\t r)});
\draw [shift={(4.289620338326439,-1.2517388441527422)},line width=0.3pt,dotted]  plot[domain=1.9483631122020275:2.7779986343945624,variable=\t]({1*3.498131805349394*cos(\t r)+0*3.498131805349394*sin(\t r)},{0*3.498131805349394*cos(\t r)+1*3.498131805349394*sin(\t r)});
\draw [shift={(5,-4)},line width=0.3pt,dotted]  plot[domain=1.2490457723982544:1.892546881191539,variable=\t]({1*6.324555320336759*cos(\t r)+0*6.324555320336759*sin(\t r)},{0*6.324555320336759*cos(\t r)+1*6.324555320336759*sin(\t r)});
\draw [shift={(8,4)},line width=0.3pt,dotted]  plot[domain=2.761086276477428:3.522099030702158,variable=\t]({1*5.385164807134505*cos(\t r)+0*5.385164807134505*sin(\t r)},{0*5.385164807134505*cos(\t r)+1*5.385164807134505*sin(\t r)});
\draw [shift={(4.98,0.87)},line width=0.3pt]  plot[domain=1.1956789173869815:1.9391452482038598,variable=\t]({1*5.513374647164838*cos(\t r)+0*5.513374647164838*sin(\t r)},{0*5.513374647164838*cos(\t r)+1*5.513374647164838*sin(\t r)});

\draw [shift={(7,-4)},line width=0.3pt]  plot[domain=1.7582829347064102:2.0985756302534626,variable=\t]({1*10.568348178730501*cos(\t r)+0*10.568348178730501*sin(\t r)},{0*10.568348178730501*cos(\t r)+1*10.568348178730501*sin(\t r)});
\draw [shift={(2.2893076789530022,0.8609518983926844)},line width=0.3pt,dotted]  plot[domain=0.5719374225131915:1.1100853033987088,variable=\t]({1*6.1587383534027715*cos(\t r)+0*6.1587383534027715*sin(\t r)},{0*6.1587383534027715*cos(\t r)+1*6.1587383534027715*sin(\t r)});
\draw [shift={(-3.8248660995981982,6.515125141880457)},line width=0.3pt]  plot[domain=5.78157208045328:6.080505527190989,variable=\t]({1*11.534137667290178*cos(\t r)+0*11.534137667290178*sin(\t r)},{0*11.534137667290178*cos(\t r)+1*11.534137667290178*sin(\t r)});
\draw [shift={(1.0051395185676406,9.658462131480439)},line width=0.3pt,dotted]  plot[domain=4.927616807817248:5.2586744150013605,variable=\t]({1*10.136798349605346*cos(\t r)+0*10.136798349605346*sin(\t r)},{0*10.136798349605346*cos(\t r)+1*10.136798349605346*sin(\t r)});
\draw [shift={(7.4834803873773765,7.262626011358502)},line width=0.3pt]  plot[domain=3.7796052198636967:4.190869928707741,variable=\t]({1*8.657435377482424*cos(\t r)+0*8.657435377482424*sin(\t r)},{0*8.657435377482424*cos(\t r)+1*8.657435377482424*sin(\t r)});
\draw [shift={(12.102652426972485,-0.4232162619926742)},line width=0.3pt,dotted]  plot[domain=2.651618106853904:2.9264193657061193,variable=\t]({1*11.821627606562812*cos(\t r)+0*11.821627606562812*sin(\t r)},{0*11.821627606562812*cos(\t r)+1*11.821627606562812*sin(\t r)});

\begin{scriptsize}
\draw [fill=black] (5,4) circle (1.5pt);
\draw[color=black] (5.5,4.07) node {$A_1$};
\draw [fill=black] (1,4) circle (1.5pt);
\draw[color=black] (0.3,4.1) node {$A_2$};
\draw [fill=black] (1,0) circle (1.5pt);
\draw[color=black] (0.7,-0.3) node {$A_3$};
\draw [fill=black] (5,0) circle (1.5pt);
\draw[color=black] (5.2,-0.5) node {$A_4$};

\draw [fill=black] (7,6) circle (1.5pt);
\draw[color=black] (7.5,6.2) node {$A'_1$};
\draw [fill=black] (3,6) circle (1.5pt);
\draw[color=black] (2.5,6.3) node {$A'_2$};
\draw [fill=black] (3,2) circle (1.5pt);
\draw[color=black] (3.2,1.6) node {$A'_3$};
\draw [fill=black] (7,2) circle (1.5pt);
\draw[color=black] (7.4,1.8) node {$A'_4$};

\draw [fill=black] (5.03016410999222,6.383146430312612) circle (1.5pt);
\draw[color=black] (5.24,6.8) node {$X_1$};
\draw [fill=black] (1.671900339964754,5.140082235089057) circle (1.5pt);
\draw[color=black] (1.2,5.5) node {$X_2$};
\draw [fill=black] (0.5291265259383815,2.1062561947944514) circle (1.5pt);
\draw[color=black] (-0.1,2.3) node {$X_3$};
\draw [fill=black] (3.1700555906023364,-0.24445759866013894) circle (1.5pt);
\draw[color=black] (3.3,-0.7) node {$X_4$};
\draw [fill=black] (6.28835805918443,0.9690428451400352) circle (1.5pt);
\draw[color=black] (6.8,0.9) node {$X_5$};
\draw [fill=black] (7.467906845290419,4.194443878296596) circle (1.5pt);
\draw[color=black] (8,4.2) node {$X_6$};

\end{scriptsize}
\end{tikzpicture}
\includegraphics[width=0.4\textwidth]{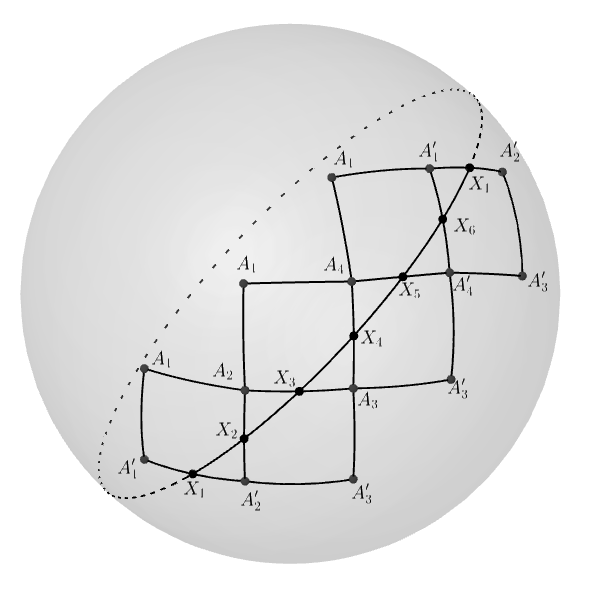}
\caption{ }
\label{cube_geod_type_2}
\end{figure}

\textbf{Type 3.} 
Take the midpoints $X_1$ and $X_2$ of the edges $A_2A_3$ and $A'_1A'_4$. 
First, develop the edges $A_3A_2A'_2A'_3$, $A_2A'_2A'_1A_1$ and $A_1A'_1A'_4A_4$ 
on a unite sphere and connect the points $X_1$, $X_2$ inside the development. 
Then develop the edges $A'_1A'_2A'_3A'_4$, $A'_3A'_4A_3A_4$, $A_3A_4A_1A_2$
and connect the points $X_1$ and $X_2$ inside the development
(see Figure~\ref{cube_geod_type_3}). 
Since the developments are equal spherical polygons, 
then $\angle X_1X_2A'_1=\angle X'_1X_2A'_4$ and
$\angle X_2X_2A_2= \angle X_2X_1A_3$.
Hence, the closed broken line $X_1X_2$ forms a simple closed geodesic $\gamma$ 
on the spherical cube with the planar angle $\alpha\in (\pi/2, 2\pi/3)$.

By construction, $\gamma$  encloses a domain $D_1$ containing 
the edges $A_2A_1$, $A_1A_4$ and $A_4A'_4$ from one side and
a domain $D_2$ with the edges $A'_1A'_2$, $A'_2A'_3$ and $A'_3A_3$ from another. 
These edges form a broken line that does not bound a facet. 
Any broken line $\sigma$ that bounds the domains $D_1$ and $D_2$
is equivalent to $\gamma$. 
Indeed, assume $\sigma$ starts at the edge $A_2A_3$.
Then on the facet $A_2A_3A'_3A'_2$, it can no go to the edges $A'_2A'_3$ and $A'_3A_3$.
Hence, $\sigma$ goes to $A_2A'_2$. 
Then, since $\sigma$ can no go to $A'_1A'_2$ and $A_1A_2$, it goes to $A_1A'_1$.
After, $\sigma$ can go only to $A'_1A'_4$. 
Continuing in the same way, we obtain that $\sigma$ intersects 
the same edges as $\gamma$ in the same order. 
Therefore, by Lemma~\ref{uniqness} $\gamma$ is a unique geodesic of such type.

In general, there are $12$ simple closed geodesics 
that enclose domains with three edges that form a broken line
on a cube with the planar angle $\alpha\in (\pi/2, 2\pi/3)$.
These geodesics can be mapped into each other with symmetries  of the cube. 

\begin{figure}[h]
 \centering
\begin{tikzpicture}[line cap=round,line join=round,x=0.7cm,y=0.7cm]
\clip(-2,-2) rectangle (8.5,7);

\draw [shift={(5,2)},line width=0.3pt]  plot[domain=2.677945044588987:3.6052402625905993,variable=\t]({1*4.47213595499958*cos(\t r)+0*4.47213595499958*sin(\t r)},{0*4.47213595499958*cos(\t r)+1*4.47213595499958*sin(\t r)});
\draw [shift={(-2,2)},line width=0.3pt]  plot[domain=-0.27829965900511144:0.27829965900511133,variable=\t]({1*7.280109889280519*cos(\t r)+0*7.280109889280519*sin(\t r)},{0*7.280109889280519*cos(\t r)+1*7.280109889280519*sin(\t r)});
\draw [shift={(3,-4)},line width=0.3pt]  plot[domain=1.3258176636680326:1.8157749899217608,variable=\t]({1*8.246211251235321*cos(\t r)+0*8.246211251235321*sin(\t r)},{0*8.246211251235321*cos(\t r)+1*8.246211251235321*sin(\t r)});
\draw [shift={(3,8)},line width=0.3pt]  plot[domain=4.4674103172578254:4.957367643511554,variable=\t]({1*8.246211251235321*cos(\t r)+0*8.246211251235321*sin(\t r)},{0*8.246211251235321*cos(\t r)+1*8.246211251235321*sin(\t r)});
\draw [shift={(4,3)},line width=0.3pt]  plot[domain=1.892546881191539:2.819842099193151,variable=\t]({1*3.1622776601683795*cos(\t r)+0*3.1622776601683795*sin(\t r)},{0*3.1622776601683795*cos(\t r)+1*3.1622776601683795*sin(\t r)});
\draw [shift={(9.130809898184463,1.8950343697549739)},line width=0.3pt]  plot[domain=2.049591997369353:2.6703128472203277,variable=\t]({1*4.6250506643380485*cos(\t r)+0*4.6250506643380485*sin(\t r)},{0*4.6250506643380485*cos(\t r)+1*4.6250506643380485*sin(\t r)});
\draw [shift={(3,4)},line width=0.3pt]  plot[domain=5.176036589385496:5.81953769817878,variable=\t]({1*4.47213595499958*cos(\t r)+0*4.47213595499958*sin(\t r)},{0*4.47213595499958*cos(\t r)+1*4.47213595499958*sin(\t r)});
\draw [shift={(3,4)},line width=0.3pt]  plot[domain=-0.46364760900080615:0.4636476090008061,variable=\t]({1*4.47213595499958*cos(\t r)+0*4.47213595499958*sin(\t r)},{0*4.47213595499958*cos(\t r)+1*4.47213595499958*sin(\t r)});
\draw [shift={(4.289620338326439,-1.2517388441527422)},line width=0.3pt,dotted]  plot[domain=1.9483631122020275:2.7779986343945624,variable=\t]({1*3.498131805349394*cos(\t r)+0*3.498131805349394*sin(\t r)},{0*3.498131805349394*cos(\t r)+1*3.498131805349394*sin(\t r)});
\draw [shift={(5,-4)},line width=0.3pt,dotted]  plot[domain=1.2490457723982544:1.892546881191539,variable=\t]({1*6.324555320336759*cos(\t r)+0*6.324555320336759*sin(\t r)},{0*6.324555320336759*cos(\t r)+1*6.324555320336759*sin(\t r)});
\draw [shift={(8,4)},line width=0.3pt,dotted]  plot[domain=2.761086276477428:3.522099030702158,variable=\t]({1*5.385164807134505*cos(\t r)+0*5.385164807134505*sin(\t r)},{0*5.385164807134505*cos(\t r)+1*5.385164807134505*sin(\t r)});
\draw [shift={(4.98,0.87)},line width=0.3pt]  plot[domain=1.1956789173869815:1.9391452482038598,variable=\t]({1*5.513374647164838*cos(\t r)+0*5.513374647164838*sin(\t r)},{0*5.513374647164838*cos(\t r)+1*5.513374647164838*sin(\t r)});

\draw [shift={(0.12347182636276519,6.285124874348751)},line width=0.3pt,dotted]  plot[domain=5.0796670840213:5.667743375101344,variable=\t]({1*6.989302058082156*cos(\t r)+0*6.989302058082156*sin(\t r)},{0*6.989302058082156*cos(\t r)+1*6.989302058082156*sin(\t r)});
\draw [shift={(2.634308080250562,7.1667925665536245)},line width=0.3pt,dotted]  plot[domain=5.286131060544825:5.698637844858971,variable=\t]({1*5.82532687505157*cos(\t r)+0*5.82532687505157*sin(\t r)},{0*5.82532687505157*cos(\t r)+1*5.82532687505157*sin(\t r)});
\draw [shift={(5.298477845826164,4.483456112017055)},line width=0.3pt]  plot[domain=3.613176603898064:4.1984777344146575,variable=\t]({1*5.40028051578995*cos(\t r)+0*5.40028051578995*sin(\t r)},{0*5.40028051578995*cos(\t r)+1*5.40028051578995*sin(\t r)});

\draw [shift={(14.345155035406624,-0.9790502418609638)},line width=0.3pt,dotted]  plot[domain=2.7106234164463183:2.9277616702631177,variable=\t]({1*14.143809376954263*cos(\t r)+0*14.143809376954263*sin(\t r)},{0*14.143809376954263*cos(\t r)+1*14.143809376954263*sin(\t r)});
\draw [shift={(4.934310755567628,-2.167384957441445)},line width=0.3pt]  plot[domain=1.3922559798681071:2.0221051561364702,variable=\t]({1*7.869510474943066*cos(\t r)+0*7.869510474943066*sin(\t r)},{0*7.869510474943066*cos(\t r)+1*7.869510474943066*sin(\t r)});
\draw [shift={(3.707642662065193,2.4326203931926758)},line width=0.3pt]  plot[domain=0.38676188684188995:0.8753269067421325,variable=\t]({1*4.064591998424348*cos(\t r)+0*4.064591998424348*sin(\t r)},{0*4.064591998424348*cos(\t r)+1*4.064591998424348*sin(\t r)});

\begin{scriptsize}
\draw [fill=black] (5,4) circle (1.5pt);
\draw[color=black] (5.5,4.07) node {$A_1$};
\draw [fill=black] (1,4) circle (1.5pt);
\draw[color=black] (0.4,4.1) node {$A_2$};
\draw [fill=black] (1,0) circle (1.5pt);
\draw[color=black] (0.7,-0.3) node {$A_3$};
\draw [fill=black] (5,0) circle (1.5pt);
\draw[color=black] (5.2,-0.5) node {$A_4$};

\draw [fill=black] (7,6) circle (1.5pt);
\draw[color=black] (7.5,6.2) node {$A'_1$};
\draw [fill=black] (3,6) circle (1.5pt);
\draw[color=black] (2.5,6.3) node {$A'_2$};
\draw [fill=black] (3,2) circle (1.5pt);
\draw[color=black] (3.2,1.6) node {$A'_3$};
\draw [fill=black] (7,2) circle (1.5pt);
\draw[color=black] (7.4,1.8) node {$A'_4$};

\draw [fill=black] (1.4946375386569768,4.929548894766139) circle (1.5pt);
\draw [fill=black] (6.331883190510121,5.577031508912076) circle (1.5pt);
\draw [fill=black] (2.633165227860871,-0.238047842174117) circle (1.5pt);
\draw [fill=black] (5.7961715686278374,2.2742418532687037) circle (1.5pt);

\draw [fill=black] (0.48763891662130066,2.03011992501219) circle (1.5pt);
\draw[color=black] (0.1,2) node {$X_1$};

\draw [fill=black] (7.472004799259215,3.9657498116421706) circle (1.5pt);
\draw[color=black] (7.9,4) node {$X_2$};
\end{scriptsize}
\end{tikzpicture}
\includegraphics[width=0.4\textwidth]{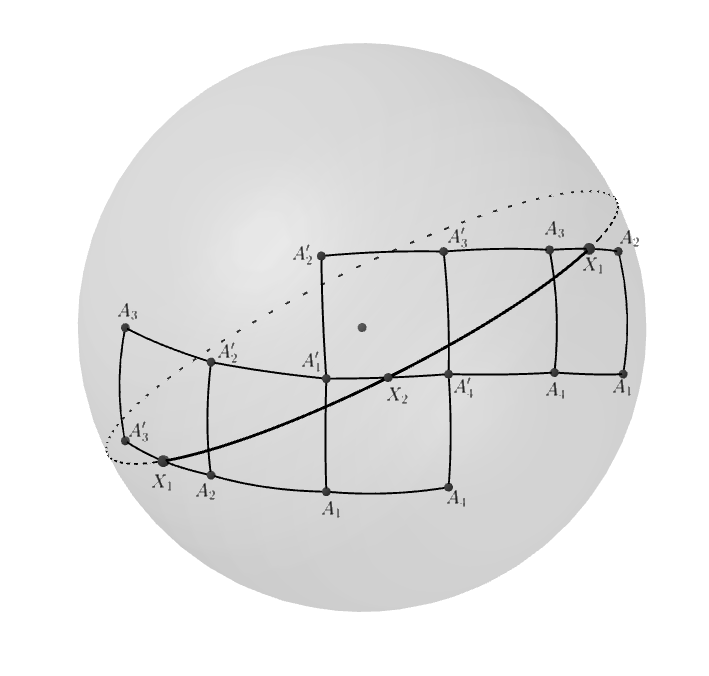}
\caption{ }
\label{cube_geod_type_3}
\end{figure}

\end{proof}

The author believes that the same method can be used to find geodesics on spherical icosahedra and dodecahedra.

\bibliographystyle{alpha}
\bibliography{bibliography}

\begin{thebibliography}{Mon87}

\bibitem[Ale05]{Alek50}
A.~D. Alexandrov.
\newblock {\em Convex Polyhedra}.
\newblock Springer, 2005.

\bibitem[AP18]{AkPet2018}
A.~Akopyan and A.~Petrunin.
\newblock Long geodesics on convex surfaces.
\newblock {\em Math Intelligencer}, 40:26--31, 2018.

\bibitem[Bor20]{Bor2020}
A.~A. Borisenko.
\newblock An estimation of the length of a convex curve in two-dimensional aleksandrov spaces.
\newblock {\em J. Math. Phys. Anal. Geom.}, 16:221--227, 2020.

\bibitem[Bor22]{Bor2022}
A.~A. Borisenko.
\newblock A necessary and sufficient condition for the existence of simple closed geodesics on regular tetrahedra in spherical space.
\newblock {\em Sb. Math.}, 213(2):161--172, 2022.

\bibitem[BS21]{BorSuh2021}
A.~A. Borisenko and D.~D. Sukhorebska.
\newblock Simple closed geodesics on regular tetrahedra in spherical space.
\newblock {\em Sb. Math.}, 212(8):1040--1067, 2021.

\bibitem[FF07]{FucFuc07}
D.~B. Fuchs and E.~Fuchs.
\newblock Closed geodesics on regular polyhedra.
\newblock {\em Mosc. Math. J.}, 7(2):265--279, 2007.

\bibitem[HW75]{Hard}
G.~H. Hardy and E.~M. Wright.
\newblock {\em An Introduction to the Theory of Numbers}.
\newblock Oxford University Press, 1975.

\bibitem[Mon87]{Mon}
H.L. Montgomery.
\newblock Fluctuations in the mean of euler’s phi function.
\newblock {\em Proc. Indian Acad. Sci. Math. Sci}, 97:239--245, 1987.

\bibitem[Pog49]{Pog49}
A.~V. Pogorelov.
\newblock Quasi-geodesic lines on a convex surface.
\newblock {\em Mat. Sb. (N.S.)}, 25(67):275--306, 1949.

\bibitem[Pro07]{Protasov07}
V.~Yu. Protasov.
\newblock Closed geodesics on the surface of a simplex.
\newblock {\em Sb. Math.}, 198(2):243--260, 2007.

\bibitem[Top63]{Top63}
V.~A. Toponogov.
\newblock A bound for the length of a convex curve on a two-dimensional surface.
\newblock {\em Sibirsk. Mat. Zh.}, 4(5):1189--1193, 1963.

\end{thebibliography}

\end{document}